\documentclass[11pt]{amsart}
\usepackage{hyperref}
\usepackage{dsfont}
\usepackage[all]{xy}
\usepackage{graphics}
\usepackage{epsfig}
\usepackage{amsmath}
\usepackage{amscd}
\usepackage{pdfsync}
\usepackage{comment}
\usepackage{todonotes}
\usepackage{mdwlist}
\allowdisplaybreaks
\date\today

\usepackage{amsfonts,latexsym,amssymb}

\newcommand{\bbD}{{{\mathbb{D}}}}
\newcommand{\bbZ}{{{\mathbb{Z}}}}

\newcommand{\bbC}{{{\mathbb{C}}}}
\newcommand{\bbP}{{{\mathbb{P}}}}

\newcommand{\bbW}{{{\mathbb{W}}}}
\newcommand{\bbX}{{{\mathbb{X}}}}

\newcommand{\cD}{{\mathcal{D}}}

\newcommand{\cH}{{\mathcal{H}}}

\newcommand{\cK}{{\mathcal{K}}}
\newcommand{\cL}{{\mathcal{L}}}

\newcommand{\cO}{{\mathcal{O}}}

\newcommand{\cV}{{\mathcal{V}}}
\newcommand{\cW}{{\mathcal{W}}}
\newcommand{\cX}{{\mathcal{X}}}
\newcommand{\cY}{{\mathcal{Y}}}
\newcommand{\cZ}{{\mathcal{Z}}}

\newcommand{\ba}{\mathbf a}  
  
\newcommand{\bg}{\mathbf g}

\newcommand{\bz}{\mathbf z}

\newcommand{\bF}{\mathbf F}
\newcommand{\bK}{\mathbf K} 
\newcommand{\bL}{\mathbf L}

\newcommand{\bW}{\mathbf W}
\newcommand{\bw}{\mathbf w}

\newcommand{\ha}{{\hat{a}}}
\newcommand{\hi}{{\hat{i}}}
\newcommand{\hb}{{\hat{b}}}

\newcommand{\hj}{{\hat{j}}}
\newcommand{\hz}{{\hat{z}}}
\newcommand{\he}{{\hat{e}}}
\newcommand{\hA}{{\hat{A}}}
\newcommand{\hB}{{\hat{B}}}
\newcommand{\hC}{{\hat{C}}}

\newcommand{\hS}{{\hat{S}}}
\newcommand{\hT}{{\hat{T}}}

\newcommand{\tj}{{\tilde{j}}}

\newcommand{\tu}{{\tilde{u}}}
\newcommand{\tv}{{\tilde{v}}}

\newcommand{\tP}{{\tilde{P}}}

\newcommand{\ttb}{{\mathtt{b}}}

\newcommand{\ttv}{{\mathtt{v}}}
\newcommand{\ttw}{{\mathtt{w}}}

\newcommand{\fg}{\mathfrak{g}}
\newcommand{\hfg}{\hat{\mathfrak{g}}}

\newcommand{\hepsilon}{{\hat{\epsilon}}}
\newcommand{\hiota}{{\hat{\iota}}}
\newcommand{\hpsi}{{\hat{\psi}}}

\newcommand{\hmu}{{\hat{\mu}}}
\newcommand{\halpha}{{\hat{\alpha}}}

\newcommand{\tepsilon}{{\tilde{\epsilon}}}
\newcommand{\tiota}{{\tilde{\iota}}}

\newcommand{\bchi}{{\boldsymbol{\chi}}}
\newcommand{\balpha}{\boldsymbol{\alpha}}
\newcommand{\bbeta}{\boldsymbol{\beta}}
\newcommand{\bpsi}{{\boldsymbol{\psi}}}
\newcommand{\bomega}{{\boldsymbol{\omega}}}

\newcommand{\bcL}{{\boldsymbol{\mathcal L}}}
\newcommand{\bdots}{{\boldsymbol{:}}}

\newcommand{\tttv}{{\tilde{\mathtt{v}}}}

\newcommand{\tbbW}{{\tilde{\mathbb{W}}}}

\newcommand{\sfv}{{\mathsf{v}}}
\newcommand{\sfb}{{\mathsf{b}}}

\newcommand{\sfO}{{\mathsf{O}}}
\newcommand{\sfT}{{\mathsf{T}}}
\newcommand{\sfR}{{\mathsf{R}}}

\newcommand{\sfE}{{\mathsf{E}}}
\newcommand{\sfEz}{{\mathsf{E}^{\mathsf{reg}}}}
\newcommand{\sfOz}{{\mathsf{O}^{\mathsf{reg}}}}
\newcommand{\sfKs}{{\mathsf{K}_s}}
\newcommand{\sfKsz}{{\mathsf{K}_s^{\mathsf{reg}}}}

\newcommand{\hbbW}{{\hat{\mathbb{W}}}}

\newcommand{\KK}{{{\mathcal K}^{*}(\kappa)}}
\newcommand{\KKl}{{{\mathcal K}^{*}_l(\kappa)}}
\newcommand{\KKz}{{{\mathcal K}^{*}(\kappa)^z}}
\newcommand{\KKlz}{{{\mathcal K}^{*}_l(\kappa)^z}}

\newcommand{\cLK}{{{\mathcal L}(K)}}

\newcommand{\cu}{\check{u}}
\newcommand{\clambda}{{\check{\lambda}}}

\newcommand{\lO}{{}_O}
\newcommand{\lD}{{}_D}
\newcommand{\lP}{{}_P}
\newcommand{\lOp}{{}_{O'}}
\newcommand{\lDp}{{}_{D'}}
\newcommand{\lbz}{{}_{{|\bz|}}}
\newcommand{\lzj}{{}_{z_j}}

\newcommand{\sSz}{\sigma_{{}_{Sz}}}

\newcommand{\hotimes}{{\hat{\otimes}}}

\newtheorem{proposition}{Proposition}[section]
\newtheorem{theorem}[proposition]{Theorem}
\newtheorem{lemma}[proposition]{Lemma}

\newtheorem{corollary}[proposition]{Corollary}

\theoremstyle{definition}
\newtheorem{definition}[proposition]{Definition}
\newtheorem{remark}[proposition]{Remark}
\newtheorem{notation}[proposition]{Notation}

\begin{document}

\title[Rational and global forms of Chiral CFTs I]{Rational and global forms of certain Chiral Conformal Field Theories; Vertex Algebras I}


\author[T. R. Ramadas]{T. R. Ramadas}
\address{Abdus Salam ICTP\\
       Trieste Italy 34014\\
       Italy\\
\& (After 1 January 2014) Chennai Mathematical Institute\\
H1 Sipcot IT Park\\  
Siruseri, TN 603103\\
India} 
\email{ramadas@ictp.it,ramadas@cmi.ac.in}

\date{\today}

\begin{abstract} Quantum fields are traditionally formalised as (unbounded) operator-valued distributions; two-dimensional chiral quantum fields as formal power series with operator coefficients.  Chiral conformal field theories (CFTs) are essentially algebraic and geometric in nature. The thesis of this work is that a (Euclidean) chiral field is a morphism of (not always quasi-coherent)  sheaves, and the morphism can in many cases be described explicitly. In particular it is meaningful, and useful, to consider the value of a chiral field at a point. (Once this is noted, the sheaf-theory is best put aside.)

The CFTs of the title are the boson, neutral fermion, the $bc$-system, and current algebras (``chiral WZW models'').   By rational forms we mean constructions on the complex projective line that replace power series, formal or otherwise, by rational functions. By global forms we understand extensions to smooth projective curves of arbitrary genus. 

We define operator products and the notion of a vertex algebra structure in the specific case of the boson; the latter definition is particularly transparent.We define actions of the Heisenberg and Virasoro algebras; we also construct certain natural pairings. We then construct a hermitian structure and make explicit its reflection positivity.

We briefly consider the complex-analytic version: the chiral boson on a disc.

In the case of current algebras we exhibit actions of affine lie algebras.  We also construct pairings in this case. 

We give an account of a field theory associated to a one-dimensional lattice, and finish by constructing the neutral fermion, the $b-c$ system, and the boson in arbitrary genus. 

In a first sequel to this work we will study current algebras at length, exhibiting the Virasoro action and vertex algebra structure. Reflection positivity holds in the integrable case, i.e., when the central charge is integral.  We will describe constructions  of conformal blocks and the KZ connection that avoid use of local coordinates and glueing.  We will construct the ADW/Hitchin/KZB/TUY connection; we hope to give a transparent proof of its unitarity.

This is a preliminary version.
\end{abstract}

\thanks{It is a pleasure to thank E. Looijenga and M.S. Narasimhan for several conversations and suggestions regarding this work. Earlier versions adopted a function-theoretic approach, and it was a comment by E. Looijenga that led me to working with rational forms and functions rather than Bergman/Hardy spaces on the disc.}

\maketitle

\pagebreak

\tableofcontents

\pagebreak

\part{\textbf{{\Large Introduction}}}

Let $\bbP$ be a complex projective line.  Much of this work is concerned with constructions, in terms of rational functions, of certain ``Euclidean chiral quantum fields" on $\bbP$. We give no formal definition of an Euclidean chiral quantum field  but each example will be a natural morphism of (not always quasi-coherent) sheaves. Our constructions are relatively elementary; in particular, algebraic geometry is used mostly to provide a natural language and to make precise notions of continuity and limits.

In a later section we extend constructions to smooth projective curves of arbitrary genus.

The subtlest definition of this work is that of a current, outlined in \S \ref{subsection: Current algebra: a summary} below. We start, however, with a detailed description of the boson. Although simple,  the treatment without recourse to power series is new and crucial preparation for what follows.

We will also make contact with the conventional description of the field in terms of formal power series.

 \section{\textbf{Detailed summaries}}

We summarise the themes of this work.

\subsection{The boson}\label{subsection: The boson: a summary}

A meromorphic form on a smooth curve is said to be {\it of the second kind} if all its residues vanish; or equivalently it is locally (in the analytic topology) the exterior 
derivative of a meromorphic function. On $\bbP$ any such form is globally the  derivative of a meromorphic function, unique up to a constant.

Denote by $\bK$ the space of meromorphic one-forms of the second kind on $\bbP$. Let $\bW$ denote the symmetric algebra over $\bK$. Given $z \in \bbP$, let  $\bK^z \subset \bK$ be the space of forms regular at $z$, and $\bW^z$ the corresponding symmetric algebra. Let $K$ denote the canonical line bundle of $\bbP$. 

\begin{enumerate}

\item Define fields $e$ and $i$ as follows: (a) $e(z): \bW \to \bW \otimes K_z$ is multiplication by
\begin{equation*}
e_z= -\frac{du}{(u-u(z))^2} \otimes du_z
\end{equation*}
for any global coordinate $u$ (i.e., meromorphic function with one pole) regular at $z$, and (b) $i(z): \bW^z \to \bW^z \otimes K_z$ is -- up to a minus sign --  the derivation that extends the evaluation map at $z$. The field $b$ is defined to be the sum:
\begin{equation*}
b(z) \equiv i(z)+e(z): \bW^z \to \bW \otimes K_z
\end{equation*}
We will argue that $i, e$ and $b$ can be thought of as functions from $\bbP$ to operators, covariant with respect to automorphisms of $\bbP$. More precisely, 
\begin{itemize}
\item $\bW^z$ and $\bW$ are fibres (\emph{not} stalks) at $z$ of sheaves of $\cO_{\bbP}$-modules equivariant with respect to $Aut(\bbP)$,
\item $i,e$ and $b$ are equivariant morphisms, and
\item $i(z),e(z)$ and $b(z)$ the corresponding maps of fibres.
\end{itemize}

\item If $z_1 \ne z_2$, then $b(z_1)$ and $b(z_2)$ can be composed on $\bW^{z_1} \cap \bW^{z_2}$ and \emph{commute} (modulo a \emph{signless} exchange $K_{z_1} \otimes K_{z_2} \leftrightarrow K_{z_2} \otimes K_{z_1}$). Together with (\ref{OPEfirstexample}) below, this shows that $b$ is ``\emph{local with respect to itself'}', to use the terminology of vertex algebras.

\item One can immediately define and compute ``$n$-point functions'' of $b$, denoted $<b(z_1)\dots b(z_n)>$,  which are meromorphic polydifferentials on $\overbrace{\bbP \times \dots \times \bbP}^{n\  factors}$ with double poles along the partial diagonals (\S \ref{Wick}). These are defined by the equation:
\begin{equation*}
\begin{split}
b(z_1) \circ \dots b(z_n) \circ \mathbf{1} &= <b(z_1)\dots b(z_n)> \mathbf{1}\\ &+ \text{\small higher order terms in the symmetric algebra}
\end{split}
\end{equation*}
(We denote by $\mathbf{1}$ the element $1 \in \bbC \subset \bW$.) The $n$-point functions are given by ``Wick's theorem''. 

\suspend{enumerate}

\noindent \textit{The Quantum Field Theory context.} The $n$-point functions are Schwinger functions of the chiral boson, i.e., Wightman functions of the ``left-moving'' (or ``right-moving'' according to convention) fields analytically continued from $1+1$ dimensional Minkowski space to two-dimensional Euclidean space.  In the case of field theories which are described by path integrals, Schwinger functions are vacuum expectation values (VEVs) of  products of multiplication operators on a measure space.  This is an elegant explanation of their symmetry under permutation of arguments. Chiral field theories tend not to have path-integral (=Lagrangian) descriptions. Nonetheless the Schwinger functions have (anti)symmetry properties - this is expressed by the fact that they are VEVs of products of (anti)commuting operators. The Minkowski theory, with its package of Hilbert space and Poincar\'e group action can be reconstructed from the Euclidean theory once certain conditions are met - crucial among them reflection positivity.
\resume{enumerate}

\item\label{OPEfirstexample} We have our first example of an operator product expansion and renormalised product:
\begin{equation*}
\begin{split}
\lim_{z_2 \to z_1} b(z_2) \circ b(z_1) &- \frac{1}{(u(z_2)-u(z_1))^2} du_{z_2} \otimes^s du_{z_1}\\ &= i(z_1) \circ i(z_1) + e(z_1) \circ e(z_1) + 2 e(z_1) \circ i(z_1)
\end{split}
\end{equation*}
The limit is explained in \S \ref{limits}. In  \S \ref{renorm} we will define the notion of the renormalised product of two fields; the above equation will implies that the renormalised product $:b(z)^2:$ equals $i(z) \circ i(z) + e(z) \circ e(z) + 2 e(z) \circ i(z)$.

\item It is clear that the field $e$ ``creates singularities''. For any subset $Y \subset \bbP$ it is useful to consider ${}_Y\bW$, the symmetric algebra over the space of forms with poles contained in $Y$. If $Y$ is a finite set, it is easy to show that this is the space generated by repeated applications of $e(z)$ and its derivatives for $z \in Y$. Less evident is the fact that repeated applications of $b$, its derivatives and their renormalised products, evaluated at $z \in Y$ produce the same result.

\item\label{Heis} Consider a contractible domain $D$ biholomorphic to a disc and bounded by a contour $\gamma = \partial D$. Given a meromorphic function $\phi$, the integral
$$
\Phi_\gamma  \equiv ``\frac{1}{2\pi i} \int_\gamma \phi(z) b(z)"
$$
is formally defined (since $b$ is an operator-valued 1-form) as an operator on the symmetric algebra over $\bK_\gamma$, the subspace of $\bK$ consisting of forms regular along $\gamma$.  We give a rigorous defininiton (in \S \ref{DefsHeis}) of
$$
\lD\Phi_\gamma:\lD\bW \to  \lD\bW
$$
where $\lD\bW$ is the symmetric algebra over the space $\lD\bK$ of forms with all singularities in $D$. To return to the algebraic (rather than analytic) context, we define:
\begin{itemize}
\item given a point $O$ (by considering intersections of discs $D$ containing $O$)
$$
\lO \Phi:\lO\bW \to \lO\bW
$$ 
where $\lO \bW$ is the symmetric algebra over forms with poles only at $O$, and
\item given a point $P$ at "infinity" (by considering unions of discs $D$ excluding $P$)
$$
\Phi_P:\bW_P \to \bW_P
$$ 
where $\bW_P$ is the symmetric algebra over forms \emph{regular} at $P$.
\end{itemize}
We take $\gamma$ to be a ``sufficiently small loop around'' $O$ (respectively, $P$), and we define $\lO \Phi$ and $\Phi_P$ in terms of algebraic residues. The actions define Heisenberg-type central extensions of the space of rational functions.

\item We define the ``energy-momentum tensor'' $T$ by\begin{equation*}
2T(z) \equiv i^2(z)+e^2(z)+2e(z)i(z): \bW^z \to \bW \otimes K^2_z
\end{equation*}
This field is local with respect to $b$ as well as itself -- i.e.,, values at distinct points can be composed and commute, and products have poles of ``uniformly bounded'' order as the points approach each other. Given a meromorphic vector field $X$, (since $T$ is an operator-valued quadratic differential), we are led to consider the integral
$$
\bL^X_\gamma  \equiv ``-\frac{1}{2\pi i} \int_\gamma X(z) T(z)"
$$
Definitions as in the case of (\ref{Heis}) above lead to actions of Virasoro-type extensions of the Lie algebra of meromorphic vector fields.

\suspend{enumerate}

\noindent To go further, we need to introduce a package of paired vector spaces and operators so that $n$-point functions can be expressed in terms of the pairing. This involves a geometric partitioning of $\bbP$ into \emph{either} (\texttt{a}) disjoint disks $D,\ D'$ sharing a common boundary - in which case we pass to a complex analytic framework - \emph{or} (\texttt{b}) a point $P$ and its (Zariski-open) complement. The second is clearly more adapted to the algebro-geometric viewpoint, and the first will enable us to make contact with reflection positivity and hermitian structures on conformal blocks.
\resume{enumerate}

\item Given a partition, we identify two subspaces of $\bW$, and then exhibit their natural pairing. In case (\texttt{a}) the pairing is between $\lDp\bW$ and $\lD\bW$. In case (\texttt{b}) we pair ${}_{P}\bW$  and $\bW_P$. The pairing is such that (roughly speaking) $b$ is skew-symmetric with respect to it. 

\item To obtain hermitian structures, we choose an antiholomorphic involution $C:\bbP \to \bbP$; in terms of a suitable coordinate $u$,
$$
u(C(z))=1/\overline{u(z)}
$$
The fixed point set is a circle that separates two discs $D$ and $D'$, which are exchanged by $C$. We have an antilinear isomorphism $\lD\bW \to \lDp\bW$ induced by 
$$
\lD\bK \ni \alpha \mapsto C^* \overline{\alpha} \in \lDp\bK
$$
This, together with the pairing $\lDp\bW \times \lD\bW \to \bbC$ yields a hermitian structure on $\lD\bW$. An easy argument shows that this is non-negative, and in fact nondegenerate as well.  (\emph{The situation is much subtler for current algebras.}) The induced inner product on $\lD\bK$ is the natural inner product on forms regular on $D'$, and the completion yields the Bergman space of the disc $D'$.

\item The splitting of $b$ into a sum of $i$ and $e$ achieves several ends. In particular, 
\begin{itemize}
\item the field $e$ can be composed freely, and (as already remarked) given $Y \subset  \bbP$,  its values $\{e(z)|z \in Y\}$ (together with derivatives) generate the space ${}_Y \bW$, and
\item  the pairings described above are uniquely determined by requiring that $e$ and $i$ are, up to a sign, adjoints.
\end{itemize}

\end{enumerate}

\subsection{Vertex algebra structure}

In Part \ref{bosonvertexalgebra} we exhibit the vertex algebra structure underlying the boson. This is done in terms of rational functions rather than power series; we formulate a definition of a vertex algebra appropriate to this context and reprove the basic results of the theory.

\subsection{Current algebra}\label{subsection: Current algebra: a summary}

Let $\bg$ be a complex lie algebra endowed with a nondegenerate invariant symmetric bilinear form $(,)$; let $\fg$ denote the lie algebra of meromorphic functions $\bbP \to \bg$.  Let $n$ representations of $\bg$ be given on  $W_1,\dots,W_n$; let
\begin{equation*}
\cW = U\fg \otimes_{U\bg} W_1 \otimes \dots \otimes W_n
\end{equation*}
(regarding $\bg \subset \fg$ as constant maps)
and  $U\fg$, etc. denoting the corresponding universal enveloping algebras.  Given $z \in \bbP$, let $\fg^z \subset \fg$ denote the space of functions regular at $z$; define
\begin{equation*}
\cW^z = U\fg^z \otimes_{U\bg} W_1 \otimes \dots \otimes W_n
\end{equation*}
Let $n$ distinct points $z_1,\dots,z_n$ be given. The representation $W_j$ is supposed to represent an ``insertion'' at $z_j$.

Fix a point $P$ distinct from the $z_j$, and let $u$ be a coordinate with a pole at $P$. The constructions below do not depend on the choice of $u$; the current itself will not depend on $P$ either. 

\begin{enumerate}

\item Define fields $\epsilon$ and $\iota$ as follows: (a) Given $\nu \in \fg$ and $z \notin \{P,\text{poles of $\nu$}\}$, $\epsilon^{\nu(z)}(z): \cW \to \cW \otimes K_z$ is left multiplication by
\begin{equation*}
\epsilon^{\nu(z)}_z= \frac{\nu(z)}{u-u(z)} \otimes du_z
\end{equation*}
for any global coordinate $u$ with a pole at $z$, and (b) $\iota^{\nu(z)}(z): \cW^z \to \cW^z \otimes K_z$ is defined, for $z \notin \{P,\text{poles of $\nu$}\} \cup \{z_j\}$, inductively by
\begin{equation*}
\begin{split}
\iota^{\nu(z)}(z) \circ \balpha  - \balpha \circ \iota^{\nu(z)}(z)&=-[\epsilon^{\nu(z)}(z),\balpha]-(\nu(z), d\balpha(z))\\
&+ \iota^{[\nu(z), \balpha(z)]}(z)+\epsilon^{[\nu(z), \balpha(z)]}(z)
\end{split}
\end{equation*}
The starting point of the inductive definition is
\begin{equation*}
\iota^{\nu(z)}(z) (\ttw_1\otimes \dots \otimes \ttw_n)=\sum_j \ttw_1 \otimes \dots \otimes \frac{\nu(z)(\ttw_j)du_{z}}{u(z)-u(z_j)} \otimes \dots \otimes \ttw_n
\end{equation*}

\item The field $j$ is defined to be the sum (defined for $z \notin \{P,\text{poles of $\nu$}\} \cup\{z_j\}$):
\begin{equation*}
j^{\nu(z)}(z)=\iota^{\nu(z)}(z)+\epsilon^{\nu(z)}(z): \cW^z \to \cW \otimes K_z
\end{equation*}
Although $\iota$ and $\epsilon$ depend on the choice of $P$, the current extends to $P$ and is independent of this choice. In the absence of ``insertions'', it is an $Aut(\bbP)$-covariant object.

\suspend{enumerate}
We will often consider \emph{constant} functions $\nu$ taking the value $v\in \bg$. In this case we write $j^\nu=j^v$.
\resume{enumerate}

\item Given $v_1,\ v_2\ \in \bg$ and points $z_1 \ne z_2$, $j^{v_1}(z_1)$ and $j^{v_2}(z_2)$ can be composed on $\cW^{z_1} \cap \cW^{z_2}$ and  commute, though this is much less evident than in the case of the chiral boson. Together with (\ref{OPEj}) below this shows that $j$ is local with respect to itself. \emph{\textbf{Note that the field $\epsilon$ is NOT local with respect to itself, nor $\iota$.}}

\item\label{OPEj} As $z_2 \to z_1$ we have
\begin{equation*}
\begin{split}
j^{v_2}(z_2)j^{v_1}(z_1)&+\frac{(v_1,v_2)du_{z_1}du_{z_2}}{(u(z_2)-u(z_1))^2}-\frac{1}{u(z_2)-u(z_1)} du_{z_2}j^{[v_2,v_1]}(z_1)\\
\longrightarrow &\iota^{v_2}(z_1)\iota^{v_1}(z_1)
+\epsilon^{v_2} (z_1)  \epsilon^{v_1} (z_1)\\
+&\epsilon^{v_2} (z_1)  \iota^{v_1} (z_1)
+\epsilon^{v_1} (z_1)  \iota^{v_2} (z_1)\\
+& [d\hiota^{[v_2,v_1]}/dz](z_1)du_{z_1}^2\\
\end{split}
\end{equation*}
where we use the notation $\iota(z)=\hiota(z) du_z$.

\item Given a meromorphic function $\nu: \bbP \to \bg$ and a contour $\gamma$, (since $j$ is an operator-valued 1-form) the integral
$$
J^{\nu}_\gamma  \equiv ``\frac{1}{2\pi i} \int_\gamma j^{\nu(z)} (z)"
$$
is formally defined as an operator.  This leads to rigorous definitions of operators (a) acting on $\lDp\cW'$ and $\lD \cW$ and (b) on $\lP \cW$ and $\cW_P$. The actions define central extensions of $\fg$; in particular,
$$
[\lP J^{\mu},\lP J^{\nu}]= \lP J^{[\mu,\nu]} + Res_P(\mu,d\nu)
$$

\item If we partition $\bbP$, we are led to define paired spaces $\lDp\cW'$ and $\lD\cW$ (respectively, $\lP\cW'$ and $\cW_P$) as in the case of the chiral boson.  The pairing is such that $\iota$ and $\epsilon$ are, up to a sign, adjoints. This is a rational analogue of the Shapovalov form.

\end{enumerate}

\subsection{Boson-fermion equivalence; CFTs in higher genus}

In the final sections, we outline our approach to boson-fermion equivalence in genus zero, and then go on to describe certain field theories (the neutral fermion, the $b-c$ system, the chiral boson) in arbitrary genus.  Specialising to genus zero, we recover the description of the neutral fermion in terms of modes. We show that in the case of the $b-c$ system, a composite field $\sfb=\bdots bc\bdots $ has a two-point function appropriate to a boson. 

As for current algebra in higher genus, work is in progress.

\subsection{Notation}\label{Notation}

All tensor products are over $\bbC$ unless otherwise specified. 

We denote by $K_z$ the fibre at $z$ of the canonical line bundle $K$, i.e., the bundle of (1,0)-forms.  By a coordinate on $\bbP$, we mean a meromorphic function $u$ with one (simple) pole. If $z$ is \emph{not} this pole, $du_z \in K_z$ will denote the (nonzero) value of the differential at $z$.

We might need to fix a base-point $P$ in $\bbP$.  Once this is done, when we choose a coordinate $u$, it will be such that $u(P)=\infty$.  On occasion, we will choose in addition a point $O \ne P$  and demand that $u(O)=0$. Such a coordinate (with its pole at $P$ and vanishing at $O$) is unique up to a scalar, so one might well wonder why we do not start with this choice. This is partly a matter of taste; we will also see that $O$ and $P$ play rather different roles.

We might also work with a coordinate $\tu=1/u$, with a pole at $O$ and vanishing at $P$.

We will often consider a (open) domain $D$ (with smooth boundary)  biholomorphic to a disc. Its oriented boundary $\partial D$ is a contour (smooth, simple, oriented closed curve) which we will sometimes also denote by $\gamma$. The complement of the closure of $D$ will be denoted $D'$. Any choice of $P$ will be in $D'$, and $O$ in $D$. Note that any one of the triple $\{D,D',\gamma\}$ determines the other two.

We denote by $\cK$ the field of meromorphic functions. Recall that $\bK$ denotes the space of meromorphic one-forms of the second kind (that is, those with all residues zero) on $\bbP$. 

Given a subset $Y \subset \bbP$, we denote by  $\bK_Y$ the space of meromorphic forms (of the second kind) regular at all points of $Y$, and by ${}_Y\bK$ the space of such forms with all singularities contained in $Y$. If $Y$ consists of a single point $P$, we set $\bK_P=\bK_{\{P\}}$, etc.. As an exception, given a ``variable" point $z \in \bbP$, we let $\bK^z$  be the subspace of $\bK$ consisting of forms that are regular at $z$. These conventions apply in related contexts, for example, to spaces ``built out of $\bK$" such as the symmetric algebra $\bW$.  To summarise, \emph{the appearance of a subset as a subscript on the \textup{left} indicates that singularities are contained in that subset; its appearance on the \textup{right} (as a subscript or superscript) signifies that the relevant functions are regular there.}

When a coordinate $u$ is chosen we will write $b(z)=\hb(z) du_z$ (where $\hb(z)$ maps $\bW^z$ to $\bW$), and so on for other fields.

\pagebreak

\part{\textbf{{\Large The boson on the projective line}}}

 \section{\textbf{The boson: the fields}}

\subsection{The boson}\label{section: The boson}

Let $\bW$ denote the symmetric algebra over $\bK$:
\begin{equation*}
\bW =\bbC \oplus \bK  \oplus S^2\bK \oplus \dots \ \ \text{(algebraic direct sum)}
\end{equation*}
We denote by $\mathbf{1}$ the vacuum vector $1\ \in \bbC \subset \bW$.

Given $z \in \bbP$, we define $e_z \in \bK \otimes K_z$, by
\begin{equation*}
e_z= -\frac{du}{(u-u(z))^2} \otimes du_z
\end{equation*}
where $u$ is any coordinate regular at $z$. Note that this is well-defined independent of the coordinate $u$. Define the ``field''\footnote[1]{To be precise, $e(z)$ is the value of the field $e$ at $z$.} $e(z):\bW \to \bW \otimes K_z$ to be the multiplication by $e_z$:
\begin{equation*}
e(z) ({\alpha}_1 \otimes^s \dots  \otimes^s {\alpha}_p) =  {\alpha}_1 \otimes^s \dots \otimes^s {\alpha}_p \otimes^s e_z
\end{equation*}
where $\alpha_i \in \bK$ and by $\otimes^s$ we mean the symmetric tensor product.

Given  $z \in \bbP$, we let $\bK^z$  be the subspace of $\bK$ consisting of forms that are regular at $z$. Define $\bW^z$ to be the corresponding symmetric algebra over $\bK^z$. We now define a second field $i$, this time acting from $\bW^z$ to $\bW^z \otimes K_z$ by the formula:
\begin{equation*}
i(z) (\alpha_1 \otimes^s \dots \otimes^s \alpha_p) =  -\sum_i ({\alpha}_1 \otimes^s \dots \widehat{{{\alpha}_i}} \dots \otimes^s {\alpha}_p) {\alpha}_i(z)
\end{equation*}
Here $\alpha_1,\dots , \alpha_p$ belong to $\bK^z$ and the hat marks a term to be omitted. Needless to say, we define $i(z)({\mathbf 1})=0$. Note that $i$ can be defined inductively by this last condition and
$$
[i(z),\alpha] = -\alpha(z)
$$
where $\alpha$ is any form regular at $z$, which within the commutator stands for the operation $\alpha \otimes^s$.

Note that $i$ and $e$ are functions from $\bbP$ to operators, covariant with respect to automorphisms of $\bbP$. Further, given a domain (respectively, Zariski open) $U \subset \bbP$ and a vector $\ttv$ in $\bW$ that belongs to the intersection of $\{\bW^z|z \in U\}$, the vectors $i(z)\ttv$ and $e(z)\ttv$ vary holomorphically (respectively, algebraically) in $z$. We will refer to operator-valued functions such as $e(z)$ and $i(z)$ as fields.
We set
\begin{equation*}
b(z)=i(z)+e(z)
\end{equation*}
with domain $\bW^z$ and range $\bW \otimes K_z$. 

\subsection{Boson current algebra}

As remarked earlier, (on $\bbP$) taking the derivative sets up a bijection $d:\cK/\bbC \to \bK$ between the space of meromorphic functions (modulo constants) and the space of meromorphic differentials of the second kind.

We will use this to give an alternate description of the boson, wherein it becomes a simple case of a ``current algebra" associated to $\bbC$ regarded as an abelian lie algebra.   The reader might skip this section and return to it as a preparation for \S \ref{section: Current algebra: the fields}.

We will need to choose a point $P$ and a coordinate $u$ such that $u(P )=\infty$. On $\bbP \setminus P$ we will construct fields $\iota$ and $\epsilon$ that depend on the choice of $P$ (but not $u$), and show that their sum $j=\iota+\epsilon$ extends to $P$, and is a natural $Aut(\bbP)$-invariant field on $\bbP$.

Recall that the local ring at $P$ is the space of meromorphic functions regular at $P$, and conventionally denoted $\cO_P$; let $\mathfrak{m}_P \subset \cO_P$ be the maximal ideal consisting of functions vanishing there. Let $\mathbb{W}_{P,-}$ be the symmetric algebra over $\mathfrak{m}_P$:
\begin{equation*}
\mathbb{W}_{P,-} = \bbC \oplus \mathfrak{m}_P \oplus
 S^2 \mathfrak{m}_P \oplus \dots
\end{equation*}
(As before, let $\mathbf{1}=1 \in \bbC \subset \mathbb{W}_{P,-}$.) For $z \in \bbP \setminus P$, define
\begin{equation*}
    \epsilon_z = \frac{1}{u-u(z)} \otimes du_z
\end{equation*}
(Note that the expression $\frac{1}{u-u(z)} \otimes du_z$ is a canonical element of $\mathfrak{m}_P \otimes K_z$, independent of the coordinate.) Define the field $\epsilon(z):\mathbb{W}_{P,-}   \to  \mathbb{W}_{P,-} \otimes K_z$ to be multiplication by $\epsilon_z$.  Given $z \in \bbP \setminus P$, we define $\mathfrak{m}_P^z$ to be the subspace of $\mathfrak{m}_P$ consisting of functions that are regular at $z$ (that is, the intersection of the local ring at $z$ and the maximal ideal at $P$). Let
\begin{equation*}
\mathbb{W}^z_{P,-} = \bbC \oplus \mathfrak{m}^z_P \oplus
 S^2 \mathfrak{m}^z_P \oplus \dots
\end{equation*}
Define $\iota(z):\mathbb{W}^z_{P,-} \to \mathbb{W}^z_{P,-} \otimes K_z$ by the formula:
\begin{equation}\label{iota1}
\iota(z) (\balpha_1 \otimes^s \dots \otimes^s \balpha_p) =  -\sum_i ({\balpha}_1 \otimes^s \dots \widehat{{{\balpha}_i}} \dots \otimes^s {\balpha}_p) d{\balpha}_i (z)
\end{equation}
and $\iota(z)(\mathbf{1})=0$. The equation (\ref{iota1}) is equivalent to
\begin{equation}\label{iota2}
\iota(z) \circ \balpha - \balpha \circ \iota(z)= - d{\balpha} (z)
\end{equation}

Let $\hbbW_{P}$ be the symmetric algebra over $\cO_P$, modulo the ideal $<constants>$ generated by the constant functions $\bbP \to \bbC$ :
\begin{equation*}
\hbbW_{P} = \bbC \oplus \cO_P \oplus
 S^2 \cO_P \oplus \dots /<constants>
\end{equation*}
Clearly, the inclusion $\mathfrak{m}_P \to \cO_P$ induces a linear isomorphism $\mathbb{W}_{P,-} \to \hbbW_P$.  Given $z \ne P$, we have an analogous ismorphism $\mathbb{W}^z_{P,-} \to \hbbW^z_P$. We can regard $\epsilon$ and $\iota$ as maps from $\hbbW_P$ and $\hbbW^z_P$ respectively.

Let $\bK_P \subset \bK$  be the sub-space of forms regular at $P$. When we need to consider forms that are regular at a second (variable) point $z$, we will use the notation $\bK^z_P$. We let $\bW_P$ denote the symmetric algebra over $\bK_P$ and $\bW_P^z$ the symmetric algebra over $\bK_P^z$. 

Note that taking the derivative $\balpha \mapsto \alpha \equiv d\balpha$ sets up a linear isomorphism $\mathfrak{m}_P \to \bK_P$.
The induced isomorphisms $\mathbb{W}_{P,-} \to \bW_P$ and $\mathbb{W}^z_{P,-} \to \bW^z_P$ take the fields $\epsilon, \ \iota$ (respectively) to the fields $e,\ i$ described in section \S \ref{section: The boson}. More precisely,

\begin{proposition} Let $z \in \bbP \setminus P$. The operators $e(z):\bW \to \bW \otimes K_z$ and $i(z):\bW^z \to \bW^z \otimes K_z$ leave invariant $\bW_P$ and $\bW_P^z$ respectively, and we have commutative diagrams
\begin{equation*}
\begin{CD}
{\hbbW_P \sim \mathbb{W}_{P,-}} @>\epsilon(z)>> {\hbbW_{P} \otimes K_{z} \sim \mathbb{W}_{P,-} \otimes K_{z}}\\
@VdV{\sim}V @VdV{\sim}V   \\
\bW_P @>e(z)>> \bW_P \otimes K_{z} \\
\end{CD}
\end{equation*}
and
\begin{equation*}
\begin{CD}
{\hbbW^z_P \sim \mathbb{W}^z_{P,-}} @>\iota(z)>> {\hbbW^z_{P} \otimes K_{z} \sim \mathbb{W}_{P,-}^z \otimes K_{z}} \\
@VdV{\sim}V @VdV{\sim}V   \\
\bW_P^z @>i(z)>> \bW_P^z \otimes K_{z} \\
\end{CD}
\end{equation*}
\end{proposition}

Suppose given a second point, $\tP$; denote the corresponding fields $\tepsilon$ and $\tiota$. We wish to compare these with $\epsilon$ and $\iota$ respectively. Choose $u$ such that $u(\tP)=0, u(P)=\infty$, and set $\tu=1/u$.  We have, for $z \ne P,\ \tP$
\begin{equation}
\label{PtildeP}
\begin{split}
\tepsilon_z&=\frac{1}{\tu-\tu(z)} d\tu_z\\
&=\frac{u u(z)}{u(z)-u} . \frac{-du_z}{u(z)^2}\\
&=\frac{u}{u-u(z)}. \frac{du_z}{u(z)}\\
&=\frac{1}{u-u(z)}{du_z}+\frac{du_z}{u(z)}\\
&=\epsilon_z+\frac{du_z}{u(z)}\\
\end{split}
\end{equation}
Let $\cK_{\{P,\tP\}}$ be the intersection in $\cK$ of the local rings at $P$ and $\tP$; i.e., the algebra of rational functions regular at both points. Define $\hbbW_{\{P,\tP\}}$ to be the symmetric algebra over $\cK_{\{P,\tP\}}$ modulo the ideal generated by constant functions, and define similarly $\hbbW^z_{\{P,\tP\}}$. We have inclusions $\hbbW_{\{P,\tP\}} \hookrightarrow \hbbW_P$, etc. Using (\ref{PtildeP}) we see
\begin{proposition}
For $z\ne P, \tP$ the maps $\hbbW_{\{P,\tP\}} \to \hbbW_{\{P,\tP\}} \otimes K_z$ induced by $\epsilon, \ \tepsilon$ agree. Similarly, the maps $\hbbW^z_{\{P,\tP\}} \to \hbbW^z_{\{P,\tP\}} \otimes K_z$ induced by $\iota, \ \tiota$ agree.
\end{proposition}

Tying together these propositions, we get
\begin{proposition} Let $\hbbW$ be the symmetric algebra over $\cK$ modulo the ideal generated by constant functions, etc. Then $\epsilon$, $\iota$ patch to give maps making then following diagrams commute:
\begin{equation*}
\begin{CD}
\hbbW= @>\epsilon(z)>> \hbbW \otimes K_{z}\\
@VdV{\sim}V @VdV{\sim}V   \\
\bW @>e(z)>> \bW \otimes K_{z} \\
\end{CD}
\end{equation*}
and
\begin{equation*}
\begin{CD}
\hbbW^z= @>\iota(z)>> \hbbW^z \otimes K_{z}\\
@VdV{\sim}V @VdV{\sim}V   \\
\bW^z @>i(z)>> \bW^z \otimes K_{z} \\
\end{CD}
\end{equation*}

\end{proposition}

\section{\textbf{Algebro-geometric interlude: fields as morphisms of sheaves}}

Informally, spaces such as $\bW$ and $\bW^z$ are fibres of  infinite-rank vector bundles; fields are morphisms of such vector bundles.  It turns out that the best way to formalise these notions is terms of sheaves and morphisms, which we now proceed to do. 

Let $\bbX$ be a smooth complex projective curve, $\bbX'$ a second copy.  Given an open set $U$ of $\bbX$, we denote by $U'$ the same set regarded as a subset of $\bbX'$. For any sheaf $E$ on $\bbX$, $E'$ will denote the corresponding sheaf on $\bbX'$. Consider the following two sheaves of $\cO_\bbX$-algebras on $\bbX$:  
\begin{itemize}
\item $\bbX \underset{open}{\supset} U \mapsto  \sfO(U)=\lim^1_{D}  \Gamma(U \times \bbX', \cO_{\bbX \times \bbX'} (D))$ where the limit is taken over effective divisors $D \subset \bbX \times \bbX'$ such that $D$ is either ``horizontal'' (pulled back from $\bbX'$) or a multiple of the diagonal $\Delta$.
\item $\bbX \underset{open}{\supset} U \mapsto  \sfOz(U)=\Gamma(U\times U', \cO_{U \times U'})$.
\end{itemize}
Given $z\in \bbX$, the fibre (the tensor product of the stalk with the residue field) is respectively
\begin{itemize}
\item the function field $\cK$ of $\bbX$, and
\item the local ring $\cO_z$.
\end{itemize}
More generally, given a vector bundle $E$ on $\bbX$, consider the sheaves of $\cO_\bbX$-modules on $\bbX$:
\begin{itemize}
\item  $\bbX \underset{open}{\supset} U  \mapsto   \sfE(U) \equiv \Gamma(U \times \bbX', \pi_{{}_{\bbX'}}^* E' (D))$ where the limit is taken over effective divisors $D \subset \bbX \times \bbX'$ as above. 
\item $\bbX \underset{open}{\supset} U  \mapsto   \sfEz(U) \equiv \Gamma(U\times U', \pi_{{}_{\bbX'}}^* E')$.
\end{itemize}
Given $z\in \bbX$, the fibre is now respectively
\begin{itemize}
\item the space of meromorphic sections $\cK(E)$ on $\bbX$, and
\item the space of meromorphic sections of $E$ regular at $z$, which we denote $\cK^z(E)$.
\end{itemize}
Note that these are locally free of rank $=rank\ E$ over the (respective) sheaves of $\cO_\bbX$-algebras described earlier. In each case the second sheaf of the pair is a (\textbf{\emph{not quasicoherent}}) subsheaf of the first sheaf, which is indeed quasi-coherent.

We will consider the symmetric algebras, over $\cO_\bbX$, of $\sfE$ and $\sfEz$. Set 
$$
\bbX'^{[l]}=\overbrace{\bbX' \times \dots \times \bbX'}^{l\  factors}
$$
Consider the $Sym_l$-equivariant bundle $\otimes_{i=1}^l E'_i$ on $\bbX'^{[l]}$ with the lift of the action of the symmetric group $Sym_l$ being taken ``without signs''. For any divisor $D$ on $\bbX \times \bbX'$, let $D_i$ denote the divisor on
$$
\bbX \times \bbX'^{[l]}
$$
obtained by pulling back $D$ from the factor $\bbX \times \bbX'_i$. Then, given $U \underset{open}\subset \bbX$,
\begin{itemize}
\item $S^l \sfE (U) =\lim^1_{D} \Gamma (U \times \bbX'^{[l]}, \cO_\bbX \otimes \otimes_{i=1}^l E'_i(D_i))^{Sym_l}$ where the limit is taken over effective divisors $D \subset \bbX \times \bbX'$ as above (i.e.,``either horizontal or diagonal''). 
\item $S^l \sfEz (U) =  \Gamma (U \times U'^{[l]}, \cO_\bbX \otimes \otimes_{i=1}^l E'_i)^{Sym_l}$  
\end{itemize}
With a little work, one sees that given $z\in \bbX$, the fibres are respectively\begin{itemize}
\item  $S^l \cK(E)$ and
\item $S^l \cK^z(E)$
\end{itemize}

When $E=K$ the canonical bundle, we have to consider a variant. Consider the sheaves of $\cO_\bbX$-modules on $\bbX$:
\begin{itemize}
\item   $\bbX \underset{open}{\supset} U  \mapsto   \sfKs(U) \equiv \Gamma_2(U\times \bbX', \pi_{{}_{\bbX'}}^* E'(D))$ where the limit is taken over effective divisors $D \subset \bbX \times \bbX'$ as above, and the subscript $2$ indicates that we only consider sections whose residue along $D$ (which would be a function on $D$) vanishes.
\item $\bbX \underset{open}{\supset} U  \mapsto   \sfKsz (U) \equiv \Gamma_2(U\times U', \pi_{{}_{\bbX'}}^* E')$. Here the subscript refers to the restriction that the residue along the complement of $U'$ vanishes.
\end{itemize}
Given $z\in \bbX$, the fibre is now respectively
\begin{itemize}
\item the space $\bK$ of meromorphic forms of the second kind on $\bbX$, and
\item the space of  $\bK^z$ of meromorphic forms of the second kind on $\bbX$ regular at $z$
\end{itemize}

We next consider the symmetric algebras, over $\cO_\bbX$, of $\sfKs$ and $\sfKsz $. 
Consider the $Sym_l$-equivariant bundle $\otimes_{i=1}^l K'_i$ on $\bbX'^{[l]}$ with the lift of the action of the symmetric group $Sym_l$ being taken ``without signs''. 
(Invariant sections are polydifferentials.) Then, given $U \underset{open}\subset \bbX$,
\begin{itemize}
\item $S^l \sfKs (U) =\lim^1_{D} \Gamma_2 (U \times \bbX'^{[l]}, \cO_\bbX \otimes \otimes_{i=1}^l K'_i(D_i))^{Sym_l}$ where the limit is taken over effective divisors $D \subset \bbX \times \bbX'$ as above. The subscript $2$ refers to the restriction that we consider sections whose residue along each $D_i^{reduced}$ (which would be a meromorphic section of $\otimes_{j \ne i} K'_j(D_j)$) vanishes.
\item $S^l \sfKsz  (U) =  \Gamma_2 (U \times U'^{[l]}, \cO_\bbX \otimes \otimes_{i=1}^l K'_i)^{Sym_l}$ . The subscript $2$ restricts the sections to those that have vanishing residue along each component of the complement of each $U'_i$.
\end{itemize}
In this case, one sees that given $z\in \bbX$, the fibres are respectively\begin{itemize}
\item  $S^l \bK$ and
\item $S^l \bK^z$
\end{itemize}

\begin{remark} The sheaves defined above are all $\cD$-modules.
\end{remark}

Restriction to the diagonal $\Delta_l$ induces a map of sheaves $i^{[l]}:S^l \sfKsz  \to S^{l-1} \sfKsz  \otimes_{\cO_\bbX} K_\bbX$; explicitly
\begin{equation*}
\Gamma_2 (U \times U'^{[l]}, \cO_\bbX \otimes \otimes_{i=1}^l K'_i)^{Sym_l} \to \Gamma_2 (U \times U'^{[l-1]}, K_\bbX \otimes \otimes_{i=1}^{l-1} K'_i)^{Sym_{l-1}} 
\end{equation*}
induces
\begin{equation*}
i^{[l]}(U): S^l \sfKsz  (U) \to S^{l-1} \sfKsz  (U) \otimes_{\cO_\bbX(U)}K_\bbX(U)
\end{equation*}

Specialise to $\bbX=\bbP$. We can regard $e_z$ (defined in \ref{section: The boson}) as the value at $z$ of a canonical section of $\sfKs \otimes K_\bbP$ on $\bbP$:
$$
e_z(u)=-\frac{du}{(u-u(z))^2} du_z \in \Gamma_2(\bbP \times \bbP',K_\bbP\otimes K_{\bbP'}(2\Delta))
$$
Multiplying by this section and averaging over $Sym_{l+1}$ gives the sheaf map $e^{[l]}:S^l \sfKs \to S^{l+1} \sfKs \otimes_{\cO_\bbX} K_\bbP$; explicitly
\begin{equation*}
\Gamma_2 (U \times \bbP'^{[l]}, \cO_\bbP \otimes \otimes_{i=1}^l K'_i(D_i))^{Sym_l} \to \Gamma_2 ( U \times \bbP'^{[l+1]}, K_\bbP \otimes \otimes_{i=1}^{l+1} K'_i(D_i+2\Delta_i))^{Sym_{l+1}}
\end{equation*}
induces
\begin{equation*}
\begin{split}
e^{[l]}(U): S^l \sfKs (U) \to S^{l+1} \sfKs (U) \otimes_{\cO_\bbP(U)}K_\bbP(U) \\
\end{split}
\end{equation*}

Adding $i^{[l]}$ and $e^{[l]}$, we get the map of sheaves $b^{[l]}:S^l \sfKsz  \to S^{l-1} \sfKs \otimes_{\cO_\bbP} K_\bbP \oplus S^{l+1} \sfKs \otimes_{\cO_\bbP} K_\bbP$ induced by
\begin{equation*}
\begin{split}
\Gamma_2 (U \times \bbP'^{[l]}, \cO_\bbP \otimes \otimes_{i=1}^l K'_i(D_i))^{Sym_l} &\to \Gamma_2 (U \times \bbP'^{[l-1]}, K_\bbP \otimes \otimes_{i=1}^{l-1} K'_i(D_i))^{Sym_{l-1}}\\ 
&\oplus  \Gamma_2 ( U \times \bbP'^{[l+1]}, K_\bbP \otimes \otimes_{i=1}^{l+1} K'_i(D_i+2\Delta_i))^{Sym_{l+1}}\\
\end{split}
\end{equation*}
for any effective divisor $D$ whose support is $\bbP \setminus U$.

\section{\textbf{The boson, continued}}

\subsection{Derivatives}\label{subsection: Derivatives}

Derivatives will be defined in a straightforward way that is best illustrated in context. 

To start with, consider the action of global vector fields. Let $X$ be a regular vector field on $\bbP$. This acts on $\bK$ by Lie derivative $\cL_X$, preserving (for any $z$) $\bK^z$; the action is extended to $\bW$ and $\bW^z$ as a derivation:
\begin{equation*}
\bL^X (\alpha_1 \otimes^s \dots \otimes^s \alpha_p) =  \sum_i ({\alpha}_1 \otimes^s \dots \cL_X(\alpha_i) \dots \otimes^s {\alpha}_p)
\end{equation*}
Note that the same formula defines an action of \emph{meromorphic} vector fields; this will be part of the action of the Virasoro algebra to be defined later. The notation $\bL^X$ is anticipates one that will be introduced in that context.

Derivatives of fields - we use the notation $\bcL_X$ to denote the derivative of a field with respect to a (regular) vector field $X$ - are defined because a field is a morphism of $D$-modules. For example
\begin{equation*}
\begin{split}
\bcL_X e(z) &= -\{\frac{2\xi(u(z))du_z}{(u-u(z))^3}+\frac{\xi'(u(z))du_z}{(u-u(z))^2}\}du \otimes^s \ [.]\\
&= -\frac{\{2\xi(u(z))+\xi'(u(z))(u-u(z))\}du_z}{(u-u(z))^3}du \otimes^s \ [.]\\
\end{split}
\end{equation*}
where $X(z)=\xi(u(z))d/du$ and  $\xi'(u(z))=d\xi(u(z))/du$.

On the other hand the commutator 
$$
[\bL^X,e(z)]=-\frac{\{2\xi(u)-\xi'(u)(u-u(z))\}du_z}{(u-u(z))^3}du \otimes^s \ [.]
$$
Since $X$ is a globally regular vector field, $\xi$ is a polynomial function of degree $\le 2$, and one checks that
$$
\bcL_X e(z)=[\bL^X, e(z)]
$$
This is an expression of the fact that $e$ is $Aut\ \bbP$-covariant. A similar equation  holds for the derivative of $i$ and therefore $b$.:
$$
\bcL_X b(z)=[\bL^X, b(z)]
$$

\subsection{Composition of operators, $n$-point functions, Locality, OPE}\label{subsection: OPEetc.}

\subsubsection{Composition of operators} 

Given distinct points $z_1$ and $z_2$, any (one) operator of the pair $e(z_1), \ i(z_1)$ can be composed with any of the pair $e(z_2), \ i(z_2)$ on $\bW^{z_1,z_2} \equiv \bW^{z_1} \cap \bW^{z_2}$. Clearly $e(z_1),\ e(z_2)$ commute, as do $i(z_1),\ i(z_2)$; on the other hand
\begin{equation}
\label{iecommute}
i(z_1) \circ e(z_2) - e(z_2) \circ i(z_1) = \frac{1}{(u(z_1)-u(z_2))^2} du_{z_1} \otimes^s du_{z_2}
\end{equation}
An important consequence is that for distinct points $z_1$ and $z_2$, the operators $b(z_1), \ b(z_2)$ (can be composed and) commute.

We formulate all this carefully, clarifying in passing the notation $du_{z_1} \otimes^s du_{z_2}$ in the above expression.

Consider first the diagram of maps
\begin{equation*}
\begin{CD}
\bW @>e(z_1)>> \bW \otimes K_{z_1} @>e(z_2)>> \bW \otimes K_{z_2} \otimes K_{z_1} \\
@| @. @VV{\text{exchange}}V \\
\bW @>e(z_2)>> \bW \otimes K_{z_2} @>e(z_1)>> \bW \otimes K_{z_1} \otimes K_{z_2}\\
\end{CD}
\end{equation*}
Note that the above diagram commutes provided the exchange map $K_{z_2} \otimes K_{z_1} \to K_{z_1} \otimes K_{z_2}$ is $\phi \otimes \psi \mapsto \psi \otimes \phi$ (\emph{without} a minus sign).

We also have the commutative diagram
\begin{equation*}
\begin{CD}
\bW^{z_1,z_2} @>i(z_1)>> \bW^{z_1,z_2} \otimes K_{z_1} @>i(z_2)>> \bW^{z_1,z_2} \otimes K_{z_2} \otimes K_{z_1}\\
@| @. @VV{\text{exchange}}V \\
\bW^{z_1,z_2} @>i(z_2)>> \bW^{z_1,z_2} \otimes K_{z_2} @>e(z_1)>> \bW^{z_1,z_2} \otimes K_{z_1} \otimes K_{z_2}\\
\end{CD}
\end{equation*}

Finally, we have the diagram (\emph{not} commutative) of maps
\begin{equation*}
\begin{CD}
\bW^{z_1,z_2} @>i(z_1)>> \bW^{z_1,z_2} \otimes K_{z_1} @>e(z_2)>> \bW^{z_1} \otimes K_{z_2} \otimes K_{z_1}  @>{\text{inclusion}}>> \bW \otimes K_{z_2} \otimes K_{z_1}\\
@| @. @. @VV{\text{exchange}}V \\
\bW^{z_1,z_2} @>e(z_2)>> \bW^{z_1} \otimes K_{z_2} @>i(z_1)>> \bW^{z_1} \otimes K_{z_1} \otimes K_{z_2}  @>{\text{inclusion}}>> \bW \otimes K_{z_1} \otimes K_{z_2}\\
\end{CD}
\end{equation*}
The equation (\ref{iecommute}) refers to this last diagram; the notation $du_{z_1} \otimes^s du_{z_2}$ is chosen to indicate the fact that the ``exchange" arrows do not have a negative sign associated to them. In other words, as $z_1,z_2$ vary, we are dealing with ``polydifferentials" rather than differential forms.

The bi-differential in (\ref{iecommute}) above can be easily checked to be $Aut(\bbP)$-invariant, and therefore well-defined independent of the coordinate. We will encounter this  invariant bi-differential often, so it is useful to set
\begin{equation}\label{invbidiff}
\omega  = \frac{1}{(u_1-u_2)^2} du_1 \otimes^s du_2
\end{equation}

\subsubsection{$n$-point functions; Wick's Theorem}\label{Wick}

We can immediately define and compute the ``$n$-point functions'' of the field $b$. Let $z_1,\dots, z_n$ be distinct points on $\bbP$; we define the $n$-point function $\langle b(z_1)b(z_2)\dots b(z_n)\rangle$ by the condition
\begin{equation*}
b(z_1)b(z_2)\dots b(z_n) \mathbf{1} = \langle b(z_1)b(z_2)\dots b(z_n)\rangle + \text{terms in} \ \oplus_{m>0} S^m \bK 
\end{equation*}
The computation is simple enough.
\begin{equation*}
\begin{split}
\langle b(z_1)b(z_2)\dots b(z_n)\rangle &=\langle (e(z_1)+i(z_1))b(z_2) \dots b(z_n) \rangle \\
&=\langle i(z_1)b(z_2)\dots b(z_n) \rangle \\
&=\langle \sum_{l=2}^n  b(z_2)\dots [i(z_1),b(z_l)] \dots b(z_n) \rangle \\
&= \sum_{l=2}^n  \frac{du_{z_1} \otimes^s du_{z_l}}{(u(z_1)-u(z_l))^2} \langle b(z_2)\dots \widehat{b(z_l)} \dots b(z_n) \rangle \\
\end{split}
\end{equation*}
where the hat marks a term to be omitted. We have used the fact that $i$ annihilates the vacuum, and that the image of $e$ is strictly of positive degree. By induction, we get the expression given by Wick's theorem:
\begin{equation*}
\langle b(z_1)b(z_2)\dots b(z_n)\rangle=
\begin{cases}
\begin{aligned}
& \ \ \ 0 \ \ \ \ \ \text{if $n$ is odd, and}\\
&\{\sum_{\substack{partitions\\ of\ \{1,\dots,n\}\\ into\ pairs}}  \ \  \prod_{\substack{(unordered)\ pairs\ \{a,b\}\\ in\ the\ partition}} \frac{1}{(u(z_a)-u(z_b))^2} \}\\
&\ \ \  \times du_{z_a} \otimes^s \dots \otimes^sdu_{z_n}\ \ \text{if $n$ is even}\\
\end{aligned}
\end{cases}
\end{equation*}

\subsubsection{Operator product expansions (OPE)} We wish to study the behaviour of the product $b(z_b)b(z_a)$ as $z_b \to z_a$. If $\ttv \in \bW^{z_a}=\bbC \oplus \bK^{z_a} \oplus S^2\bK^{z_a} \oplus \dots$, there are finitely many nonzero components of $v$ in the summands on the right. There will be a (Zariski) neighbourhood  $U$ of $z_a$ (depending on $\ttv$) such that all these components are built out of functions regular in that neighbourhood, and $b(z_b) \circ b(z_a) [\ttv]$ makes sense for $z_b \ne z_a$ belonging to such an ``adapted" neighbourhood. If
we define $\bdots b(z)^2\bdots =i(z) \circ i(z) + e(z) \circ e(z) + 2 e(z) \circ i(z)$, we have
\begin{itemize}
  \item $\bdots b(z)^2\bdots $ is defined on $\bW^z$, and
\item For $\ttv \in \bW^{z_a}$ and $z_b \ne z_a \in U$, with $U$ a neighbourhood adapted to $\ttv$, 
\begin{equation}\label{normalorderedbsquared}
\lim_{z_b \to z_a} (b(z_b) \circ b(z_a) - \omega(z_b,z_a))[\ttv] = \bdots b(z_a)^2\bdots [\ttv]
\end{equation}
\end{itemize}

\subsubsection{Limits}\label{limits}

The above limit is to be understood as follows.  

For $(z_b,z_a) \in U_b \times U_a \setminus \Delta_{ba}$ (with $U_a$ and $U_b$ copies of $U$), set
\begin{equation*}
(b(z_b) \circ b(z_a) - \omega(z_b,z_a))[\ttv] = \sum_{m} \ttv''_m(z_b,z_a)
\end{equation*}
For each $m$, $\ttv''_m$ corresponds to an invariant section $\tttv''_m$ of $\otimes_{i=1}^m K_i  \otimes K_b \otimes K_a \otimes \cO (2(\sum_{i_b=1}^m \Delta_{i_b,1} + \sum_{i_a=1}^m\Delta_{i_a,2})+N\Delta_{ba}$ on $U[m] \times U_a\times U_b$, where
\begin{itemize}
\item $U[m]=\overset{m\  factors}{U \times \dots \times U}$
\item $K_a$ (respectively, $K_b$) is the canonical bundle of the factor $U_a$ (respectively, $U_b$),
\item $K_i$ is the canonical bundle of the $i^{th}$ factor in the product $U[m]$,
\item $\Delta_{i,a}$ is the partial diagonal $\{(x_a,\dots,x_m,z_a,z_b)\in \bbP[m] \times \bbP_a \times \bbP_b|x_i=z_a\}$; and $\Delta_{i,b}$  is defined similarly, 
\item $\Delta_{ba}$ is the partial diagonal $z_a=z_b$, and
\item by `invariant' we mean `invariant under the permutation of the first $m$ factors',
\item $N$ is a non-negative integer.
\end{itemize}

Given all this, the equation (\ref{normalorderedbsquared}) expresses two facts:  (i) $N$ can be replaced by zero  so that the section $\tv_m$ extends across the divisor $\Delta_{ba}$ (for all $m$) and (ii) the restriction of the extended section to $\Delta_{ba}$ is given by $m^{th}$ component of the right hand side of (\ref{normalorderedbsquared}). Informally, we write
\begin{equation*}
\lim_{z_b \to z_a} b(z_b)b(z_a)-\omega (z_a,z_b) = \bdots b(z_a)^2\bdots 
\end{equation*}
We will refer to statements such as these as operator product expansions (OPE's). These serve a dual purpose - the ``regular part'' defines new fields (eg., the energy momentum tensor, see below), and the ``singular part'' helps to compute commutation relations between operators defined \textit{via} contour integrals (residues) that we will encounter in subsequent sections.

\begin{remark} Recall that when a coordinate is chosen, the field $\hb$ is defined by $b(z)=\hb(u(z))du_z$. In terms of $\hb$, the OPE above becomes
\begin{equation*}
\lim_{z_b \to z_a} \hb(z_b)\hb(z_b)-\frac{1}{(u(z_a)-u(z_b))^2}=\bdots \hb(z_a)^2\bdots 
\end{equation*}
\end{remark}

\subsubsection{Locality} We have seen that the field $b$ has the following two properties: (a) its ``values" at  non-coinciding points can be composed and commute, and (b) as these points tend towards each other, the singularity is a pole of degree (2 in this case) independent of the vector on which it acts. We say that ``the field $b$ is local with respect to itself". 
More generally, we will say that two fields $\hA$ and $\hB$ are ``mutually local'' if 
\begin{itemize}
\item given $z_2 \ne z_1$, $\hB(z_2)$ and $\hA(z_1)$ can be composed and
$$
\hB (z_2) \circ \hA (z_1) = \hA (z_1) \circ \hB (z_2), \ \text{and}
$$
\item there exists $N \ge 0$ such that for any $\ttv \in \bW^{z_1}$
$$
\underset{z_2 \to z_1}\lim (u(z_2)-u(z_1))^N \hB (z_2) \circ \hA (z_1) [\ttv] =0
$$
for any coordinate $u$ regular at $z_1$.
\end{itemize}

\subsubsection{Energy-momentum tensor}

We set
\begin{equation*}
T(z) = \frac{\bdots b(z)^2\bdots }{2}
\end{equation*}
This is the ``energy-momentum operator". Let us examine its OPE with the field $b$.

Note that if $z_i, \ i=1,2,3$ are distinct points, the operators $b(z_i)$ mutually commute; so $b(z_3) \circ b(z_2)$ commutes with $b(z_1)$. Hence so does $b(z_3) \circ b(z_2) - \omega(z_3,z_2)$  since the second term is a scalar. Taking the limit $z_3 \to z_2$, we see that $T(z_2)$ commutes with $b(z_1)$.

We will now consider the product $T(z_2) \circ b(z_1)$ and consider two limits: (a) the limit as $z_2 \to z_1$, with $z_1$ fixed and (b) the limit as $z_1 \to z_2$, with $z_2$ fixed.

Consider the product $2b(z_2) \circ T(z_1)$. Expanding in terms of $i$ and $e$ and moving the $i$'s to the right, we get (dropping the $\circ$'s for notational simplicity) $e^2(z_2)e(z_1)+3e(z_2)e(z_1)i(z_2)+3e(z_2)i(z_2)i(z_1)+i^2(z_2)i(z_1)+2 \omega (z_1,z_2) b(z_2)$. From this we read off
\begin{proposition}
Let $z_2\ , z_1 \in \bbP$, and $\ttv \in \bW^{z_1}$; set $S(z_1)=e^3(z_1)+3e^2(z_1)i(z_1)+3e(z_1)i^2(z_1)+i^3(z_1)$ Then, as $z_2\to z_1$,
\begin{equation*}
[2b(z_2)T(z_1)-2 \omega (z_2,z_1) b(z_1)]\ttv \to S(z_1) \ttv
\end{equation*}
(This should be read as an element of $\bW \otimes K_{z_2} \otimes K^2_{z_1}$ tending to an element of $\bW \otimes K^3_{z_1}$.)
\end{proposition}

We now wish to reverse the roles of $T$ and $b$, and we have to deal with the limit, as $z_2\to z_1$, of $\omega(z_2,z_1)b(z_2)$. To this purpose we need to choose a point $P \ne z_1$, and an adapted coordinate $u$. Once this is done, we have

\begin{proposition}
Let $z_2\ , z_1 \in \bbP$,  $\ttv \in \bW^{z_1}$ and $S(z_1)$ be as above. Then, as $z_2\to z_1$,
\begin{equation*}
[2\hT(z_2)\hb(z_1)-2 \omega (z_2,z_1) b(z_1)-2\frac{1}{u(z_2)-u(z_1)}\hb'(z_1)]\ttv \to [\hS(z_1)+\hb''(z_1)]\ttv\\
\end{equation*}
where $\hb'$ and $\hb''$ are derivatives of the field $\hb$ with respect to the vector field $d/du$.
\end{proposition}

We see that the fields $T$ and $b$ not only commute at distinct points, but also that the product has singularities of order at most two (independent of the vector on which it acts) as the points coalesce. Thus they too, are ``mutually local". We define $\bdots \hb(z)\hT(z)\bdots =\hS(z)$ and $\bdots \hT(z)\hb(z)\bdots =\hS(z)+\hb''(z)$; these are both operators $\bW^z \to \bW$. We set $\bdots b(z)T(z)\bdots \equiv \bdots \hb(z)\hT(z)\bdots  du_z^3$ and $\bdots T(z)b(z)\bdots \equiv \bdots \hT(z)\hb(z)\bdots  du_z^3$; these are operators $\bW^z \to \bW \otimes K_z^3$, whose definition depends on the choice of $P$, \emph{but not the adapted coordinate} $u$.

We turn next to the product of $T$ with itself. An easy computation shows that $T(z_2) \circ T(z_1) - \frac{1}{2} \omega^2$ has a pole of order two as $z_2 \rightarrow z_1$. To go further, we need to choose as before a point $P$ and an adapted coordinate $u$. Once this is done, we compute:
\begin{equation*}
\begin{split}
4T(z_2) \circ T(z_1) &= e^2(z_2) T(z_1) + i^2(z_2) T(z_1) + 2e(z_2) i(z_2) T(z_1)\\
&=e^2(z_2) T(z_1)\\ 
&+ i^2(z_2)i^2(z_1) + i^2(z_2)e^2(z_1)+2i^2(z_2)e(z_1)i(z_1)\\ 
&+ 2e(z_2) i(z_2)i^2(z_1) + 2e(z_2) i(z_2)e^2(z_1) + 4e(z_2) i(z_2) e(z_1)i(z_1)\\
&=e^2(z_2) T(z_1)+i^2(z_2)i^2(z_1)+2e(z_2) i(z_2)i^2(z_1)\\
&+2i^2(z_2)e(z_1)i(z_1)+2e(z_2) i(z_2)e^2(z_1) + 4e(z_2) i(z_2) e(z_1)i(z_1)\\
&+ i^2(z_2)e^2(z_1)\\
&=e^2(z_2) T(z_1)+i^2(z_2)i^2(z_1)+2e(z_2)i(z_2)i^2(z_1)\\
&+2e(z_1)i^2(z_2)i(z_1)+2e(z_2)e^2(z_1)i(z_2)+4e(z_2)e(z_1)i(z_2) i(z_1)\\
&+4i(z_2)i(z_1) \omega(z_2,z_1)+4e(z_2)e(z_1) \omega(z_2,z_1)  + 4e(z_2)i(z_1) \omega(z_2,z_1) \\
&+ i(z_2)e^2(z_1)i(z_2)+2i(z_2)e(z_1)\omega(z_2,z_1)\\
&=e^2(z_2) T(z_1)+i^2(z_2)i^2(z_1)+ 2e(z_2) i(z_2)i^2(z_1)\\
&+2e(z_1)i^2(z_2)i(z_1)+2e(z_2)e^2(z_1)i(z_2)+4e(z_2)e(z_1)i(z_2) i(z_1)+e^2(z_1)i^2(z_2)\\
&+4i(z_2)i(z_1) \omega(z_2,z_1)+4e(z_2)e(z_1) \omega(z_2,z_1)  + 4e(z_2)i(z_1) \omega(z_2,z_1) \\
&+2e(z_1)i(z_2)\omega(z_2,z_1)
+2e(z_1)i(z_2)\omega(z_2,z_1)+2\omega^2(z_2,z_1)\\
\end{split}
\end{equation*}
which shows
\begin{equation*}
\begin{split}
\hT(z_2) \circ \hT(z_1)&-\frac{1}{2(u(z_2)-u(z_1))^4}-2\hT(z_1) \frac{1}{(u(z_2)-u(z_1))^2}-\hT'(z_1)\frac{1}{u(z_2)-u(z_1)}-\bdots T(z)^2
\underset{z_2 \to z_1}\to 0 \\
\end{split}
\end{equation*}
where we set $\bdots T(z)^2 \bdots =\frac{1}{4}\{
\he^4(z) + 4\he^3(z)\hi(z)+6\he^2(z)\hi^2(z) + 4\he(z) \hi^3(z)+\hi^4(z)\}du_z^4$.

\subsubsection{Renormalised product}\label{renorm}

We encountered above ``renormalised'' products of fields. We now give a general prescription for defining a renormalised product of two mutually local fields as above. Fix a point $P$, and choose a coordinate $u$ with a pole at $P$. The renormalised products $\bdots \hB (z_1) \hA (z_1) \bdots^{[n]}$ are defined inductively by requiring that for each $\ttv \in \bW^{z_1}$,
\begin{equation*}
\begin{split}
\underset{z_2 \to z_1}\lim \{\hB (z_2) \circ \hA (z_1) &- \frac{1}{(u(z_2)-u(z_1))^N}\bdots \hB (z_1) \hA (z_1) \bdots^{[N]}\\ &-\frac{1}{(u(z_2)-u(z_1))^{N-1} }\bdots \hB (z_1) \hA (z_1) \bdots^{[N-1]}\\ &\dots\\ 
&- \bdots \hB (z_1) \hA (z_1) \bdots \}[\ttv] =0
\end{split}
\end{equation*}
If $A(z)=\hA(z)du_z^{c_a}$ and $B(z)=\hB(z)du_z^{c_b}$ are ``tensorial fields'' (whose definition depends only on the choice of $P$), this will be true also of the renormalised products $\bdots \hB (z_1) \hA (z_1) \bdots^{[n]} du_{z_1}^{c_a+c_b+n}$.

Suppose given a field $\hC$ that is mutually local with respect to $\hA$ and $\hB$. Then $\hC$ is mutually local with respect to the renormalised products as well. This is a basic result in the theory of vertex algebras (Dong's Lemma).  Since we do not give general definitions we point to specific cases as they occur.

\subsection{Mode expansions}\label{modes}

Suppose a coordinate $u$ is chosen, vanishing at a point which we denote $O$, and with a pole at $P$.  

\subsubsection{Boson modes}\label{bosonmodes}

Let $H^0(\bbP \setminus O, K)$ denote the space of \emph{algebraic} sections of the canonical bundle $K$ on $\bbP \setminus P$ and let $\lO \bW$ denote the symmetric algebra over $H^0(\bbP \setminus O,K)$:
\begin{equation*}
\lO \bW=\bbC \oplus H^0(\bbP \setminus O, K) \oplus S^2 H^0(\bbP \setminus O, K) \oplus ......
\end{equation*}
(Note that a meromorphic form regular outside one point must have vanishing residue there and is therefore of the second kind.)

For $l \ge 1$, define operators $b_{-l}:  \bW \to  \bW$ by
\begin{equation*}
b_{-l}[\chi] = \chi \otimes^s d{u^{-l}}   = -\chi \otimes^s lu^{-l-1}du
\end{equation*}
(Note that $b_{-l}$ leaves the flag $\lO \bW \subset \bW_P \subset \bW$ invariant.) For $l \ge 1$, define operators $b_{l}: \bK_P \to \bbC$ by setting
\begin{equation*}
b_{l} [u^{-m-1}du ] =  -\delta_{l,m}, \ m \ge 1
\end{equation*}
and extending by linearity to  $\bK_P$.  Extend as a derivation $b_l:\bW_P \to \bW_P$ by setting $b_l ({\mathbf 1})=0$ and
\begin{equation*}
b_l (\alpha_1 \otimes^s \dots \otimes^s \alpha_p) =  \sum_i ({\alpha}_1 \otimes^s \dots \widehat{{{\alpha}_i}} \dots \otimes^s {\alpha}_p) b_l(\alpha_i)
\end{equation*}
The operators $b_l, \ l \ne 0$ are all defined on $\bW_P$ and satisfy the commutation relations $[b_l,b_m]=l\delta_{l+m,0}$. 

Let $z \in \bbP \setminus \{O,P\}$ and set $u(z)=w \in \bbC$. The series $\sum_{l \ge 1} b_{-l} w^{l-1}[\mathbf{1}]$ converges uniformly on the subsets $\{u||u|\ge |w|+\delta\}$ (for $\delta >0$) to $-\frac{du}{(u-w)^2} \in \bK_P$. Formally,
\begin{equation*}
\sum_{l \ge 1} b_{-l} w^{l-1}  = \hat{e} (z)
\end{equation*}
On the other hand,  $\sum_{l \ge 1} b_{l} w^{-l-1}[u^{-m-1}du]=-w^{-m-1}$, so that if $\alpha = a du\in  H^0(K, \bbP \setminus O)$,
\begin{equation*}
\sum_{l \ge 1} b_{l} w^{-l-1} [\alpha] = -a(z)
\end{equation*}
(This is actually a finite sum for any given $\alpha$.)
Thus
\begin{equation}
\sum_{l \ge 1} b_l w^{-l-1}= \hi(z)|_{{}_{\lO \bW\subset \bW^z}} : \lO \bW \to \lO \bW
\end{equation}
We have used the notation $i(z)=\hi(z) du_z, \  e(z)=\he(z) du_z$.

For ease of reference, we summarise the definition of the operators $b_l$:
\begin{equation*}
b_l=
\begin{cases}
\begin{aligned}
& du^l \otimes^s \ \ \ \ \ \ \ \ \ l \le -1\\
& du^m \mapsto l \delta_{l+m} \ l \ge 1,\ \text{extended as a derivation}\\
\end{aligned}
\end{cases}
\end{equation*}

\subsubsection{Boson current algebra modes}\label{bosoncurrentmodes}

Let $\lO \mathbb{W}_-=\bbC \oplus H^0(\bbP \setminus O, \mathfrak{m}_P) \oplus S^2 H^0(\bbP \setminus O, \mathfrak{m}_P) \oplus ......\ $, where $H^0(\bbP \setminus O, \mathfrak{m}_P)$ is the ideal of \emph{algebraic} functions on $\bbP \setminus O$ vanishing at $P$.

For $l \ge 1$, define operators $j_{-l}:  \mathbb{W}_- \to  \mathbb{W}_-$ by
\begin{equation*}
j_{-l}[\boldsymbol{\chi}] =  u^{-l} \otimes^s \boldsymbol{\chi}
\end{equation*}
Note that $j_l$ preserves $\lO \mathbb{W}_- \subset \mathbb{W}_-$.  For $l \ge 1$, define operators $j_{l}: \mathfrak{m}_P \to \bbC$  by setting
\begin{equation*}
j_{l} [u^{-m}] =  l\delta_{l,m}, \ m \ge 1
\end{equation*}
and extending by linearity to  $\mathfrak{m}_P$; next extend as a derivation
$j_l:\mathbb{W}_- \to \mathbb{W}_-$ by setting $j_l(\mathbf{1})=0$ and
\begin{equation*}
j_l (\balpha_1 \otimes^s \dots \otimes^s \balpha_p) =  \sum_i (\balpha_1 \otimes^s \dots \widehat{{\balpha_i}} \dots \otimes^s \balpha_p) \   j_l (\alpha_i)
\end{equation*}
The operators $j_l,\ l \ne 0$ are all defined on $\mathbb{W}_-$ and satisfy the commutation relations $[j_l,j_m]=l\delta_{l+m,0}$.

Let $z \in \bbP\setminus \{O,P\}$ and set $u(z)=w \in \bbC$.  
The series $\sum_{l \ge 1} j_{-l} w^{l-1}[\mathbf{1}]$ converges uniformly on the subsets $\{u||u|\ge |w|+\delta\}$ (for $\delta >0$) to $\frac{1}{(u-w)} \in \mathbb{W}_-$. Formally
\begin{equation*}
\sum_{l \ge 1} j_{-l} w^{l-1}   =  \hepsilon(z)
\end{equation*}
On the other hand, $\sum_{l \ge 1} j_{l} w^{-l-1}[u^{-m}]=lw^{-m-1}$, so that for $\balpha \in H^0(\bbP \setminus O, \mathfrak{m}_P)$, we have
\begin{equation*}
\sum_{l \ge 1} j_{l} w^{-l-1} [\balpha] = -d\balpha(w)/dw
\end{equation*}
(The sum is finite for any given $\balpha$.) Thus
\begin{equation}
\sum_{l \ge 1} j_{l} w^{-l-1}  = \hat{\iota}(z)|_{{}_{\lO \mathbb{W}_-\subset \mathbb{W}^z_-}} : \lO \mathbb{W}_- \to \lO \mathbb{W}_-
\end{equation}
We have used the notation $\iota(z)=\hiota(z) du_z, \  \epsilon(z)=\hepsilon(z) du_z$.

For ease of reference, we summarise the definition of the operators $j_l$:
\begin{equation*}
j_l=
\begin{cases}
\begin{aligned}
& u^l \otimes^s \ \ \ \ \ \ \ \ \ l \le -1\\
& u^m \mapsto l \delta_{l+m} \ l \ge 1,\ \text{extended as a derivation}\\
\end{aligned}
\end{cases}
\end{equation*}

\begin{remark} One might ask about a ``zero-mode", in other words, an operator $b_0$ or $j_0$ that would be central in the algebra generated by the other modes.   The natural way to introduce such an operator is {\it via} insertions, and the current algebra {\it avatar} of the boson is best suited to this. See \S \ref{bosoninsertions}.
\end{remark}

\section{\textbf{The boson: the states and symmetries}}

We now introduce subspaces of $\bW$ associated to subsets of $\bbP$, and set up notation, some of which has been anticipated above.

There are three main reasons to consider such subspaces:
\begin{itemize}
\item There exist certain natural pairings (and eventually, inner products). Going modulo the null-spaces, we get quotient modules over the (vertex) algebra of fields. The $n$-point functions determine the pairings and can in turn be expressed in terms of pairing of appropriate vectors. 
\item There exist links with the formalism of quantum field theory wherein which vectors and operators are reconstructed from Euclidean $n$-point (= Schwinger) functions.
\item Operators formally defined as contour integrals of fields act on these spaces in a way coherent with respect to the pairings. In particular, we obtain representations of Heisenberg, Virasoro, and eventually, affine lie algebras.
\end{itemize}

\subsection{Contour integrals and projections}

We record a preliminary result. Suppose given a simply connected domain $D \subset \bbP$, with smooth oriented boundary $\partial D=\gamma$, with $\gamma$ winding counterclockwise around $D$. Let $D'$ denote the complementary disc. Let $\bK_\gamma$ denote the space of meromorphic forms of the second kind regular along $\gamma$. (We are recapitulating part of \S \ref{Notation}.) It is elementary that 
$$
\bK_\gamma = \lD\bK \oplus \lDp \bK
$$
Let $\alpha \mapsto \lD[\alpha]$  denote the projection to the first factor. We have

\begin{lemma} \label{contourprojection} Write $\alpha=d\balpha$, with $\balpha$
a meromorphic function. We have, for $\hz \in D'$,
$$
\lD[\alpha] (\hz) = -\frac{du_\hz}{2\pi i}\int_{\gamma} \frac{\balpha(z)}{(u(\hz)-u(z))^2} du_z
$$
\end{lemma}

This is a restatement of the next lemma, whose proof is a direct computation.
 
\begin{lemma} Let $\balpha$ be a meromorphic function regular along $\gamma$. Consider the contour integral
\begin{equation}\label{hphi}
\ha(\hz)   \equiv -\frac{1}{2\pi i}\int_{\gamma} \frac{\balpha(z)}{(u(\hz)-u(z))^2} du_z
\end{equation}
with $\hz \in D'$ (respectively, $\hz \in D$). Then $\hz \mapsto \ha(\hz)du_\hz$ is the restriction to $D'$ (respectively, $D$) of a meromorphic form (\textup{which we denote $\halpha$}). In fact, 
\begin{itemize}
\item for $\hz \in D'$: if $\balpha$ is regular in $D'$, $\halpha=d\balpha$, and if $\balpha$ is regular in $D$,  then $\halpha=0$.
\item  for $\hz \in D$: if $\balpha$ is regular in $D'$, then $\halpha=0$, and if $\balpha$ is regular in $D$ then $\halpha=-d\balpha$.
\end{itemize}
\end{lemma}

\subsection{The Heisenberg algebra}

The operator $b(z)$ maps $\bW^z$ to $\bW\otimes K_z$. Given a meromorphic function $\phi$ (``test function") on $\bbP$ and a contour  $\gamma \subset \bbP$ which avoids the poles of $\phi$, we would like to integrate $\phi(z)b(z)$ along $\gamma$ to get an operator $H^\phi_\gamma$, which we mostly denote $\Phi_\gamma$:
\begin{equation}\label{contour1}
H^\phi_\gamma \equiv \Phi_\gamma =`` \frac{1}{2\pi i} \int_{\gamma} \phi(z) b(z)''
\end{equation}
The integral has to be given a meaning. We do this in two steps.

Let $\alpha = a(z) du \in \bK_\gamma$ and consider the commutator $[\Phi_\gamma, \alpha]$. Working formally, we have
\begin{equation*}
\begin{split}
[\Phi_\gamma, \alpha] &= \frac{1}{2\pi i} \int_{\gamma} \phi(z) [b(z),\alpha]\\
&=\frac{1}{2\pi i} \int_{\gamma} \phi(z) [i(z)+e(z),\alpha]\\
&=\frac{1}{2\pi i} \int_{\gamma} \phi(z) [i(z),\alpha]\\
&=-\frac{1}{2\pi i} \int_{\gamma} \phi(z) \alpha(z)\\
\end{split}
\end{equation*}
Next consider $\Phi_\gamma \mathbf{1}$. Again working formally,
\begin{equation*}
\begin{split}
\Phi_\gamma\mathbf{1} &= \frac{1}{2\pi i} \int_{\gamma} \phi(z) b(z) \mathbf{1}\\
&=\frac{1}{2\pi i}\int_{\gamma} \phi(z) (i(z)+e(z))\mathbf{1}\\
&=\frac{1}{2\pi i}\int_{\gamma} \phi(z) e(z)\mathbf{1}\\
&=-\frac{1}{2\pi i}\int_{\gamma} \frac{\phi(z)du du_z}{(u-u(z))^2}\\
\end{split}
\end{equation*}
This is the integral of a one-form (in the $z$ variable) with values in the infinite-dimensional vector space $\bK$:
$$
z \mapsto -\frac{1}{2\pi i}  \frac{\phi(z) du}{(u-u(z))^2} du_z
$$
and a meaning needs to be assigned to its integral. 

\subsubsection{Definitions of $\lD\Phi$, $\Phi_P$ and $\lO\Phi$}\label{DefsHeis}  

In view of the Lemma \ref{contourprojection} it is natural to \emph{define} the operator $\lD H^\phi \equiv \lD \Phi:\bW_{\partial D} \to \bW_{\partial D}$ (representing the formal integral along $\gamma$):

\begin{definition} For $\phi$ regular along the oriented boundary $\partial D$ of a disc $D$, we set
\begin{equation*}
\begin{split}
\lD \Phi [\alpha_1 \otimes^s \dots \otimes \alpha_p]&=- \sum_i  ({\alpha}_1 \otimes^s \dots \widehat{{{\alpha}_i}} \dots \otimes^s {\alpha}_p) \times \frac{1}{2\pi i}  \int_{\partial D} \phi \alpha_i\\ 
&+
\begin{cases}
d\phi \otimes^s \alpha_1 \otimes^s \dots \otimes^s \alpha_p \ \ \text{if $\phi$ is regular in $D'$, and}\\
0 \ \ \text{if $\phi$ is regular in $D$.}
\end{cases}
\end{split}
\end{equation*}
\end{definition}

\begin{definition} Since $\lD \Phi$ does not create singularities outside $D$, it induces an operator, $\lD H^\phi \equiv \lD\Phi:\lD\bW \to \lD\bW$ for which we retain the same notation.
\end{definition}

Let us now define algebro-geometric versions of the operators $\Phi$. 

\begin{definition} Let a point $P \in \bbP$ be given.  We will yet again use the fact that the space of meromorphic functions on $\bbP$ is the vector space direct sum $\cO_P \oplus \bbC[\bbP \setminus P]$, modulo constants.  Regarding $\bbP \setminus P$ as the union of discs $D$ that exhaust it, we \emph{define}, for an \emph{arbitrary meromorphic function} $\phi$, the operator $H^\phi_P \equiv \Phi_P:\bW \to \bW$ (representing an integral over a contour winding  \emph{clockwise} around $P$,  and ``close enough to $P$''):
\begin{equation*}
\begin{split}
\Phi_P [\alpha_1 \otimes^s \dots \otimes \alpha_p]&=\sum_i  ({\alpha}_1 \otimes^s \dots \widehat{{{\alpha}_i}} \dots \otimes^s {\alpha}_p) \times Res_P(\phi \alpha_i)\\
&+
\begin{cases}
d\phi \otimes^s \alpha_1 \otimes^s \dots \otimes^s \alpha_p \ \ \text{if $\phi$ is regular at $P$, and}\\
0 \ \ \text{if $\phi$ is regular away from $P$.}
\end{cases}
\end{split}
\end{equation*}
where $Res_O\P$ denotes the residue at $P$ of a meromorphic form. The minus sign has disappeared because the residue around a point $P$ is defined as an integral around a curve winding \emph{counterclockwise} around it, which is the orientation opposite to that of $\partial D$ for any of the discs $D$.
\end{definition}

\begin{definition} Since $\Phi_P$ does not create singularities at $P$, it induces an operator, $H_P^\phi \equiv \Phi_P:\bW_P \to \bW_P$ for which we retain the same notation.
\end{definition}

\begin{definition} Let a point $O$ be given, and consider the intersection of discs $D$ containing it.  (In this limit the $\lD\bW$ approaches $\lO\bW$.) We \emph{define} $\lO H^\phi \equiv \lO\Phi$ acting on $\bW$:
\begin{equation*}
\begin{split}
\lO\Phi [\alpha_1 \otimes^s \dots \otimes \alpha_p]&=-\sum_i  ({\alpha}_1 \otimes^s \dots \widehat{{{\alpha}_i}} \dots \otimes^s {\alpha}_p) \times Res_O(\phi \alpha_i)\\
&+
\begin{cases}
d\phi \otimes^s \alpha_1 \otimes^s \dots \otimes^s \alpha_p \ \ \text{if $\phi$ is regular away from $O$, and}\\
0 \ \ \text{if $\phi$ is regular at $O$.}
\end{cases}
\end{split}
\end{equation*}
This represents an integral over a contour winding counterclockwise around $O$ and close enough to $O$ so that the only poles of $\phi$ within the enclosed disc $D$ are at $O$.
\end{definition}

\begin{definition} Since $\lO\Phi$ does not create singularities away from $O$, it induces an operator, $\lO H^\phi \equiv \lO \Phi:\lO\bW \to \lO\bW$ for which we retain the same notation.
\end{definition}

In the notation of \S \ref{bosonmodes}, we have, for $l$ a positive integer:
\begin{equation}\label{Heisenbergbosonmodes}
\begin{split}
\lO H^{u^{-l}}&=b_{-l}\\
\lO H^{u^l}&= b_l
\end{split}
\end{equation}

\subsubsection{A formal computation} 

Given two functions $\phi$ and $\psi$, let us formally compute the commutator of $\lO H^\phi \equiv \lO\Phi$ and $\lO H^\psi \equiv \lO\Psi$; we use a contour integral argument that follows the traditional method in conformal field theory and can now be justified with very little work.
For the moment, we drop the suffix and set $\lO \Phi=\Phi$, etc.   Choose a coordinate $u$ so that $u(O)=0$ and write $b(z)=\hb(u(z))du_z$, etc.  .Then
\begin{equation*}
\Phi \circ \Psi [\ttv] = \frac{1}{2\pi i} \int_{\gamma_2}  du_2
 \frac{1}{2\pi i} \int_{\gamma_1} du_1 \ \phi(u_2)\psi(u_1) \hb(u_2)\hb(u_1)[\ttv]
\end{equation*}
where \emph{the contour $\gamma_1$ is  closer to $O$ than $\gamma_2$}, and both contours are chosen so that all singularities of $\phi, \psi$ and $\ttv$ are excluded except possibly for poles at $O$ itself. (To keep notation simple, we have set $u(z_1)=u_1$ and $u(z_2)=u_2$.)  Then
\begin{equation*}
\begin{split}
(\Phi \circ \Psi - \Psi \circ \Phi) [\ttv] &= \frac{1}{2\pi i} \int_{\gamma_2}  du_2
 \frac{1}{2\pi i} \int_{\gamma_1} du_1 (\phi(u_2)\psi(u_1)-\psi(u_2)\phi(u_1)) \hb(u_2)\hb(u_1)[\ttv]\\
 &= \frac{1}{2\pi i} \int_{\gamma_2}  du_2
 \frac{1}{2\pi i} \int_{\gamma_1} du_1 (\phi(u_2)\psi(u_1)-\psi(u_2)\phi(u_1))\\
 & \ \ \ \ \ \ \ \times (\hb(u_2)\hb(u_1)-\frac{1}{(u_2-u_1)^2})[\ttv]\\
 &\ \  + \frac{1}{2\pi i} \int_{\gamma_2}  du_2
 \frac{1}{2\pi i} \int_{\gamma_1} du_1  \frac{\phi(u_2)\psi(u_1)-\psi(u_2)\phi(u_1)}{(u_2-u_1)^2} [\ttv]\\
\end{split}
\end{equation*}
For $z_1,\ z_2$ in a small $u$-polydisc around $(O,O)$, let $\tilde{\ttv}(u_1,u_2)=(\phi(u_2)\psi(u_1)-\psi(u_2)\phi(u_1)) (\hb(u_2)\hb(u_1)-\frac{1}{(u_2-u_1)^2})[\ttv] $; this is well-defined and holomorphic, including when $u_1=u_2$, and antisymmetric in the two variables.
The first term in the last equation then becomes:
\begin{equation*}
  \frac{1}{2\pi i} \int_{\gamma_2}  du_2
 \frac{1}{2\pi i} \int_{\gamma_1} du_1   \tilde{\ttv}(u_1,u_2)
  \end{equation*}
with $\tilde{\ttv}$ an antisymmetric holomorphic (vector-valued) function in the two variables, defined in the product of an annulus with itself. The double contour integral is then easily seen to vanish. As for the second term, we rewrite it as:
\begin{equation*}
\begin{split}
\frac{1}{2\pi i} \underset{|u_2|=2r}{\int}  \phi(u_2) du_2 (-\frac{1}{2\pi i} \underset{|u_1|=3r}{\int} du_1 \frac{\psi(u_1)}{(u_2-u_1)^2}+\frac{1}{2\pi i} \underset{|u_1|=r}{\int} du_1
 \frac{\psi(u_1)}{(u_2-u_1)^2}) [\ttv]
\end{split}
\end{equation*}
for any real positive $r$ that is chosen small enough.  

Finally we get (for $r$ small enough):
\begin{equation*}
(\Phi \circ \Psi - \Psi \circ \Phi) [\ttv] =  -[\frac{1}{2\pi i} \underset{|u|=2r}\int du \phi(u)\psi'(u) \ttv = -Res_O(\phi d\psi) \ttv \ .
\end{equation*}

\subsubsection{The Heisenberg extension}

In fact, the following can be checked rigorously from the definitions in \S \ref{DefsHeis}:

\begin{proposition}
Let $\phi$, $\psi$ be meromorphic functions. Then
\begin{enumerate}
\item $\lD\Phi \circ \lD\Psi - \lD\Psi \circ \lD\Phi = -\frac{1}{2\pi i}\int_{\partial D} \phi d\psi$ if $\phi$ and $\psi$ are regular along $\partial D$.
\item $\Phi_P \circ \Psi_P - \Psi_P \circ \Phi_P = Res_P(\phi d\psi)$
\item $\lO\Phi \circ \lO\Psi - \lO\Psi \circ \lO\Phi = - Res_O(\phi d\psi)$
\end{enumerate}
\end{proposition}

We denote by $\lO\cH$ the central extension of Lie algebras:
$$
0 \to \bbC \to \lO\cH \to \cK \to 0
$$
where $\lO\cH=\cK \oplus \bbC$ and the Lie bracket is
$$
[(\phi,s),(\psi,t)]=-Res_O(\phi d\psi)
$$
Then $\phi \mapsto \lO\Phi: \lO\bW \to \lO\bW$ is a representation of $\lO\cH$. Note that constant functions map to zero under $\phi \to \lO\Phi$. This will not be the case with an insertion at $O$, see \S \ref{bosoninsertions} below.

Note also the following
\begin{proposition} Let $O,O'$ be distinct points, and $\phi,\phi'$ two meromorphic functions. Then
$$
[\lO\Phi, \lOp\Phi']=0
$$
on $\bW$.
\end{proposition}

\begin{proof} Note first that $[[\lO\Phi, \lOp\Phi'],\alpha]=0$. So we need to check that
$[\lO\Phi, \lOp\Phi']\mathbf{1}=0$. This is a case-by-case argument. If $\phi$ is regular at $O$ and $\phi'$ at $O'$, this is easy to check. If $\phi$ is regular at $O$ and $\phi'$ away from $O'$, we have
$$
[\lO\Phi, \lOp\Phi']\mathbf{1}= \lO\Phi \circ \lOp\Phi'\mathbf{1}=-Res_O(\phi d\phi') \mathbf{1}=0
$$
and finally, if $\phi$ is regular away from $O$ and $\phi'$ away from $O'$
\begin{equation*}
\begin{split}
[\lO\Phi, \lOp\Phi']\mathbf{1}&=\lO\Phi (d\phi')-\lOp \Phi' (d\phi)\\
&=-Res_O(\phi d \phi')+Res_{O'}(\phi' d \phi)\\
&=Res_O(\phi' d \phi)+Res_{O'}(\phi' d \phi)\\
&=0
\end{split}
\end{equation*}
\end{proof}

\subsection{Action of meromorphic vector fields}

The operator $T(z)$ maps $\bW^z$ to $\bW \otimes K_z^2$. Given a meromorphic vector field $X$ on $\bbP$, pairing with $T(z)$ yields a meromorphic operator-valued one-form. Let $\gamma$ be a contour which avoids the poles of $X$. We wish to define
\begin{equation*}
\bL^X_\gamma \equiv \cX_\gamma = ``-\frac{1}{2\pi i} \int_{\gamma}  X(z) T(z) ''
\end{equation*}
(The minus sign is introduced to some purpose.) As in the case of the Heisenberg algebra, this formal expression can be used as a guide to define operators $\lD\cX$, $\cX_P$ and $\lP\cX$. For the first of these we would require $X$ to be regular along $\gamma$, the other two are defined with the contour ``close enough to''  $P$, and there are no restrictions on $X$. 

To reach a definition of $\cX_\gamma$, let us imitate the strategy followed earlier with the Heisenberg algebra.  Let $\alpha = a(z) du \in \bK$, and consider the commutator $[\cX_\gamma, \alpha]$. Working formally, we have
\begin{equation*}
\begin{split}
[\cX_\gamma, \alpha] &= -\frac{1}{2\pi i} \int_{\gamma} X(z) [b(z),\alpha]\\
&=-\frac{1}{4\pi i} \int_{\gamma} X(z) [i(z)^2+e(z)+2e(z)i(z),\alpha]\\
&=-\frac{1}{2\pi i} \int_{\gamma} X (z) [i(z),\alpha] (i(z)+e(z))\\
&=\frac{1}{2\pi i} \int_{\gamma} X(z) \alpha (z) b(z)\\
&=H^{i_X\alpha}_\gamma
\end{split}
\end{equation*}

Next consider $\cX_\gamma \mathbf{1}$. Write $X=\xi(u) d/du$. Again working formally, 
\begin{equation*}
\label{Iplus}
\cX_\gamma (\mathbf{1})=-\frac{1}{4\pi i} \int_\gamma {\xi({z})} \frac{{du_{\hz_1}}}{(u(\hz_1)-u(z))^2} \frac{{du_{\hz_2}}}{(u(\hz_2)-u(z))^2}  du_z\\
\end{equation*}
Evaluating this, with $\hz_1,\hz_2, \in D'$ we find
\begin{equation*}
\cX_\gamma (\mathbf{1})=
\begin{cases}
\begin{aligned}
& \{\frac{2(\xi(u_1)-\xi(u_2))-(\xi'(u_1)+\xi'(u_2))(u_1-u_2)}{2(u_1-u_2)}\} \frac{du_1 \otimes^s du_2}{(u_1-u_2)^2}\\
&\ \ \ \ \text{if $X$ is regular in $D'$, and}\\
&0 \ \ \ \text{if $X$ is regular in $D$.}
\end{aligned}\\
\end{cases}
\end{equation*}
This is consistent because the first expression vanishes for vector fields that are regular everywhere on $\bbP$.

\begin{remark}
For any meromorphic vector field $X=\xi(u)d/du$, let 
\begin{equation*}
\omega_X = \{\frac{2(\xi(u_1)-\xi(u_2))-(\xi'(u_1)+\xi'(u_2))(u_1-u_2)}{2(u_1-u_2)}\}\frac{du_1 \otimes^s du_2}{(u_1-u_2)^2}\\
\end{equation*}  
It is easy to check that in invariant terms 
\begin{equation}\label{omegax}
\omega_X= -\frac{1}{2}\cL_{X_1+X_2}\omega
\end{equation}
where $\omega$ is the invariant bidifferential form defined in (\ref{invbidiff}).
This last formula needs explanation. By $X_1$ we mean the vector field $X$ lifted to $\bbP \times \bbP$ from the first factor, $X_2$ is similarly defined; $\cL_{X_1+X_2}$ is the Lie derivative w.r.to the sum of the two vector fields; acting on the meromorphic symmetric  bi-differential $\omega$, it yields a another such bi-differential $\omega_X$. As a consequence of the formula (\ref{omegax}), there is an invariant differential operator
\begin{equation*}
\begin{split}
\text{{\small meromorphic vector fields on $\bbP$}} &\to \text{{\small symmetric meromorphic functions on $\bbP \times \bbP$}}\\
\xi(u)d/du &\mapsto \{\frac{2(\xi(u_1)-\xi(u_2))-(\xi'(u_1)+\xi'(u_2))(u_1-u_2)}{2(u_1-u_2)}\}
\end{split}
\end{equation*}
This function can be checked to vanish along the diagonal. Since it is clearly even, it must vanish to order two. Note also that (on any curve) the Lie derivative of \emph{any} meromorphic differential with respect to a meromorphic vector field is a form of the second kind. Thus, rather remarkably, \emph{$\omega_X$, which \textup{a priori} is just a symmetric meromorphic bidifferential on $\bbP \times \bbP$, actually belongs to $S^2 \bK$}. (I thank E. Loojenga for this argument.) If $X$ is regular at $P$, the element $\omega_X$ belongs to $S^2 \bK_P$.  
\end{remark}

\subsubsection{Definitions of $\bL^X_P$ and $\lO \bL^X$: the Virasoro algebra}\label{DefsVir}

We can now proceed to define $\bL^X_P \equiv \cX_P$ and $\lO \bL^X \equiv \lO\cX$. (We will forgo writing out the ``analytic'' case, namely, the definition of $\lD \bL^X \equiv \lD\cX$.)   It helps to organise $\cX_\gamma$ into its parts of ``degrees" 0, 1 and $-1$, setting $\cX_\gamma^0 = -\frac{1}{2\pi i} \int_\gamma X (z) e(z) \circ i(z)$, $\cX_\gamma^+ = -\frac{1}{2\pi i} \int_\gamma X (z) \frac{1}{2} e(z) \circ e(z)$,  and $\cX_\gamma^- = -\frac{1}{2\pi i} \int_\gamma X (z) \frac{1}{2} i(z)\circ i(z)$.

\begin{definition}  We define $\cX^-_P + \cX^0_P + \cX^+_P=\cX_P:\bW_P \to \bW_P$, where
\begin{equation*}
\cX^-_P (\alpha_1 \otimes^s \dots \otimes^s \alpha_p)= \sum_{i<j}(\alpha_1 \otimes^s \dots \widehat{\alpha_i}\dots \widehat{\alpha_j} \dots \otimes^s \alpha_p) R_P(X;\alpha_i \otimes^s \alpha_j)
\end{equation*}
with the hats marking terms to be omitted,
\begin{equation*}
\cX^0_P (\alpha_1 \otimes^s \dots \otimes^s \alpha_p)= \sum_i \alpha_1 \otimes^s \dots \otimes^s [\cL_X \alpha_i]_P \otimes^s \dots \otimes^s \alpha_p
\end{equation*}
and 
\begin{equation*}
\cX^+_P (\alpha_1 \otimes^s \dots \otimes^s \alpha_p)=  \cX^+_P(\mathbf{1}) \otimes^s \alpha_1 \otimes^s \dots   \otimes^s \alpha_p
\end{equation*}
where we set:
\begin{enumerate}
\item $R_P(X;\alpha_i \otimes^s \alpha_j)=Res_P([i_X\alpha_i]\alpha_j) \ (=Res_P([i_X\alpha_j]\alpha_i)$
\item for any meromorphic form $\eta$ of the second kind
\begin{equation*}
[\eta]_P=
\begin{cases}
0 \ \ \ \ \text{if $\eta$ is regular away from $P$, and}\\
\eta \ \ \ \ \text{if $\eta$ is regular at $P$}\\
\end{cases}
\end{equation*}
\item and
\begin{equation*}
\cX_P^+(\mathbf{1})=
\begin{cases}
0 \ \ \ \ \text{if $X$ is regular away from $P$, and}\\
\omega_X \ \ \ \ \ \ \ \ \ \text{if $X$ is regular at $P$.}
\end{cases}
\end{equation*}
\end{enumerate}
\end{definition}
Note that if $X$ is a regular vector field, it acts on $\bW_P$ in the obvious way:
\begin{equation*}
\cX_P (\alpha_1 \otimes^s \dots \otimes^s \alpha_p)= \sum_i \alpha_1 \otimes^s \dots \otimes^s \cL_X \alpha_i \otimes^s \dots \otimes^s \alpha_p
\end{equation*}

We now turn to $\lO \cX:\lO \bW \to \lO \bW$. Just as above, we write $\lO \cX=\lO \cX^- +\lO \cX^0+\lO \cX^+$.  

\begin{definition}\label{deflOcX}  We have $\lO \cX=\lO \cX^- + \lO \cX^0 + \lO \cX^+$, where
\begin{equation*}
\lO \cX^- (\alpha_1 \otimes^s \dots \otimes^s \alpha_p)= -\sum_{i<j}(\alpha_1 \otimes^s \dots \widehat{\alpha_i}\dots \widehat{\alpha_j} \dots \otimes^s \alpha_p) R_O(X;\alpha_i \otimes^s \alpha_j)
\end{equation*}
with the hats marking terms to be omitted,
\begin{equation*}
\lO \cX^0 (\alpha_1 \otimes^s \dots \otimes^s \alpha_p)= \sum_i \alpha_1 \otimes^s \dots \otimes^s \lO [\cL_X \alpha_i] \otimes^s \dots \otimes^s \alpha_p
\end{equation*}
and 
\begin{equation*}
\lO \cX^+ (\alpha_1 \otimes^s \dots \otimes^s \alpha_p)=  \lO \cX^+(\mathbf{1}) \otimes^s \alpha_1 \otimes^s \dots   \otimes^s \alpha_p
\end{equation*}
where we set:
\begin{enumerate}
\item for any meromorphic form $\eta$ of the second kind
\begin{equation*}
\lO [\eta]=
\begin{cases}
\eta \ \ \ \ \text{if $\eta$ is regular away from $O$, and}\\
0 \ \ \ \ \text{if $\eta$ is regular at $O$}\\
\end{cases}
\end{equation*}
\item and
\begin{equation*}
\lO \cX^+(\mathbf{1})=
\begin{cases}
\omega_X \ \ \ \ \text{if $X$ is regular away from $O$, and}\\
0 \ \ \ \ \ \ \ \ \ \text{if $X$ is regular at $O$.}
\end{cases}
\end{equation*}
\end{enumerate}
\end{definition}
As before, if $X$ is a regular vector field, it acts on $\lO\bW$ in the obvious way:
\begin{equation*}
\lO\cX (\alpha_1 \otimes^s \dots \otimes^s \alpha_p)= \sum_i \alpha_1 \otimes^s \dots \otimes^s \cL_X \alpha_i \otimes^s \dots \otimes^s \alpha_p
\end{equation*}
Note that 
\begin{equation}\label{CXalpha}
[\lO\cX,\alpha]=\lO H^{i_X\alpha}
\end{equation}

We turn to the study of commutators. We shall deal only with operators on $\lO\bW$; so $\lO\cX=\cX$, etc.  The basic result is

\begin{proposition}
Given a meromorphic function $\phi$ and a meromorphic vector field $X$,
\begin{equation*} 
[\bL^X,H^\phi]=H^{X(\phi)}
\end{equation*}
\end{proposition}

\begin{proof} We first check that given any $\alpha \in \lO\bK$,
\begin{equation*} 
\begin{split}
 [[\bL^X,H^\phi],\alpha]&=[[\bL^X,\alpha],H^\phi]]\\
&=[H^{i_X\alpha},H^\phi]\\
&=-Res_O ([i_X\alpha]d\phi)\\
&=-Res(X(\phi)\alpha)\\
&=[H^{X(\phi)},\alpha]
\end{split}
\end{equation*}
\end{proof}

\begin{proposition}
Let meromorphic vector fields $X,Y$ be given, and set $Z=[X,Y]$. Then  
\begin{equation*} 
\begin{split}
\{[\cX,\cY]-\cZ\}&\alpha_1 \otimes^s \dots \otimes^s \alpha_p\\
&=\sum_j \alpha_1 \circ \dots \circ \widehat{\alpha_j} \circ \dots \circ \alpha_p  \circ \{[H^{i_X\alpha_j},\cY]+[\cX,H^{i_Y\alpha_j}]-H^{i_Z\alpha_j}\} \mathbf{1}\\
&+\alpha_1 \otimes^s \dots \otimes^s \alpha_p\{[\cX,\cY]-\cZ\} \mathbf{1}\\
\end{split}
\end{equation*}
\end{proposition}

\begin{proof} We have, for $\alpha \in \lO\bK$
\begin{equation*} 
\begin{split}
[[\cX,\cY]-\cZ,\alpha]&=[[\cX,\alpha],\cY]+[\cX,[\cY,\alpha]]-[\cZ,\alpha]\\
&=[H^{i_X\alpha},\cY]+[\cX,H^{i_Y\alpha}]-H^{i_Z\alpha}\\
\end{split}
\end{equation*}
In turn, given any $\beta \in \lO\bK$
\begin{equation*}
\begin{split}
[[H^{i_X\alpha},\cY]+[\cX,H^{i_Y\alpha}]-H^{i_Z\alpha},\beta]
&=[H^{i_X\alpha},H^{i_Y\beta}]+[H^{i_X\beta},H^{i_Y\alpha}]+Res_O(i_Z\alpha \beta)\\
&=-Res_O(i_X\alpha \cL_Y \beta)-Res_O(i_X\beta \cL_Y \alpha)+Res_O(i_Z\alpha \beta)\\
&=0\\
\end{split}
\end{equation*}
Therefore
\begin{equation*}
\begin{split}
\{[\cX,\cY]-\cZ\}&\alpha_1 \otimes^s \dots \otimes^s \alpha_p\\ &= \sum_j \alpha_1 \circ \dots \circ \{[H^{i_X\alpha_j},\cY]+[\cX,H^{i_Y\alpha_j}]
-H^{i_Z\alpha_j}\} \circ \dots \circ \alpha_p\mathbf{1}\\&+\alpha_1 \otimes^s \dots \otimes^s \alpha_p\{[\cX,\cY]-\cZ\} \mathbf{1}\\
&=\sum_j \alpha_1 \circ \dots \circ \widehat{\alpha_j} \circ \dots \circ \alpha_p  \circ \{[H^{i_X\alpha_j},\cY]+[\cX,H^{i_Y\alpha_j}]-H^{i_Z\alpha_j}\} \mathbf{1}\\
&+\alpha_1 \otimes^s \dots \otimes^s \alpha_p\{[\cX,\cY]-\cZ\} \mathbf{1}\\
\end{split}
\end{equation*}
\end{proof}
We will study the algebra generated by the operators $\lO \bL^X$ in a sequel.
 
Given a coordinate $u$ vanishing at $O$, we introduce the notation $L_n=\lO \bL^{X_n}$, where $X_n=-u^{n+1} d/du$.  Note that $X$ is regular away from $O$ precisely when $n \le 1$, and regular at $O$ precisely when $n \ge -1$,  and regular everywhere for $n=-1,0,+1$.  Using (\ref{Heisenbergbosonmodes}) and (\ref{CXalpha}), we have $[L_m,b_n]=-nb_{n+m}$. From this it follows easily that
the operator $[L_l,L_m]-(l-m)L_{l+m}$ commutes with the operators $b_n$. Evaluating it on the vacuum, we obtain
\begin{proposition}
\begin{equation}\label{vircommute}
[L_l,L_m]=(l-m)L_{l+m}+\frac{\bL^3-l}{12} \delta_{l+m}
\end{equation}
\end{proposition}

\subsection{(Virasoro) primary vectors}

Fix a point $O$; let $\cV_O$ denote the Lie algebra of meromorphic vector fields which are regular at $O$ and vanish there. The differential action on the fibre $K_z$ gives a linear map
\begin{equation*}
\begin{split}
\cV_O &\to \bbC\\
X=\xi(u)\frac{d}{du} &\mapsto \xi'(O) \equiv X'(O)
\end{split}
\end{equation*}
(This is independent of the coordinate $u$.)

\begin{definition} A vector $\ttv \in \lO\bW$ is said to be {\it primary of weight $\Delta$} (with respect to $O$) if 
$$
\lO \bL^X (\ttv) = \Delta X'(O) \ttv
$$
for all $X \in \cV_O$.
\end{definition}

One can check that $\mathbf{1}$ is primary. A more interesting primary vector is the one-form $\alpha_{{}_O}$ (unique up to a scalar factor), with a double pole at $O$ and regular elsewhere, which is primary of weight 1. To check this, refer back to the definition \ref{deflOcX}. Given a meromorphic vector field regular at $O$, we have
$$
\lO \bL^X \alpha_{{}_O} = \lO[\cL_X \alpha_{{}_O}]
$$
where $\lO[\cL_X \alpha_{{}_O}]$ is the projection to the part regular away from $O$. Write $X=X_r+X_s$ where $X_r$ is regular everywhere and $X_{s}$ vanishes to order 2 at $O$. Then
$\cL_{X_r} \alpha$ is regular away from $O$ and $\cL_{X_{s}} \alpha$ is regular at $O$, so
$$
\lO[\cL_X \alpha_{{}_O}] = \cL_{X_r} \alpha \underset{\#}=   X_r'(O) \alpha = X'(O) \alpha
$$
where the equality $\#$ is easily checked.
 
\section{\textbf{The boson current algebra with insertions}}\label{bosoninsertions}

We will deal with insertions at length when we consider current algebras. Here we show how an insertion corresponds to a ``zero mode'' and how the Heisenberg algebra gets modified.

We will be considering a modification of the boson current algebra, so we fix a point $P$. Suppose given $n$ complex numbers $\lambda_1,\dots,\lambda_n$ and $n$ distinct points $z_1,\dots,z_n$, all also distinct from $P$. We consider fields $\epsilon$ and $\iota$ as before:
\begin{enumerate}
\item $\epsilon(z): \bbW_- \to \bbW_- \otimes K_z$
\item $\iota(z): \bbW_- \to \bbW_- \otimes K_z$
\end{enumerate}
The definition of $\epsilon$ remains unchanged, but the definition of $\iota$ is modified as folllows:
\begin{equation*}
\begin{split}
[\iota(z), \balpha] &= -d\balpha(z),\ \  \balpha \in \mathfrak{m}_P^z\\
\iota(z) \mathbf{1} &= \sum_j \frac{\lambda_j}{u(z)-u(z_j)} \otimes \mathbf{1}du_z
\end{split}
\end{equation*}

Suppose first that there is only one point, which we denote $O$ rather than $z_1$, and let $\lambda$ (rather than $\lambda_1$) be the corresponding parameter. In this case, one can check that the mode expansion of \S \ref{bosoncurrentmodes} gets modified by the addition of an operator $j_0$, defined by
\begin{equation*}
j_0 (\balpha_1 \otimes^s \dots \otimes \balpha_p) =  \lambda (\balpha_1 \otimes^s  \dots \otimes^s \balpha_p)   
\end{equation*}
which commutes with all the $j_l$. As for the Heisenberg operators, the modification is as follows:
\begin{definition} $\lO H^\phi \equiv \lO\Phi: \lO\bbW_- \to \lO\bbW_-$:
\begin{equation*}
\begin{split}
\lO\Phi [\balpha_1 \otimes^s \dots \otimes \balpha_p]=-\sum_i  ({\balpha}_1 \otimes^s \dots \widehat{{{\balpha}_i}} \dots \otimes^s {\balpha}_p) \times Res_O(\phi d\balpha_i)\\
+
\begin{cases}
\begin{split}
&\{\lambda \phi(P)+(\phi-\phi(P)) \otimes^s\} \balpha_1 \otimes^s \dots \otimes^s \balpha_p\\
&\text{if $\phi$ is regular away from $O$}\\
&\ \ \ \ \ \ \ \ \ \ \ \text{and}\\
\end{split}\\
\begin{split}
&\lambda \phi(O) \balpha_1 \otimes^s \dots \otimes^s \balpha_p\\
&\text{if $\phi$ is regular at $O$}
\end{split}
\end{cases}
\end{split}
\end{equation*}
\end{definition}

\emph{Note that the constant function $\phi=1$ acts by the scalar $\lambda$.}

In the general case, with $n$ insertions, we have for each $l=1,\dots,n$, operators
\begin{definition} $H_\bL^\phi \equiv \Phi_j: \lbz\bbW_- \to \lbz\bbW_-$:
\begin{equation*}
\begin{split}
\Phi_l [\balpha_1 \otimes^s \dots \otimes \balpha_p]=-\sum_i  ({\balpha}_1 \otimes^s \dots \widehat{{{\balpha}_i}} \dots \otimes^s {\balpha}_p) \times Res_{z_l}(\phi d\balpha_i)\\
+
\begin{cases}
\begin{split}
&\{(\sum_j \lambda_j) \phi(P)-\sum_{j \ne l} \lambda_j \phi(z_j)+(\phi-\phi(P)) \otimes^s\}\balpha_1 \otimes^s \dots \otimes^s \balpha_p\\
&\text{if $\phi$ is regular away from $z_l$}\\
&\ \ \ \ \text{and}
\end{split}\\
\begin{split}
&\lambda_l \phi(z_l)  \balpha_1 \otimes^s \dots \otimes^s \balpha_p\\
&\text{if $\phi$ is regular at $z_l$.}
\end{split}
\end{cases}
\end{split}
\end{equation*}
\end{definition}

The following is straightforward to prove, and we will in fact state and prove a more general result in the context of curent algebras.

\begin{proposition} Let $\phi,\psi$ be given, and consider the corresponding operators $\Phi_l,\Psi_l$ acting on $\lbz \bbW_-$.
\begin{enumerate}
\item If $\phi=1$ (the constant function with value 1),  then 
$\Phi_l \mathbf{1}=\lambda_l \mathbf{1}$.
\item $[\Phi_l,\Psi_l]=-Res_{z_l} (\phi d\psi)$.
\item $[\Phi_l,\Psi_{l'}]=0$ if $l\ne l'$,
\item $\sum_l \Phi_l=\Phi_P|_{\lbz \bbW_-}$
\end{enumerate}
\end{proposition}

See (\ref{zeromodes}) for a context in which zero modes occur ``naturally".

\section{\textbf{The boson: pairings}}

\subsection{Pairings}\label{bosonpairing}

Let $\gamma$ be a contour in $\bbP$,  separating domains $D$ and $D'$.  We define fields $e(z), \  i(z), \ z\in D$ as before, but consider them as acting on $\lD\bW$ and $\lD\bW^z$ respectively.  \emph{Note that $i(z)$ extends to points $z \in \bbP$; in fact, for $z \notin D$, $i(z)$ is defined on $\lD\bW=\lD\bW^z$.} For $z' \in D'$, define $e'(z')$ and $i'(z')$ similarly as fields acting on $\lDp\bW$ and $\lDp\bW^{z'}$ respectively.

The spaces $\lD\bK$ and $\lDp\bK$ are paired as follows. Given $\alpha=df \in \lD\bK$ and $\alpha'=df' \in \lDp\bK$,
\begin{equation}\label{integralpairing}
\begin{split}
\lDp\bK \times \lD\bK &\to \mathbf{C}\\
\langle \alpha',\alpha \rangle&=-\frac{1}{2 \pi i} \int_{\gamma} f' df\\
\end{split}
\end{equation}
(Note that $f$ is a meromorphic function with all poles in $D$ and $f'$ ia meromorphic function with all poles in $D'$.) Note that if $X$ is a regular vector field, 
\begin{equation*}
\begin{split}
\langle \alpha',\bL^X \alpha \rangle &=-\frac{1}{2 \pi i} \int_{\gamma} f' d i_X \alpha\\
&=\frac{1}{2 \pi i} \int_{\gamma} df'  i_X df\\
&=\frac{1}{2 \pi i} \int_{\gamma} df'(Xf)\\
\end{split}
\end{equation*}
On the other hand
\begin{equation*}
\langle \bL^X \alpha',\alpha \rangle=-\frac{1}{2 \pi i} \int_{\gamma} (Xf')df
\end{equation*}
which shows that
\begin{equation}\label{pairingX}
\langle \bL^X \alpha',\alpha \rangle=-\langle \alpha',\bL^X \alpha \rangle
\end{equation}
since $(Xf')df=df'(Xf)$.

It is easy to check that given $\alpha \in \lD\bK,\ \ \alpha' \in \lDp\bK$, we have
\begin{equation}\label{pairinglowestdegree}
\begin{split}  
\langle \alpha',e(z) \mathbf{1}\rangle &= \alpha'(z)=-\langle i'(z)\alpha', \mathbf{1}\rangle\ \text{(equality of forms regular in $D$)}\\
\langle e'(z')\mathbf{1'},\alpha\rangle &= \alpha(z')=- \langle \mathbf{1'},i(z')\alpha\rangle \ \text{(equality of forms regular in $D'$)}\\
\end{split}
\end{equation}
Also, for any natural number $N$,
\begin{equation*}
\begin{split}  
\langle \alpha',\bcL_X^N e(z) \mathbf{1}\rangle &= \{\cL_X^N\alpha'\}(z)=-\langle \bcL_X^N i'(z)\alpha', \mathbf{1}\rangle \\
\langle \bcL_X^N e'(z')\mathbf{1'},\alpha\rangle &= \{\cL_X^N \alpha\}(z')=- \langle \mathbf{1'},\bcL_X^N i(z')\alpha\rangle \\
\end{split}
\end{equation*}
We would like to extend the pairing between $\lD\bK$ and $\lDp\bK$ to a pairing between $\lD\bW$ and $\lDp\bW$. It is clear enough how to do this, since the latter are the respective symmetric algebras, but with the case of current algebras in mind, we formulate the following 

\begin{proposition}\label{propbosonpairing} There is a unique pairing between  $\lD\bW$ and $\lDp\bW$ such that $\langle {\mathbf 1}',\mathbf 1\rangle =1$, and if $\chi \in \lD\bW,\ \ \chi' \in \lDp\bW$, we have
\begin{itemize}
\item the equality sections of $K$ on $D$: 
\begin{equation}\label{pairing1}
\langle \chi',e(z) \chi\rangle =-\langle i'(z) \chi', \chi\rangle
\end{equation}
and
\item (for $z' \in D'$) the equality of meromorphic sections of $K$ on $D'$:
\begin{equation}\label{pairing2}
\langle e'(z')\chi',\chi\rangle =-\langle \chi',i(z') \chi\rangle 
\end{equation}
\end{itemize}
and such that 
\begin{equation*}
\begin{split}
\langle \chi',\bcL_X^N e(z) \chi\rangle &=\cL_X^N \langle \chi', e(z) \chi\rangle= 
-\cL_X^N \langle  i'(z) \chi', \chi\rangle=
-\langle\bcL_X^N  i'(z) \chi', \chi\rangle\\
\langle \bcL_X^N e'(z')\chi',\chi\rangle &=\cL_X^N \langle  e'(z')\chi',\chi\rangle =-
\cL_X^N \langle \chi',i(z') \chi\rangle=
\langle \chi',\bcL_X^Ni(z') \chi\rangle \\
\end{split}
\end{equation*}

\end{proposition}

\begin{proof} As preparation for dealing with current algebra (where we deal with universal enveloping algebra rather than the symmetric algebra) we will use only the filtrations by degrees on $\lD\bW$ and $\lD\bW'$, but \emph{not} the gradings. Suppose that the pairing $\langle \chi',\chi\rangle $ is given for $degree\ \chi  \le N$ and  $degree\ \chi'  \le N'$, with $N,N'$ nonnegative integers. Let $degree\ \chi  = N$ and  $degree\ \chi'  = N'$, and consider the pairing (for $z \in D$) $\langle \chi',e(z) \chi\rangle $; by (\ref{pairing1}) above, this should be equal to $\langle i'(z) \chi', \chi\rangle $. But $degree\ i'(z') \chi' \le N'$ - in fact strict inequality holds, though this will be not the case in the case of current algebras - so that the pairing $\langle \chi',e(z) \chi\rangle $ is determined. Differentiating with respect to $z$ and adding terms corresponding to different choices of $z$ allows us to determine all pairings $\langle \chi',\chi\rangle $ with $degree\ \chi = N+1$. Exchanging roles of $\chi$ and $\chi'$, we can similarly increment $N'$ by 1.
Thus we have shown that the pairing, if it exists, is unique.

The existence follows from the next lemma.
\end{proof}

\begin{lemma}\label{lemmaconsistentpairing}Suppose that a pairing as in the proposition exists for all $\chi \in \lD\bW$ with $degree\ \chi \le N$ and $\chi' \in \lDp\bW$ with $degree\ \chi' \le N'$. Then, given such $\chi,\ \chi'$, and for $z \in D$ and $z' \in D'$, we have 
\begin{equation*}
\langle i'(z) e'(z') \chi', \chi\rangle =\langle \chi', i(z') e(z)\chi\rangle 
\end{equation*}
\end{lemma}

\begin{proof} Note that $degree\ i'(z) e'(z') \chi' = degree\ \chi'$ and $degree\ i(z') e(z)\chi = degree\ \chi$, so that both sides of the above equation are determined by assumption of the Lemma. The proof itself is a computation that uses (\ref{iecommute}) (choosing a coordinate, we write $i(z)=\hi(z)du_z$, etc.):
\begin{equation*}
\begin{split}
\langle \hi'(z) \he'(z') \chi', \chi\rangle &=\frac{1}{(u(z)-u(z'))^2}\langle \chi', \chi\rangle +\langle  \he'(z') \hi'(z) \chi', \chi\rangle \\
&=\frac{1}{(u(z)-u(z'))^2}\langle \chi', \chi\rangle +\langle  \chi', \he(z) \hi(z') \chi\rangle \\
&=\langle  \chi',\hi(z') \he(z) \chi\rangle \\
\end{split}
\end{equation*}
\end{proof}

Let $\phi$ be a meromorphic function regular along $\gamma$. A formal computation, which can easily be made rigorous, shows:

\begin{proposition} With respect to the pairing between $\lDp\bW$ and $\lD\bW$ defined above, we have
\begin{equation*}
\langle \lDp\Phi (\chi'), \chi \rangle = \langle  \chi', \lD\Phi (\chi) \rangle
\end{equation*}
\end{proposition}

(Recall that the contour integral along the common boundary is taken in opposite directions for the two domains $D$ and $D'$. This explains the absence of a minus sign.)

Taking unions and intersections, one obtains a pairing between $\bW_P$ and $\lP\bW$.  Fix a coordinate $u$  (with a pole at $P$). For $l \ne 0$, define operators ${}_{l}\ttb:\lP\bW \to \lP\bW$ as follows:
\begin{equation*}
{}_l \ttb=
\begin{cases}
\begin{aligned}
& du^m \mapsto l \delta_{l+m} \ l \le -1,\ \text{extended as a derivation}\\
& du^l \otimes^s \ \ \ \ \ \ \ \ \ l \ge 1\\
\end{aligned}
\end{cases}
\end{equation*}
The operators ${}_l \ttb, \ l \ge 1$, acting on the vacuum, generate $\lP\bW$. The maps $b_l$ as well as the space $\lO \bW$ are defined in the section \S \ref{bosonmodes} above; $O$ is the point where the co-ordinate $u$ vanishes.   

\begin{proposition} There is a unique pairing between  $\bW_P$ and $\lP\bW$ such that $\langle {\mathbf 1}',\mathbf 1\rangle =1$, and if $\chi' \in {}_{P}\bW,\ \ \chi \in \bW_P$, we have
\begin{itemize}
\item for $z \ne P$, 
\begin{equation*}
\langle  \chi', e(z) \chi\rangle = -\langle i(z) \chi',\chi\rangle
\end{equation*}
\item and for $l \ge 1$ 
\begin{equation*}
\langle {}_{l}\ttb (\chi'),\chi\rangle =-\langle \chi', b_{l} (\chi)\rangle
\end{equation*}
\end{itemize}
If $\chi' \in {}_{P}\bW,\ \ \chi \in \lO \bW$, we have
\begin{itemize}
\item for $l \le 1$ 
\begin{equation*}
\langle {}_{l}\ttb (\chi'),\chi\rangle =-\langle \chi', b_{l} (\chi)\rangle
\end{equation*}
\end{itemize}
\end{proposition}

It is straightforward to check that the pairing between  $\bW_P$ and $\lP\bW$ is independent of the choice of coordinate $u$.

\begin{proposition} With respect to the pairing between $\lP\bW$ and $\bW_P$ defined above, we have
\begin{equation*}
\langle \lP\Phi (\chi'), \chi \rangle = \langle  \chi', \Phi_P (\chi) \rangle
\end{equation*}
\end{proposition}

\subsection{Hermitian Pairing}\label{subsection: Hermitian Pairing}

We fix an antiholomorphic involution $C:\bbP \to \bbP$. Then a choice of $P$ outside the fixed-point set of $C$ yields natural choices of the rest of the data $\gamma$ (the fixed point set), $D'$ (the component of the complement of $\gamma$ containing $P$) and $D$ (the other component, with $D'=C(D)$). In terms of adapted coordinates $C(u)=\frac{1}{\overline u}$.

When $D$ and $D'$ are related by a conjugation $C$ as above, (\ref{integralpairing}) is a perfect pairing. To see this, consider $\alpha, \beta \in \lD\bK$. These forms are regular on $D'$ (in fact on the closure); their (Bergman space) inner product is 
$(\alpha,\beta)=\frac{1}{2 \pi i} \int_{D'} \overline{\alpha} \beta$. On the other hand,
$$
\alpha \to -C^*\overline{\alpha},\ \alpha' \to -C^*\overline{\alpha'}
$$
sets up an antilinear bijection $\lD\bK \leftrightarrow \lDp\bK$, and we claim that
\begin{equation*}
\langle C^* \overline{\alpha},\beta \rangle=(\alpha,\beta)
\end{equation*}
which yields the claimed nondegeneracy. (The minus sign is introduced for consistency with later choices.) To check the claim, write $\alpha=df$. Then
\begin{equation*}
\begin{split}
\langle C^*\overline{\alpha},\beta \rangle &= -\frac{1}{2 \pi i} \int_{\gamma} C^*\overline{f} \beta\\
&=-\frac{1}{2 \pi i} \int_{\gamma} \overline{f}  \beta\\
&=\frac{1}{2 \pi i} \int_{D'} \overline{df} \beta\\
&=\frac{1}{2 \pi i} \int_{D'} \overline{\alpha} \beta\\
&=(\alpha,\beta)\\
\end{split}
\end{equation*}
Here we use the fact that $C^*\overline{f}=\overline{f}$ on $\gamma$.  

\begin{proposition}\label{prophermitianbosonlD} There is a unique hermitian structure $(,)$ on  $\lD\bW$  such that $({\mathbf1},\mathbf{1}) =1$, and if  $\chi,\psi \in \lD\bW$, we have
\begin{itemize}
\item (for $z \in D$) the equality of meromorphic sections of $K$ on $D$:
\begin{equation}\label{hermitianpairing1}
\begin{split}
(\chi,e(z) \psi) &=-C^*(i(C(z))\chi, \psi), \textrm{or equivalently}\\
u(z)^2 (\chi,\he(z)\psi) &=(\hi(C(z))\chi, \psi)\\
\end{split}
\end{equation}
and
\item (for $z' \in D'$) the equality of meromorphic sections of $K$ on $D'$:
\begin{equation}\label{hermitianpairing2}
\begin{split}
(\chi,i(z') \psi) &=-C^*(e(C(z'))\chi,\psi), \textrm{or equivalently}\\
u(z')^2(\chi,\hi(z') \psi)& =(\he(C(z'))\chi,\psi)\\
\end{split}
\end{equation}
\end{itemize}
\end{proposition}

\begin{proof}
This is a consequence of Proposition \ref{propbosonpairing}, but it is simpler to prove it directly. Lemma \ref {lemmaconsistentpairing} gets replaced by its hermitian analogue, which we state below.
\end{proof}

\begin{lemma}\label{lemmaconsistenthermitianpairing}Suppose that an inner product as in the proposition exists for all $\psi \in \lDp\bW$ with $degree\ \psi \le N$ and $\chi \in \lDp\bW$ with $degree\ \chi \le N$. Then, given such $\chi,\ \chi$, and for $w,z \in D$, we have 
\begin{equation*}
\begin{split}
C^*_{w}(\chi, i(C(w)) e(z) \psi)&=C^*_{z}(i(C(z)) e(w) \chi, \psi), 
\textrm{or equivalently}\\
u(z)^2(\chi, \hi(C(w)) \he(z) \psi)&=\overline{u}(w)^2(\hi(C(z))\he(w)\chi,\psi)\\ 
\end{split}
\end{equation*}
\end{lemma}

\begin{proof} Set $C(w)=z'$. The earlier computation gets replaced by:
\begin{equation*}
\begin{split}
(\chi, \hi(C(w)) \he(z) \psi)&=(\chi, \hi(z') \he(z) \psi)=\frac{1}{(u(z')-u(z))^2}(\chi, \psi) + (\chi,\he(z) \hi(z')\psi)\\
&=\frac{1}{(u(z')-u(z))^2}(\chi, \psi) +(\he(C(z')) \hi(C(z)) \chi,\psi) \frac{1}{u(z')^2 u(z)^2}\\
&=(\hi(C(z)) \he(C(z))\chi,\psi) \frac{1}{u(z')^2 u(z)^2}\\
&=(\hi(C(z)) \he(w) \chi,\psi) \frac{1}{u(z')^2 u(z)^2} \\
&=(\hi(C(z)) \he(w) \chi,\psi) \frac{\overline{u}(w)^2}{u(z)^2} \\
\end{split}
\end{equation*}
since
\begin{equation*}
\begin{split}
(\he(C(z)) \hi(C(z)) \chi,\psi)&-(\hi(C(z)) \he(C(z))\chi,\psi)\\&=-(\frac{\overline{u}(z)^2\overline{u}(z)^2}{(\overline{u}(z)-\overline{u}(z))^2} \chi,\psi)
\\
&=-\frac{u(z)^2u(z)^2}{(u(z)-u(z))^2}(\chi,\psi)\\
\end{split}
\end{equation*}
\end{proof}

Consider states ``supported at'' $O$, ie., $\lO \bW \subset \lD\bW$.

\begin{proposition}\label{prophermitianbosonlO} There is a unique hermitian structure $(,)$ on  $\lO \bW$  such that $({\mathbf1},\mathbf{1}) =1$, and if  $\chi,\psi \in \lO \bW$, we have
$$
(\chi, b_l \psi) = (b_{-l} \chi, \psi),\ l \in \bbZ
$$
\end{proposition}

\subsection{Reflection positivity}

Let us prove that the hermitian structure we constructed is positive. In fact we will be able to check that this is essentially the inner product on the symmetric algebra induced from the Bergman inner product.

Consider sequences of distinct points $y_1,\dots,y_l$ and $z_1,\dots,z_m$ in $D$. Then (\ref{hermitianpairing1}) gives
\begin{equation*}
\begin{split}
(\he(y_1)\dots \he(y_l)\mathbf{1},\he(z_1)\dots \he(z_m)\mathbf{1})=
u^2(z_1) &(\hi(C(z_1))\he(y_1)\dots \he(y_l)\mathbf{1},\he(z_1)\dots \he(z_m)\mathbf{1})\\
=u^2(z_1)\sum_{p}\frac{1}{(\bar{u}(C(z_1)-\bar{u}(y_p))^2}&(\he(y_1)\dots \widehat{\he(y_p)} \dots \he(y_l)\mathbf{1},\he(z_2)\dots \he(z_m)\mathbf{1})\\
=\sum_{p}\frac{1}{(1-u(z_1)\bar{u}(y_p))^2}&(\he(y_1)\dots \widehat{\he(y_p)} \dots \he(y_l)\mathbf{1},\he(z_2)\dots \he(z_m)\mathbf{1})\\
\end{split}
\end{equation*}
which yields by induction

\begin{proposition}
\begin{equation*}
(\he(y_1)\dots \he(y_l)\mathbf{1},\he(z_1)\dots \he(z_m)\mathbf{1})=
\begin{cases}
0 \ \ \ \text{if $l \ne m$ and}\\
\sum_{\sigma} \frac{1}{(1-u(z_1){\bar u}(y_{\sigma(1)}))^2} \dots \frac{1}{(1-u(z_l)\bar{u}(y_{\sigma(l)}))^2}
\end{cases} 
\end{equation*}
where the sum runs over all permutations $\sigma$ of $(1,\dots,l)$

\end{proposition}

\section{\textbf{The boson on a disc}}

Let us briefly explore an analytic (as opposed to algebraic) version of the chiral boson.

Let $\bbD$ be a simply connected Riemann surface biholomorphic to the unit disc in $\bbC$. We let $K$ denote the canonical bundle.

Consider the space of holomorphic vector fields that integrate to give automorphisms of $\bbD$.  This is a three dimensional real vector space, whose complexification is a space of holomorphic vector fields without base-points. Thus $\bbD$ can be canonically embedded as a domain in a conic $\mathbb{P}$ in a $\mathbb{P}^2$ with a real structure which induces a canonical antiholomorphic involution $C:\mathbb{P} \to \mathbb{P}$. We let $\bbD'$ denote the complement of the closure $\overline{\bbD}$. The involution $C:\mathbb{P} \to \mathbb{P}$ exchanges $\bbD$ with $\bbD'$ and acts as identity on the boundary $\partial \bbD$. We will often identify $\bbD$ with the unit disc $\{u \in \mathbb{C}| |u|<1\}$; w.r.to this coordinate,  $C(u)=\frac{1}{\overline{u}}$.

We let $V$ denote the Hilbert space of square-integrable holomorphic one-forms on $\bbD$. If $\bbD$ is identified with the unit-disc, a typical element of $V$ is $\alpha=a(u) du$, where $a$ is a holomorphic function defined for $|u|<1$ such that the area integral
\begin{equation*}
||\alpha||^2=\frac{1}{2\pi i} \int_\bbD \overline{\alpha} \alpha=\frac{1}{\pi} \int_\bbD |a(u)|^2 dx dy < \infty
\end{equation*}
where $u=x+iy$. Thus $V$ is the Bergman space $A^2(\bbD)$. We let $\overline{V}$ denote the space of square integrable \emph{anti}holomorphic one-forms; complex conjugation is an complex antilinear isometry $V \to \overline{V}$, and the equation
\begin{equation*}
(\alpha,\beta)=<\overline{\alpha},\beta>=\frac{1}{2\pi i} \int_\bbD \overline{\alpha} \beta
\end{equation*}
relates the inner product on $V$ and the duality pairing between $V$ and $\overline{V}$.

Given $z \in \bbD$, the evaluation map $ev_z:V \to K_z$ is bounded. Therefore there exists a unique vector in $e_z \in \overline{V} \otimes K_z$ such that
\begin{equation*}
<e_z,\beta>=\beta (z)
\end{equation*}
In fact $e_z$ is the Bergman kernel; trivialising $K_z$ by $du_z$, we have
\begin{equation*}
e_z= \frac{1}{(1-u(z)\overline{u}_1)^2}d\overline{u}_1 \otimes du_z \equiv \gamma_z du_z
\end{equation*}
(where $\gamma_z$, for fixed $z$, is an antiholomorphic $(0,1)$-form). Note that as $z$ varies $\gamma_z$ is a holomorphic family of vectors in $\overline{V}$. 

We will also need to consider evaluation at points \emph{outside} $\bbD$. For $z \in \mathbb{P}$ let $V^z$ denote the subspace of $V$ consisting of \emph{rational} forms regular at $z$; let $\overline{V}^z$ be defined analogously. Consider the evaluation map $ev_z:\overline{V}^z \to \overline{K}_z$; it is clear that given an element $\overline{\alpha}$ in $\overline{V}$ that is defined in a neighbourhood of $z$, the evaluation $ev_z(\overline{\alpha})=\overline{\alpha}(z)$ varies \emph{anti}-holomorphically with $z$; to compensate, we compose with the pull-back by the complex conjugation map.

Now to a precise definition. For $z \in \bbD$ define $i_z:\overline{V}^z \to K_z$ by $i_z=-C^* \circ ev_{C(z)}$ (up to a convenient sign choice, ``evaluate at $\frac{1}{\overline{z}}$ and pull back by $C$, in the process changing a $(0,1)$-form to a $(1,0)$-form."). If we identify $\bbD$ with the unit disc, the map $i_z$ is given by the formula:
\begin{equation*}
i_z(\overline{\alpha})=\overline{a}(\frac{1}{\overline{u}(z)}) {\frac{1}{u(z)^2}}du_z
\end{equation*}
Note that $e_z$ and $i_z$ would be formal adjoints, but for the complex conjugation involved in the definition of $i$.

From now on, we only consider rational forms on $\mathbb{P}$ that are regular on $\bbD$, or their conjugates.  We let $\bar{\bK}_\bbD \subset \overline{V}$ be defined by
$$
\bar{\bK}_\bbD \equiv \{\overline{\alpha}|\alpha\ a\ meromorphic\ form\ on\ \mathbb{P}\ regular\ on\ \bbD \}
$$
Similarly, we set, for $z \in \bbD$
$$
\bar{\bK}_\bbD^z  \equiv \{\overline{\alpha}|\alpha\ a\ meromorphic\ form\ on\ \mathbb{P}\ regular\ on\ \bbD\ and\ near\ C(z) \}
$$

We consider the ``bosonic Fock space over $\bar{\bK}_\bbD$", namely the symmetric algebra
\begin{equation*}
\mathcal{H}=\mathbb{C} \oplus \bar{\bK}_\bbD \oplus S^2 \bar{\bK}_\bbD \oplus ......
\end{equation*}
and (for $z \in \bbD$) the subspace
\begin{equation*}
\mathcal{H}^z=\mathbb{C} \oplus \bar{\bK}_\bbD^z \oplus S^2 \bar{\bK}_\bbD^z \oplus ......
\end{equation*}
We will let $\mathbf{1}$ denote the vacuum vector in the first summand, $1 \in \mathbb{C}$. 

We now elevate $e_z$ and $i_z$ to the level of operators on Fock space, by defining $e(z): \mathcal{H} \to \mathcal{H} \otimes K_z$ as
\begin{equation*}
e(z) (\overline{\alpha}_l \otimes^s \dots \otimes^s \overline{\alpha}_l) = \overline{\alpha}_l \otimes^s \dots \otimes^s  \overline{\alpha}_l \otimes^s  e_z 
\end{equation*}
where we mean the symmetric tensor product, and $i(z): \mathcal{H}^z \to \mathcal{H} \otimes K_z$ by the formula
\begin{equation*}
i(z) (\overline{\alpha}_l \otimes^s \dots \otimes^s \overline{\alpha}_l) = \sum_i \overline{\alpha}_l \otimes^s \widehat{{\overline{\alpha}_i}} \otimes^s \dots \otimes^s \overline{\alpha}_l \otimes i_z(\overline{\alpha}_i) 
\end{equation*}
The hat marks a term to be omitted; we set
\begin{equation*}
b(z)=i(z)+e(z)
\end{equation*}
with the domain $\mathcal{H}^z$.

Given two distinct points $z_1$ and $z_2$, the operators $i(z_1), \ e(z_1)$ and $i(z_2), \ e(z_2)$ can be composed on $\mathcal{H}^{z_1} \cap \mathcal{H}^{z_2}$. Clearly $i(z_1),\ i(z_2)$ commute as do $e(z_1),\ e(z_2)$; on the other hand
\begin{equation*}
i(z_1) \circ e(z_2) - e(z_2) \circ i(z_1) = \frac{1}{(u(z_1)-u(z_2))^2} du_{z_1} \otimes^s du_{z_2}
\end{equation*}
For distinct points $z_1$ and $z_2$, the operators $b(z_1), \ b(z_2)$ commute.

Before we return to a more algebraic point of view, we note that if we consider the Hilbert space completion of $\cH$, the operator $i(z)$ (for $z \in \bbD$) is densely defined but  \emph{not closeable}. On the other hand, one can define the Heisenberg operator $\Phi$ associated to a meromorphic function $\phi$ on $\bbD$ regular near the boundary, and \emph{these operators are densely defined and closeable}.

\pagebreak

\part{\textbf{{\Large The boson: vertex algebra package}}}\label{bosonvertexalgebra}

\section{\textbf{Preliminaries}}

Once a base-point $P$ is chosen at infinity, we have the sequence of groups:
\begin{equation}\label{AutP}
\text{Translations of }\ \bbP \setminus P \lhd Aut_P\ \bbP \twoheadrightarrow \bbC^*
\end{equation}
the map to $\bbC^*$ being given by the inverse of the differential action at $P$. Since translations are central, the quotient $\bbC^*$ acts on translations; with our conventions, this is the standard action of $\bbC^*$ on a one-dimensional vector space.
Given a translation $v$, we will use the notation
$$
(z,v) \mapsto z+v
$$
to describe its action at the point $z$; and if a coordinate $u$ is chosen, we will define $u(v)$ by $u(z+v)=u(z)+u(v)$. We denote by $\sfT_v$ the operator on meromorphic 1-forms corresponding to pulling back by $-v$:
$$
\sfT_v \{a(u)du\} = a(u-u(v))du
$$
and keep the same notation for induced operators
\begin{itemize}
\item $\sfT_v: \bW \to \bW$,
\item $\sfT_v: \bW_P \to \bW_P$,
\item $\sfT_v: {}_{O}\bW \to {}_{O+v}\bW$,
\end{itemize}
More generally, $Aut_P \ \bbP$ acts on $\bW$ preserving $\bW_P$, but it is convenient to describe this action slightly less canonically. Given $O \ne P$, the sequence (\ref{AutP}) splits, and we have 
$$
Aut_P \bbP = \{\text{translations of $\bbP \setminus P$}\} \rtimes \{\bbC^* =Aut_{O,P}\} 
$$
In terms of a coordinate $u$ adapted to $O$ and $P$, the action of $\bbC^*$ is
$$
u(\lambda(z))=\lambda u(z)
$$
We have an induced action of $\bbC^*$ on the flag of spaces $\lO \bW \subset \bW_P \subset \bW$; we will denote by $\sfR_{\lambda}$ the operator corresponding to pullback by the map $z \mapsto \lambda^{-1}(z)$. As in the case of translations $\sfR_{\lambda}$ acts on $\bW$ preserving $\bW_P$ and $\lO\bW$, and takes ${}_{z}\bW$ to ${}_{\lambda(z)}\bW$.  Because this is part of the $Aut_P \ \bbP$-action, we have
$$
\sfR_{\lambda} \circ \sfT_v \circ \sfR_{\lambda^{-1}} = \sfT_{\lambda v}
$$
Recall the definition of $\lO \bW$:
\begin{equation*}
\lO \bW=\bbC \oplus H^0(\bbP \setminus O, K) \oplus S^2 H^0(\bbP \setminus O, K) \oplus ......
\end{equation*}
We make explicit the action of $\bbC^*$ on $H^0(\bbP \setminus O, K)$:
$$
\lambda (du^{-l})=\lambda^{l} du^{-l}, \ l\ge 1
$$
The action on $\lO\bW$ is the induced action on the symmetric algebra; we write
$$
\lO \bW=\underset{m \ge 0}{\oplus} \lO\bW_m
$$
with $\lO\bW_m$ the subspace on which $\lambda$ acts by the scalar $\lambda^{m}$. We remark that $\lO\bW_0$ is spanned by the vacuum vector $\mathbf{1}$ and for $m>0$,  $\lO\bW_m$ is spanned by the vectors
\begin{equation*}
b_{-l_1}b_{-l_2} \dots {b_{-l_p}}\mathbf{1}=du^{-l_1} \otimes^s du^{-l_2}  \otimes^s  \dots  \otimes^s du^{-l_p}
\end{equation*}
with $l_j > 0$ and $\sum_j l_j=m$.

\section{\textbf{Definition of a vertex algebra structure}}

\textbf{\emph{Until further notice, we will fix a point $P$ at infinity, and work in its complement. To keep notation simple, we set $\bK \equiv \bK_P$ and $\bW \equiv \bW_P$}}.   Suppose given a sequence of points $\bz=(z_1,\dots,z_n)$ distinct from each other and from $P$. We let $|\bz|$ denote the underlying set $\{z_1,\dots,z_n\}$. Let $\lbz \bW=S^* \lbz \bK$, where $\lbz\bK$ is the space of forms (of the second kind) regular outside the $z_i$. Given a vector $\ttv \in \bW$, let $Sing(\ttv)$ denote the smallest set $|\bz|$ such that $\ttv \in \lbz \bW$. (Thanks to N. Nitsure for insisting on the introduction of this notation.)

Given a ``variable'' translation vector $v$, we set $\bW^{|\bz|+v}=S^*\bK^{|\bz|+v}$, where $\bK$ is space of meromorphic forms of the second kind regular at ($P$ as well as) points of the set $\{|\bz|+v\}$.

\begin{definition}\label{defnvertexalgebra} A \textit{vertex algebra structure} on the chiral boson is a linear map that assigns to each state $\ttv \in \lbz \bW$, and translation $v$, an operator
\begin{equation*}
Y(\ttv,v):\bW^{|\bz|+v} \to  \bW
\end{equation*}
such that:
\begin{enumerate}
\item If $\ttv' \in {}_{|\bz'|}\bW \cap \bW^{|\bz|+v}$, then $Y(\ttv,v)\ttv' \in {}_{|\bz'|\cup(|\bz|+v)}\bW$. (That is, $Sing(Y(\ttv,v)\ttv') \subset Sing(\ttv') \cup (Sing(\ttv)+v)$. So to speak, acting on any state, $Y(\ttv,v)$ adds singularities at the singular points of $\ttv$, translated by $v$.)
\item $Y(\mathbf{1},v)=Id$.
\item $Y(\ttv,0)\mathbf{1} = \ttv$.
\item \textbf{Given $\ttv \in \lbz \bW$ and $\ttv' \in {}_{|\bz'|} \bW$, the operators $Y(\ttv,v)$ and $Y(\ttv',v')$ commute if $|\bz|+v$ is disjoint from $|\bz'|+v'$}.
\item Given translations $v,v'$,
$$
Y(\ttv,v+v')= \sfT_{v'} \circ Y(\ttv,v) \circ \sfT_{-v'}
$$
\item  The assignment
$$
(\ttv, v) \mapsto Y(\ttv,v)
$$
is $Aut_P \ \bbP$-covariant. Explicitly,
\begin{itemize}
\item given a translation $v'$, $Y(\sfT_{v'}\ttv,v)=\sfT_{v'} Y(\ttv,v) \sfT_{-v'}$, and
\item given a rotation $z \mapsto \lambda (z)$ around a point $O$,
$Y(\sfR_{\lambda}\ttv,\lambda v)= \sfR_{\lambda} Y(\ttv,v) \sfR_{\lambda^{-1}}$.
\end{itemize}
(\emph{We will see below that covariance with respect to translations is guaranteed by the properties (3)-(5).})
\item Fix a point $O$. The span of $Y(\ttv_n,v_n) \dots Y(\ttv_1,v_1)\mathbf{1}$, for all sequences of distinct translation vectors $(v_1,\dots, v_n)$, and states $(\ttv_1,\dots,\ttv_n)$, \emph{all belonging to $\lO\bW$}, is $\bW_P$.
\end{enumerate}
\end{definition}

\subsection{General Results} Before exhibiting any vertex algebra structures, we turn to some basic results in the theory of vertex algebras, (re)stated and (re)proved in the context of the boson. (I expect that these will hold in our other examples of field theories.) Because we are not constrained to work with formal power series, both the statements and proofs are quite simple. 

The trick used in the starred equation in the proof of the ``Skew-symmetry" Proposition below (and again in the proof of ``Associativity'') is key to the proof of the Goddard Uniqueness Theorem, a fundamental statement about vertex algebras.

\begin{theorem} We have, for any $v\in \lbz\bW$,
$$
Y(\ttv,v)\mathbf{1}=\sfT_v(\ttv)
$$
\end{theorem}

\begin{proof} From Definition \ref{defnvertexalgebra}, (4)
\begin{equation*}
\begin{split}
Y(\ttv,v)\mathbf{1} &= \sfT_{v} \circ Y(\ttv,0) \circ \sfT_{-v}\mathbf{1}\\
&=\sfT_{v} \circ Y(\ttv,0) \mathbf{1}\\
&=\sfT_{v} (\ttv)\\
 \end{split}
\end{equation*}
\end{proof}

\begin{proposition}\textup{\textbf{``Skew-symmetry''.}} Given states $\ttv_1 \in {}_{|\bz_1|} \bW$ and $\ttv_2 \in {}_{|\bz_2|} \bW$, we have
$$
Y(\ttv_1,v)\ttv_2=\sfT_{v} Y(\ttv_2,-v)\ttv_1
$$
\end{proposition}

\begin{proof} Note that for the left hand side to make sense, $\ttv_2$ should belong to $\bW^{|\bz_1|+v}$, which is equivalent to demanding that $|\bz_2| \cap (|\bz_1|+v) = \emptyset$. For the right hand side to make sense, $\ttv_1$ must belong to $\bW^{|\bz_2|-v}$, which is equivalent to $|\bz_1|\cap (|\bz_2|-v) = \emptyset$ -- and this amounts to the same thing. The left hand side belongs to ${}_{|\bz_2| \cup (|\bz_1|+v)}\bW$, and the right hand side to $\sfT_{v} \{{}_{|\bz_1| \cup (|\bz_2|-v)}\bW\}={}_{(|\bz_1|+v)  \cup  |\bz_2|}\bW$.  Having got this out of the way, we check 
\begin{equation*}
\begin{split}
Y(\ttv_1,v)\ttv_2&\underset{*}=Y(\ttv_1,v)Y(\ttv_2,0)\mathbf{1}\\
&=Y(\ttv_2,0)Y(\ttv_1,v)\mathbf{1}\\
&=Y(\ttv_2,0)\sfT_{v}\ttv_1\\
&=\sfT_{v}\sfT_{-v}Y(\ttv_2,0)\sfT_{v}\ttv_1\\
&=\sfT_{v} Y(\ttv_2,-v)\ttv_1\\
\end{split}
\end{equation*}
\end{proof}

As promised, covariance with respect to translations follows from the other defining properties of a vertex algebra:
\begin{theorem} Given a translation $v'$, $Y(\sfT_{v'}\ttv,v)=\sfT_{v'} Y(\ttv,v) \sfT_{-v'}$ 
\end{theorem}

\begin{proof} We need to prove that
$$
Y(\sfT_{v'}\ttv,v)=Y(\ttv,v+v')
$$
Let us check that both sides act on the same space. If $\ttv \in \lbz \bW$, $\sfT_{v'} \ttv \in {}_{|\bz|+v'} \bW$, so the operator on the left acts on $\bW^{|\bz|+v'+v}$, as clearly does the operator on the right. Let $\ttv_1 \in \bW^{|\bz|+v'+v}$. Then,
\begin{equation*}
\begin{split}
Y(\sfT_{v'}\ttv,v) \ttv_1 &= \sfT_vY(\ttv_1,-v)\sfT_{v'}\ttv\\
&=\sfT_vY(\ttv_1,-v)\sfT_{-v}\sfT_{v+v'}\ttv\\
&=Y(\ttv_1,0)Y(\ttv, v+v') \mathbf{1}\\
&=Y(\ttv,v+v')\ttv_1
\end{split}
\end{equation*}
\end{proof}

\begin{theorem}(\textup{\textbf{``Associativity''/Closure under composition.}}) Given states $\ttv_1,\ttv_2$ and translations $v_1,v_2$,
$$
Y(\ttv_1,v_1) \circ Y(\ttv_2,v_2)=Y(Y(\ttv_1,v_1-v_2)\ttv_2,v_2),
$$ 
\end{theorem}

\begin{proof} Yet again we use the strategy of the proof of Goddard Uniqueness Theorem.
\begin{equation*}
\begin{split}
Y(Y(\ttv_1,v_1-v_2)\ttv_2,v_2) \ttv_3 &= Y(Y(\ttv_1,v_1-v_2)\ttv_2,v_2)Y(\ttv_3,0) \mathbf{1}\\
&= Y(\ttv_3,0) Y(Y(\ttv_1,v_1-v_2)\ttv_2,v_2) \mathbf{1}\\
&= Y(\ttv_3,0) \sfT_{v_2 }Y(\ttv_1,v_1-v_2)\ttv_2\\
&=  \sfT_{v_2 } \sfT_{-v_2 }Y(\ttv_3,0) \sfT_{v_2 }Y(\ttv_1,v_1-v_2)\ttv_2\\
&=  \sfT_{v_2 } Y(\ttv_3,-v_2)Y(\ttv_1,v_1-v_2)\ttv_2\\
&=  \sfT_{v_2 } Y(\ttv_1,v_1-v_2) Y(\ttv_3,-v_2) \ttv_2\\
&=  \sfT_{v_2 } Y(\ttv_1,v_1-v_2)\sfT_{-v_2}Y(\ttv_2,v_2) \ttv_3\\
&=  Y(\ttv_1,v_1)Y(\ttv_2,v_2) \ttv_3\\
\end{split}
\end{equation*}
\end{proof}

There are infinitesimal versions of some of the above definitions/results, of which we give an example below.  Given a translation vector $v$, it can also be regarded as a regular vector field on $\bbP$  vanishing only at $P$; $\bL^v$ denotes the corresponding operator in the Virasoro algebra. Since $v$ is everywhere regular, this is just the Lie derivative on meromorphic forms extended to the symmetric algebra. Let us carefully work out its relation with $\sfT_v$. Given a one-form $\alpha$, $\sfT_v \alpha (z) = \alpha (z-v)$, where we use the canonical trivialisation of $K$ on $\bbP \setminus P$. For any regular point $z$ of $\alpha$
$$
\{{\frac{d}{ds}}\big|_{s=0} \sfT_{sv} \alpha \} (z) \underset{definition}= \frac{d}{ds}\big|_{s=0} \alpha(z-sv)=-\cL_v \alpha = -\bL^{v} \alpha
$$
If a coordinate $u$ is chosen, the vector field $-d/du$ is one such, and the corresponding operator is $L_{-1}$. Thus $L_{-1} (\alpha) = \bL^{-d/du}$ and formally
$$
\sfT_v = \exp{(u(v)L_{-1})}
$$
We also have an action of $\bL^v$ on ${}_O\bW$, although of course this action does not integrate to an action of translations. 

When we fix a point $O \ne P$, we have a second, entirely canonical, vector field vanishing at both $P$ and $Q$ (and with residue at $P$ equal to $1$); in terms of a coordinate adapted to $P$ and $O$, this is $-u(d/du)$. This is the operator $L_0$. Given a one-form $\alpha=a(u)du$,
we have
$$
L_{-1} \alpha=-a'(u)du, \ \ L_0 \alpha=-ua'(u)du
$$

\begin{proposition}
$$
Y(L_{-1}(\ttv),v)=\frac{d}{du}Y(\ttv,v(u))
$$
\end{proposition}

\subsection{Vertex algebra structures}

From now on we will assume that an identification of the group of translations with $\bbC$ is chosen. Equivalently, we choose a trivialisation $du$ of $K$ on $\bbP \setminus P$ or a coordinate $u$ up to addition of a constant. We will write $e(z)=\he(z) du_z$, etc. For any positive integer $l$,  we set
$$
\he^{[l]} (z) = [\frac{d}{du}]^{l-1} \he(z)
$$
We set $\he^{[0]} (z)=Identity$. (In the notation of \S \ref{subsection: Derivatives}, the derivative $\frac{d}{du}$ would be denoted $\bcL_{\frac{d}{du}}$.)

 We will let $\ba$ denote a sequence $(a^1,\dots,a^n)$, where each $a^j$ in turn is an infinite sequence of integers
$$
a^j=(a^j_1 \ge a^j_2 \ge \dots )
$$
with all but finitely many entries 0. We let $| a^j|$ denote number of non-zero entries in the sequence $a^j$, and set $|\ba|=\sum_j | a^j|$.

\begin{proposition} The symmetric algebra $\lbz\bW$ over $\lbz \bK$ has a basis labelled by sequences $\ba$, consisting of vectors
$$
\ttv^e_{\ba} = \he^{[a^1]} (z_1) \circ \dots \circ  \he^{[a^n]} (z_n) \mathbf{1}
$$
where 
$$
\he^{[a^j]} (z_j) = \he^{[a^j_1]} (z_j) \circ \dots \circ \he^{[a^j_2]} (z_j) \dots
$$
\end{proposition}

\begin{proof} Note that $\lbz \bK$ has a basis  given by the vectors
$$
\he^{[l]} (z) \mathbf{1} = (l-1)! d_u(u-u(z))^{-l}, \ l  \ge 1, \ z \in |\bz|
$$
where $d_u$ signifies taking the exterior derivative with respect to the variable $u$.
Therefore a basis of $S^p \lbz \bK$ is given by
$$
\he^{[l_1]} (y_1) \dots \he^{[l_p]} (y_p) \mathbf{1}
$$
where the points $y_1, \dots, y_p$  belong to $|\bz|$ but are not necessarily distinct, and $l_p \ge 1$. Grouping the elements of the symmetric product into factors corresponding to distinct points $z_1,\dots, z_n$ yields the basis vectors $\ttv^e_\ba$, where $p=|\ba|$.
\end{proof}

Set $b(z)=\hb(z) du_z$. For any positive integer $l$,  set
$$
\hb^{[l]} (z) = [\frac{d}{du(z)}]^{l-1} \hb(z)
$$
and $\hb^{[0]} (z)=Identity$.  Given a sequence of integers
$$
a^j=(a^j_1 \ge a^j_2 \ge \dots )
$$
as above with all but finitely many entries 0, set
$$
\hb^{[a^j]} (z_j) = \bdots \hb^{[a^j_1]} (z_j) \circ \dots \circ \hb^{[a^j_2]} (z_j)\bdots  \dots
$$
where $\bdots \bdots $ denotes the renormalised product. Then

\begin{proposition} The symmetric algebra $\lbz\bW$ over $\lbz \bK$ has a basis labelled by sequences $\ba$, consisting of vectors
$$
\ttv^b_{\ba} = \hb^{[a^1]} (z_1) \circ \dots \circ  \hb^{[a^n]} (z_n) \mathbf{1}
$$
\end{proposition}

\begin{proof} Clearly the vector $\ttv^b_{a}$ belongs to $\underset{p \le |\ba|}\oplus S^p \lbz \bK$, where $p=|\ba|$. One argues by by induction that the vectors $\{\ttv^b_{\ba}||\ba| \le p\}$ form a basis of  $\underset{p \le |\ba|}\oplus S^p \lbz \bK$.
\end{proof}

We can immediately define a trivial (``commutative'') vertex algebra structure by setting
$$
Y(\ttv^e_{\hmu},v)=\he^{[a^1]}(z_1+v) \circ \dots \circ \he^{[a^n]}(z_n+v) 
$$
A \emph{second} vertex algebra structure is obtained by setting
$$
Y'(\ttv^b_{\hmu},v)=\ \ \bdots \hb^{[a^1]}(z_1+v)\bdots  \circ \dots \circ \bdots \hb^{[a^n]}(z_n+v)\bdots 
$$
 
\pagebreak

\part{\textbf{{\Large Current algebra on the projective line}}}

\section{\textbf{Current algebra: the fields}}\label{section: Current algebra: the fields}

\subsection{Current algebra : preliminaries}

Let $\bg$ be a complex lie algebra. We will assume given a nondegenerate symmetric bilinear form $(,)$ that is invariant, i.e., satisfying $([v,v_1], v_2)+ (v_1, [v,v_2]) =0\ \forall \ v,v_1,v_2 \in \bg$. This forces $\bg$ to be reductive. If it is simple,  such a bilinear form is upto a normalisation (to be fixed later) the Killing form.

For later use, we fix a complex antilinear involution $v \mapsto -v^{\dagger}$ such that the bilinear form restricted to the corresponding real form is real-valued and negative-definite. We fix a commuting linear involution $v \mapsto -v^t$ so that the composite antilinear involution  $v \mapsto (v^t)^\dagger$ defines a split real form of  $\bg$.

Given a (finite-dimensional) complex representation of  $\bg$ on a vector space $W$:
\begin{equation*}
\bg \times W \ni (v,\ttw) \mapsto v(\ttw) \in W
\end{equation*}
the dual representation on the dual space $W'$ is defined by
\begin{equation*}
v(\ttw') [\ttw] = -\ttw'[v(\ttw)]
\end{equation*}
A necessary and sufficient condition for $W$ and $W'$ to be equivalent representations is the existence of a nondegenerate bilinear form $\langle , \rangle$ on $W$ such that $\langle v(\ttw_1), \ttw_2 \rangle+\langle \ttw_1, v(\ttw_2) \rangle=0$. One can instead ask for a nondegenerate bilinear form (which we again denote) $\langle , \rangle$ on $W$ such that $\langle v(\ttw_1), \ttw_2 \rangle=\langle \ttw_1, v^t(\ttw_2) \rangle=0$. Such a form -- in fact one that is symmetric -- exists (and is unique up to normalisation if $W$ is irreducible) by the theory of highest/lowest weight representations. This theory also guarantees that representations are \emph{unitary} -- in other words, there exists a nondegenerate hermitian form $\langle , \rangle_h$ on $W$ such that $\langle v(\ttw_1), \ttw_2 \rangle_h=\langle \ttw_1, v^\dagger(\ttw_2) \rangle_h$. 

\subsection{Current algebra : the basic fields}

Let $\fg \equiv \bg \otimes \cK$. This is the Lie algebra of meromorphic maps $\balpha:\bbP \to \bg$. 

Fix a base-point $P$. We set $\fg_P \equiv \bg \otimes \cO_P$. This is the Lie algebra of meromorphic maps $\balpha:\bbP \to \bg$ such that  $\balpha$ is regular at $P$.  Consider also the Lie algebra $\fg_{P,-} \equiv \bg \otimes \mathfrak{m}_P$. This is the Lie algebra of meromorphic maps $\balpha:\bbP \to \bg$ such that
\begin{itemize}
  \item $\balpha$ is regular at $P$, and
  \item $\balpha(P)=0$
\end{itemize}
\emph{When no confusion is likely, we set $\fg_- \equiv \fg_{P,-}$, etc.}
Let $U\fg_-$ denote the corresponding universal enveloping algebra. (We will follow this convention for other Lie algebras as well.)  The Lie algebra $\bg$ is contained in $\fg_P$ (as the subalgebra of constant maps to $\bg$), and we have
\begin{equation*}
\fg_{P} = \fg_- \oplus \bg
\end{equation*}
Note that $\fg_-$ is an ideal, and there is therefore an (``adjoint") action of $\bg$ on $\fg_-$, inducing one (by algebra automorphisms, preserving the filtration by order) on $U\fg_-$ as well.

We will now define fields that act on $U\fg_-$. We set $\mathbf{1}=1 \in \bbC \subset U\fg_-$.

Given $\nu \in  \fg$ and $z \in \bbP\setminus \{P,\text{poles of $\nu$}\}$, we define $\epsilon_z^{\nu(z)} \in \fg_- \otimes K_z$ by
\begin{equation*}
    \epsilon_z^{\nu(z)} = \frac{\nu(z)}{u-u(z)} \otimes du_z
\end{equation*}
and we define $\epsilon^{\nu(z)}(z):U\fg_-  \to U\fg_- \otimes K_z$ to be multiplication on the left by $\epsilon_z^{\nu(z)}$ (and ``commuting $du_z$ to the right"). 

Given $z \in \bbP \setminus  P$, set $\fg_-^z=\bg \otimes \mathfrak{m}_P^z$. (Recall that $\mathfrak{m}_P^z=\mathfrak{m}_P \cap \cO_z$.) Given $\nu \in \fg$ and $z \in \bbP\setminus \{P,\text{poles of $\nu$}\}$, we wish to define define a map $\iota^{\nu(z)}(z): U\fg_-^z \to  U\fg_-^z \otimes K_z$. We do this as follows.  Define a map $\iota^{\nu(z)}(z)$ at the tensor algebra  level inductively by setting $\iota^{\nu(z)}(z)(\mathbf{1})=0$ and requiring
\begin{equation*}
\begin{split}
\iota^{\nu(z)}(z) (\balpha \otimes \bbeta) - \balpha \otimes \iota^{\nu(z)}(z) \bbeta
&=-[\epsilon^{\nu(z)}(z), \balpha] \otimes \bbeta-(\nu(z), d\balpha(z)) \bbeta\\
&+ \iota^{[\nu(z), \balpha(z)]}(z) \bbeta +\epsilon^{[\nu(z), \balpha(z)]}(z) \bbeta
\end{split}
\end{equation*}
where $\balpha \in \fg_-^z$ and $\bbeta$ is an arbitrary element of the tensor algebra over $\fg_-^z$. In other words,
\begin{equation}\label{iotaalphacommute}
\begin{split}
\iota^{\nu(z)}(z) \circ \balpha  - \balpha \circ \iota^{\nu(z)}(z)&=-[\epsilon^{\nu(z)}(z), \balpha]-({\nu(z)}, d\balpha(z))\\
&+ \iota^{[{\nu(z)}, \balpha(z)]}(z)++\epsilon^{[\nu(z), \balpha(z)]}(z)
\end{split}
\end{equation}
 
We claim that $\iota^{\nu(z)}(z)$ descends to a map $U\fg_-^z \to  U\fg_-^z \otimes K_z$. Consider, in the tensor algebra,  the commutator $[\iota^{\nu(z)}(z), \balpha \circ \bbeta]$. We have, by induction on degree,
\begin{equation*}
\begin{split}
[\iota^{\nu(z)}(z), \balpha \circ \bbeta] &= [\iota^{\nu(z)}(z), \balpha] \circ \bbeta + \balpha \circ [\iota^{\nu(z)}(z), \bbeta]\\
&=-[\epsilon^{\nu(z)}(z), \balpha] \circ \bbeta-({\nu(z)},d\balpha(z)) \circ \bbeta\\ 
&+ \iota^{[{\nu(z)}, \balpha(z)]}(z) \circ \bbeta +\epsilon^{[{\nu(z)}, \balpha(z)]}(z) \circ \bbeta\\
&-\balpha \circ [\epsilon^{\nu(z)}(z), \bbeta]- \balpha \circ ({\nu(z)},d\bbeta(z))\\ &+ \balpha \circ \iota^{[{\nu(z)}, \bbeta(z)]}(z) +\balpha \circ\epsilon^{[{\nu(z)}, \bbeta(z)]}(z)\\
 \end{split}
\end{equation*}
which yields
\begin{equation*}
\begin{split}
[\iota^{\nu(z)}(z), \balpha \circ \bbeta - \bbeta \circ \balpha] & =
-[\epsilon^{\nu(z)}(z), \balpha] \circ \bbeta-({\nu(z)},d\balpha(z)) \circ \bbeta\\
&+ \iota^{[{\nu(z)}, \balpha(z)]}(z) \circ \bbeta
+\epsilon^{[{\nu(z)}, \balpha(z)]}(z) \circ \bbeta\\
&-\balpha \circ [\epsilon^{\nu(z)}(z), \bbeta]- \balpha \circ ({\nu(z)},d\bbeta(z))\\ 
&+ \balpha \circ \iota^{[{\nu(z)}, \bbeta(z)]}(z)
+\balpha \circ\epsilon^{[{\nu(z)}, \bbeta(z)]}(z)\\
&+[\epsilon^{\nu(z)}(z), \bbeta] \circ \balpha+({\nu(z)},d\bbeta(z)) \circ \balpha\\
&-\iota^{[{\nu(z)}, \bbeta(z)]}(z) \circ \balpha
-\epsilon^{[{\nu(z)}, \bbeta(z)]}(z) \circ \balpha\\
&+\bbeta \circ [\epsilon^{\nu(z)}(z), \balpha] + \bbeta \circ ({\nu(z)},d\balpha(z))\\
&- \bbeta \circ \iota^{[{\nu(z)}, \balpha(z)]}(z)
- \bbeta \circ\epsilon^{[{\nu(z)}, \balpha(z)]}(z)\\
&=-[\epsilon^{\nu(z)}(z),[\balpha,\bbeta]]\\
&+ [\iota^{[{\nu(z)}, \balpha(z)]}(z),\bbeta]
+[\epsilon^{[{\nu(z)}, \balpha(z)]}(z),\bbeta]\\
&-[\iota^{[{\nu(z)}, \bbeta(z)]}(z),\balpha]
-[\epsilon^{[{\nu(z)}, \bbeta(z)]}(z) , \balpha]\\
&=-[\epsilon^{\nu(z)}(z),[\balpha,\bbeta]]\\
 &-([{\nu(z)}, \balpha(z)],d\bbeta(z))+\iota^{[{\nu(z)}, \balpha(z)],\bbeta(z)]}(z)+\epsilon^{[{\nu(z)}, \balpha(z)],\bbeta(z)]}(z)\\
&+([{\nu(z)}, \bbeta(z)],d\balpha(z))-\iota^{[{\nu(z)}, \bbeta(z)],\balpha(z)]}(z)-\epsilon^{[{\nu(z)}, \bbeta(z)],\balpha(z)]}(z)\\
&=-[\epsilon^{\nu(z)}(z),[\balpha,\bbeta]]-(\nu(z),d[\balpha,\bbeta](z))\\
&+\iota^{[\nu(z),[\balpha,\bbeta](z)]}(z)++\epsilon^{[\nu(z),[\balpha,\bbeta](z)]}(z)\\
&=[\iota^{\nu(z)}(z),[\balpha,\bbeta]]
\end{split}
\end{equation*}
and the claim is proved.

Having defined the fields $\epsilon^{\nu(z)}(z)$ and $\iota^{\nu(z)}(z)$, we add them to define the ``current":
\begin{equation*}
\boxed{j^{\nu(z)}(z)=\epsilon^{\nu(z)}(z) + \iota^{\nu(z)}(z):U\fg_-^z \to U\fg_- \otimes K_z}
\end{equation*}

The condition (\ref{iotaalphacommute}) can be rewritten in terms of $j$:
\begin{equation}\label{jalphacommute}
[j^{\nu(z)}(z),\balpha]=-({\nu(z)},d\balpha(z))+j^{[{\nu(z)},\balpha(z)]} (z)
\end{equation}

\subsection{Locality, OPE of the current}

\begin{notation} We will often consider \emph{constant} functions $\nu$ taking the value $v\in \bg$. In this case we write $j^\nu=j^v$.
\end{notation}

We now compute the commutators of the fields $\epsilon$ and $\iota$ at distinct points, and then put together the resulting formulae to study $j$.

In what follows, operators have to be composed on spaces $U\fg_-^{z_1,z_2}$ defined along the lines of the \S \ref{subsection: OPEetc.}. We write $du_{z_1} du_{z_2}= du_{z_2} du_{z_1} = du_{z_1} \otimes ^s du_{z_2}$, where the notation in the last expression is also explained in \S \ref{subsection: OPEetc.}. The next lemma follows from the definitions.

\begin{lemma} Let $z_1,z_2$ be distinct points on $\bbP$ and $v_1,v_2 \ \in \bg$. Then
\begin{equation*}
\begin{split}
&\epsilon^{v_1} (z_1)  \epsilon^{v_2} (z_2) - \epsilon^{v_2} (z_2)  \epsilon^{v_1} (z_1)\\ = &\frac {1}{u(z_1)-u(z_2)} (\epsilon^{[v_1,v_2]} (z_1) du_{z_2} - \epsilon^{[v_1,v_2]} (z_2)du_{z_1})
\end{split}
\end{equation*}
and
\begin{equation}\label{iotaepsiloncommute}
\begin{split}
\iota^{v_1} (z_1)  \epsilon^{v_2} (z_2) - \epsilon^{v_2} (z_2)  \iota^{v_1} (z_1)&=\frac{(v_1,v_2)du_{z_1}du_{z_2}}{(u(z_1)-u(z_2))^2}\\
+\frac{1}{u(z_1)-u(z_2)} \{\iota^{[v_1,v_2]}(z_1)du_{z_2}&+ \epsilon^{[v_1,v_2]}(z_2)du_{z_1}\}
\end{split}
\end{equation}
\end{lemma}

\begin{proposition} Let $z_1,z_2$ be distinct points on $\bbP$ and $v_1,v_2 \ \in \bg$. The current operators $j^{v_1}(z_1)$ and $j^{v_2}(z_2)$ (can be composed and) commute.
As $z_2 \to z_1$ we have
\begin{equation*}
\begin{split}
j^{v_2}(z_2)j^{v_1}(z_1)&+\frac{(v_1,v_2)du_{z_1}du_{z_2}}{(u(z_1)-u(z_2))^2}-\frac{1}{u(z_2)-u(z_1)} du_{z_2}j^{[v_2,v_1]}(z_1)\\
\longrightarrow &\iota^{v_2}(z_1)\iota^{v_1}(z_1)
+\epsilon^{v_2} (z_1)  \epsilon^{v_1} (z_1)\\
+&\epsilon^{v_2} (z_1)  \iota^{v_1} (z_1)
+\epsilon^{v_1} (z_1)  \iota^{v_2} (z_1)\\
+& [d\hiota^{[v_2,v_1]}/dz](z_1)du_{z_1}^2\\
\end{split}
\end{equation*}
\end{proposition}

\begin{proof} Let $\balpha \in \fg_-^{z_1,z_2}$. Using (\ref{jalphacommute}), we get
\begin{equation*}
\begin{split}
[j^{v_1}(z_1)j^{v_2}(z_2),\balpha]&=[j^{v_1}(z_1),\balpha]j^{v_2}(z_2)
+ j^{v_1}(z_1)[j^{v_2}(z_2),\balpha]\\
&=[-(v_1,d\balpha(z_1))+j^{[v_1,\balpha(z_1)]}(z_1)] j^{v_2}(z_2)\\
&+j^{v_1}(z_1) [-(v_2,d\balpha(z_2))+j^{[v_2,\balpha(z_2)]}(z_2)]
\end{split}
\end{equation*}
which yields
\begin{equation*}
\begin{split}
[j^{v_1}(z_1)j^{v_2}(z_2)-j^{v_2}(z_2)j^{v_1}(z_1),\balpha]
&=[j^{[v_1,\balpha(z_1)]}(z_1),j^{v_2}(z_2)]+[j^{v_1}(z_1),  
j^{[v_2,\balpha(z_2)]}(z_2)]\\
\end{split}
\end{equation*}
and we see the desired commutativity by induction. the starting point of the induction remains to be checked; we postpone this to a later section (\S \ref{insertions}), where this will be done in a more general context.

To study the limit as $z_2 \to z_1$ of $j^{v_2}(z_2)j^{v_1}(z_1)$, write  
\begin{equation*}
\begin{split}
j^{v_2}(z_2)j^{v_1}(z_1)&=\iota^{v_2}(z_2)\iota^{v_1}(z_1)
+\epsilon^{v_2} (z_2)  \epsilon^{v_1} (z_1)\\
&+\epsilon^{v_2} (z_2)  \iota^{v_1} (z_1)
+\epsilon^{v_1} (z_1)  \iota^{v_2} (z_2)\\
&+\iota^{v_2} (z_2)  \epsilon^{v_1} (z_1)-\epsilon^{v_1} (z_1)  \iota^{v_2} (z_2)\\
\end{split}
\end{equation*}
and use the previous lemma to get
\begin{equation*}
\begin{split}
j^{v_2}(z_2)j^{v_1}(z_1)&=\iota^{v_2}(z_2)\iota^{v_1}(z_1)
+\epsilon^{v_2} (z_2)  \epsilon^{v_1} (z_1)\\
&+\epsilon^{v_2} (z_2)  \iota^{v_1} (z_1)
+\epsilon^{v_1} (z_1)  \iota^{v_2} (z_2)\\
&-\frac{(v_1,v_2)du_{z_1}du_{z_2}}{(u(z_1)-u(z_2))^2}\\
&+\frac{1}{u(z_2)-u(z_1)} \iota^{[v_2,v_1]}(z_2)du_{z_1}+\frac{1}{u(z_2)-u(z_1)} du_{z_2} \epsilon^{[v_2,v_1]}(z_1)\\
&=\iota^{v_2}(z_2)\iota^{v_1}(z_1)
+\epsilon^{v_2} (z_2)  \epsilon^{v_1} (z_1)\\
&+\epsilon^{v_2} (z_2)  \iota^{v_1} (z_1)
+\epsilon^{v_1} (z_1)  \iota^{v_2} (z_2)\\
&-\frac{(v_1,v_2)du_{z_1}du_{z_2}}{(u(z_1)-u(z_2))^2}\\
&+\frac{1}{u(z_2)-u(z_1)} (\iota^{[v_2,v_1]}(z_2)du_{z_1}-du_{z_2} \iota^{[v_2,v_1]}(z_1))\\
&+\frac{1}{u(z_2)-u(z_1)} (du_{z_2} \epsilon^{[v_2,v_1]}(z_1) +du_{z_2} \iota^{[v_2,v_1]}(z_1))\\
\end{split}
\end{equation*}
Gathering terms, we find:
\begin{equation*}
\begin{split}
j^{v_2}(z_2)j^{v_1}(z_1)&=-\frac{(v_1,v_2)du_{z_1}du_{z_2}}{(u(z_1)-u(z_2))^2}\\
&+\frac{1}{u(z_2)-u(z_1)} du_{z_2}j^{[v_2,v_1]}(z_1)\\
&+\iota^{v_2}(z_2)\iota^{v_1}(z_1)
+\epsilon^{v_2} (z_2)  \epsilon^{v_1} (z_1)\\
&+\epsilon^{v_2} (z_2)  \iota^{v_1} (z_1)
+\epsilon^{v_1} (z_1)  \iota^{v_2} (z_2)\\
&+\frac{1}{u(z_2)-u(z_1)} (\iota^{[v_2,v_1]}(z_2)du_{z_1}-du_{z_2} \iota^{[v_2,v_1]}(z_1))\\
\end{split}
\end{equation*}
and the Proposition follows.
\end{proof}

We note the following consequences:
\begin{corollary} Let $z_1,z_2$ be distinct points on $\bbP$ and $v_1,v_2 \ \in \bg$. Then
\begin{equation}\label{iicommute}
\begin{split}
\iota^{v_1}(z_1)\iota^{v_2}(z_2)-\iota^{v_2}(z_2)\iota^{v_1}(z_1)&+\\
\frac{1}{u(z_1)-u(z_2)}(\iota^{[v_1,v_2]}(z_1)du_{z_2}-du_{z_1} \iota^{[v_1,v_2]}(z_2)) &=0
\end{split}
\end{equation}
and
\begin{equation}\label{iotajcommute}
\begin{split}
\iota^{v_1} (z_1)  j^{v_2} (z_2) - j^{v_2} (z_2)  \iota^{v_1} (z_1)&=\frac{(v_1,v_2)du_{z_1}du_{z_2}}{(u(z_1)-u(z_2))^2}\\
+\frac{du_{z_1}}{u(z_1)-u(z_2)}j^{[v_1,v_2]}(z_2)\\
\end{split}
\end{equation}
\end{corollary}

\subsection{Extending to the point at infinity}\label{globalcurrent}

\begin{remark}
By the PBW theorem, we have an isomorphism of left $U\fg_-$ modules
\begin{equation*}
U\fg_{P} = U\fg_- \oplus \ \ U\fg_-
\otimes_{\mathbb C} U^+\bg
\end{equation*}
where $U^+\bg$ is the augmentation ideal in $U\bg$. This gives natural isomorphisms of left $\fg_-$-modules:
\begin{equation*}
U\fg_- = 
U\fg_{P}/<U^+\bg>=U\fg\otimes_{U\bg} \bbC
\end{equation*}
where we mean by $<U^+\bg>$ the left ideal in $U\fg_{P}$ generated by $U^+\bg$. This is an also an isomorphism of (left) $\fg$-modules, where the action of $\bg$ on the left hand side is the adjoint action. Note that the subspace $U\bg/U^+\bg$ in the right hand side  corresponds to the ``vacuum sector'' $\bbC$ on the left.
\end{remark}

\noindent\textbf{Notation:} We set $\cV=U\fg\otimes_{U\bg} \bbC$ and $\cV^z=U\fg^z\otimes_{U\bg} \bbC$. Recall that $\fg$ is the Lie algebra of all meromorphic maps $\bbP \to \bg$; $\fg^z$ is the Lie subalgebra of maps regular at $z$.

\medskip

In this section, we will show how the current $j^{\nu}$ extends as a well-defined field 
\begin{equation*}
j^{\nu(z)}(z): \cV^z \to \cV \otimes K_z,\ \ z \in \bbP
\end{equation*}

To begin with, choose a second point $\tP \ne P$. This determines a coordinate $u$ (up to a nonzero scalar) such that $u(\tP)=0$ and $u(P)=\infty$; set $\tu=1/u$. Define fields $\tepsilon$ and $\tiota$ analogously to $\epsilon$ and $\iota$ by replacing the role of $P$ by $\tP$. Thus,
\begin{equation*}
    \tepsilon^{\nu(z)}_z = \frac{\nu(z) \otimes d\tu_z}{\tu-\tu(z)} 
\end{equation*}
and the field $\tepsilon^{\nu(z)}(z):U\fg_{\tP,-}  \to U\fg_{\tP,-} \otimes K_z$ acts by multiplication on the left by $\tepsilon^{\nu(z)}_z$.

On the other hand, $\tiota$ is defined by requiring that $\tiota^{\nu(z)}(z)$ annihilates $\mathbf{1} \in U\fg_{\tP,-}$ and
 \begin{equation*}
\tiota^{\nu(z)}(z) \circ \balpha  - \balpha \circ \tiota^{\nu(z)}(z)=-(\nu(z), d\balpha(z))
+ \tiota^{[v, \balpha(z)]}(z)
- [\tepsilon^{\nu(z)},\balpha-\balpha(z)] 
\end{equation*}
Finally, $\tj$ is defined as the sum $\tj^{\nu(z)}(z)=\tepsilon^{\nu(z)}(z)+\tiota^{\nu(z)}(z)$.

\begin{theorem} Under the identifications
$U\fg_P/<U^+\bg> = U\fg_{P,-}$, $U\fg_{\tP}/<U^+\bg> = U\fg_{\tP,-}$, we have on $\bbP \setminus \{P,\tP\}$, the equalities
\begin{equation*}
\begin{split}
 \tepsilon(z)^{\nu(z)} &= \epsilon^{\nu(z)}(z)-{\nu(z)}\frac{d\tu_z}{\tu(z)}:U\fg_{P,\tP}/<U^+\bg> \to \{U\fg_{P,\tP}/<U^+\bg>\} \otimes K_z\\
 \tiota(z)^{\nu(z)} &= \iota^{\nu(z)}(z) + {\nu(z)}\frac{d\tu_z}{\tu(z)} :U\fg^z_{P,\tP}/<U^+\bg> \to \{U\fg^z_{P,\tP}/<U^+\bg>\} \otimes K_z
\end{split}
\end{equation*}
In particular, $\tj^{\nu(z)}(z)=j^{\nu(z)}(z)$.
\end{theorem}

\begin{proof} One checks that $\iota^{\nu(z)}(z) + {\nu(z)}\frac{d\tu_z}{\tu(z)}$ satisfies properties demanded of $\tiota^{\nu(z)}(z)$. The key computation:  
\begin{equation*}
\begin{split}
&\{\iota^{\nu(z)}(z) + {\nu(z)}\frac{d\tu_z}{\tu(z)}\} \circ \balpha  - \balpha \circ \{\iota^{\nu(z)}(z) + {\nu(z)}\frac{d\tu_z}{\tu(z)}\}\\
&=-({\nu(z)}, d\balpha(z))
+ \iota^{[{\nu(z)}, \balpha(z)]}(z)+\epsilon^{[{\nu(z)}, \balpha(z)]}(z)-[\epsilon^{\nu(z)}(z),\balpha]+[{\nu(z)},\balpha] \frac{d\tu_z}{\tu(z)}\\
&=-({\nu(z)}, d\balpha(z))
+ \iota^{[{\nu(z)}, \balpha(z)]}(z)+\epsilon^{[{\nu(z)}, \balpha(z)]}(z)+[{\nu(z)}, \balpha(z)] \frac{d\tu_z}{\tu(z)}\\
&- [{\nu(z)}, \balpha(z)] \frac{d\tu_z}{\tu(z)}-[\epsilon^{\nu(z)}(z),\balpha]+[{\nu(z)},\balpha] \frac{d\tu_z}{\tu(z)}\\
&=-({\nu(z)}, d\balpha(z))\\
&+\iota^{[{\nu(z)}, \balpha(z)]}(z)+[{\nu(z)}, \balpha(z)] \frac{d\tu_z}{\tu(z)}\\
&+\epsilon^{[{\nu(z)}, \balpha(z)]}(z)-[{\nu(z)}, \balpha(z)] \frac{d\tu_z}{\tu(z)}\\
&-[\epsilon^{\nu(z)}(z)-{\nu(z)}\frac{d\tu_z}{\tu(z)},\balpha] \\
&=-({\nu(z)}, d\balpha(z))\\
&+\iota^{[{\nu(z)}, \balpha(z)]}(z)+[{\nu(z)}, \balpha(z)] \frac{d\tu_z}{\tu(z)}\\
&+\tepsilon^{[{\nu(z)}, \balpha(z)]}(z)\\
&-[\tepsilon^{\nu(z)}(z),\balpha] \\
\end{split}
\end{equation*}
\end{proof}

We record the remarkable consequence of the previous theorem:

\begin{theorem} We have a natural ($Aut(\bbP)$-covariant) field on $\bbP$
\begin{equation*}
j^{\nu(z)}(z)  :\cV^z \to \cV \otimes K_z
\end{equation*}
satisfying the OPE
\begin{equation}\label{OPEcurrent}
\begin{split}
j^{v_2}(z_2)j^{v_1}(z_1)&+\frac{(v_1,v_2)du_{z_1}du_{z_2}}{(u(z_1)-u(z_2))^2}-\frac{1}{u(z_2)-u(z_1)} du_{z_2}j^{[v_2,v_1]}(z_1)\\
\longrightarrow &\iota^{v_2}(z_1)\iota^{v_1}(z_1)
+\epsilon^{v_2} (z_1)  \epsilon^{v_1} (z_1)\\
+&\epsilon^{v_2} (z_1)  \iota^{v_1} (z_1)
+\epsilon^{v_1} (z_1)  \iota^{v_2} (z_1)\\
+& [d\hiota^{[v_2,v_1]}/dz](z_1)du_{z_1}^2\\
\end{split}
\end{equation}
where the limit can be calculated with respect to any choice of point at infinity $P \ne z_1$ (and is independent of this choice).
\end{theorem}

\subsection{$n$-point functions}

As in the case of the chiral boson, we can compute the $n$-point functions $<j^{v_1}(z_1) \dots j^{v_n}(z_n)>$ directly, though the arguments are slightly more involved. The $n$-point functions are defined by:
\begin{equation*}
<j^{v_1}(z_1) \dots j^{v_n}(z_n)> \mathbf{1}=
\underset{\cV \to \ <\text{span of} \ \mathbf{1}>}{Projection}
\{j^{v_1}(z_1) \dots j^{v_n}(z_n) \mathbf{1}\}
\end{equation*}

We start with a series of lemmas.

\begin{lemma}
The one-point function vanishes: $<j^{\nu(z)}(z)>=0$. Given $z_1 \ne z_2$, and vectors $v_1,v_2 \in \bg$,
\begin{equation*}
<j^{v_1}(z_1) j^{v_2}(z_2)> = \frac{(v_1,v_2)du_{z_1}du_{z_2}}{(u(z_1)-u(z_2))^2}
\end{equation*}
\end{lemma}

\begin{proof} The first equation is clear enough. As for the second, we have 
$<j^{v_1}(z_1) j^{v_2}(z_2)> = <\iota^{v_1}(z_1) \epsilon^{v_2}(z_2)>$; now use (\ref{iotaepsiloncommute}).
\end{proof}

\begin{lemma} Given distinct points $z_1,z_2,z_3$, and vectors $v_1,v_2, v_3 \in \bg$,
\begin{equation*}
<j^{v_1}(z_1) j^{v_2}(z_2) j^{v_3}(z_3)> 
=\frac{(v_1, [v_2,v_3])}{(u(z_1)-u(z_2))(u(z_1)-u(z_3))(u(z_2)-u(z_3))} du_{z_1}d_{z_2}du_{z_3}
\end{equation*}
\end{lemma}

\begin{proof} We adopt the notation $j(z)=\hj(z)du_z$, etc. and make repeated use of (\ref{iotaepsiloncommute}):
\begin{equation*}
\begin{split}
 <\hj^{v_1}(z_1) \hj^{v_2}(z_2) \hj^{v_3}(z_3)> &= <\hiota^{v_1}(z_1) \hj^{v_2}(z_2) \he^{v_3}(z_3)>\\ 
&=<\hiota^{v_1}(z_1) \hiota^{v_2}(z_2) \he^{v_3}(z_3)>+ 
<\hiota^{v_1}(z_1) \he^{v_2}(z_2) \he^{v_3}(z_3)>\\
&=<\hiota^{v_1}(z_1) \he^{[v_2,v_3]}(z_3)>\frac{1}{u(z_2)-u(z_3)}\\
&+  <\hiota^{[v_1,v_2]}(z_1) \he^{v_3}(z_3)>\frac{1}{u(z_1)-u(z_2)}\\
&= (v_1, [v_2,v_3]) \frac{1}{(u(z_2)-u(z_3))(u(z_1)-u(z_3))^2}\\
&+  ([v_1,v_2], v_3)\frac{1}{(u(z_1)-u(z_2))(u(z_1)-u(z_3))^2}\\
&= (v_1, [v_2,v_3])\{ \frac{1}{(u(z_2)-u(z_3))(u(z_1)-u(z_3))^2}\\
&\ \ \ \ \ \ \ \ \ \ \ \ \ \ \ \ + \frac{1}{(u(z_1)-u(z_2))(u(z_1)-u(z_3))^2}\}\\
&= \frac{(v_1, [v_2,v_3])}{(u(z_1)-u(z_3))^2}\{ \frac{1}{u(z_2)-u(z_3)}+ \frac{1}{u(z_1)-u(z_2)}\}\\
&=\frac{(v_1, [v_2,v_3])}{(u(z_1)-u(z_2))(u(z_1)-u(z_3))(u(z_2)-u(z_3))}
\end{split}
\end{equation*}
\end{proof}
One checks that this is indeed symmetric in the three indices.

\begin{theorem} Given distinct points $z_1, \dots , z_n$, and vectors $v_1,\dots ,v_n \in \bg$,
\begin{equation*}
\begin{split}
<j^{v_1}(z_1) \dots j^{v_n}(z_n)>&=\sum_{l=2}^n \{\frac{(v_1,v_l)du_{z_1} du_{z_l}}{(u(z_1)-u(z_l))^2}<j^{v_2}(z_2) \dots \widehat{j^{v_l}(z_l)} \dots  j^{v_n}(z_n)> \\
 &+\frac{du_{z_1}}{u(z_1)-u(z_l)}<j^{v_2}(z_2) \dots j^{[v_1,v_l]}(z_l) \dots  j^{v_n}(z_n)>\}\\
\end{split}
\end{equation*}
The hat $\widehat{....}$ denotes a term to be omitted.
\end{theorem}

\begin{proof} We make use of (\ref{iotajcommute}).  
\begin{equation*}
\begin{split}
 <\hj^{v_1}(z_1) \hj^{v_2}(z_2)  \dots  \hj^{v_n}(z_n)>  &=  <\hiota^{v_1}(z_1) \hj^{v_2}(z_2)  \dots  \hj^{v_n}(z_n)> \\
&\sum_{l=2}^n <\hj^{v_2}(z_2) \dots [\hiota^{v_1}(z_1),\hj^{v_l}(z_l)] \dots  \hj^{v_n}(z_n)> \\
&=\sum_{l=2}^n \{\frac{(v_1,v_l)}{(u(z_1)-u(z_l))^2}<\hj^{v_2}(z_2) \dots \widehat{\hj^{v_l}(z_l)} \dots  \hj^{v_n}(z_n)> \\
 &+\frac{1}{u(z_1)-u(z_l)}<\hj^{v_2}(z_2) \dots \hj^{[v_1,v_l]}(z_l) \dots  \hj^{v_n}(z_n)>\}\\
\end{split}
\end{equation*}
\end{proof}

We record the particular case of the four-point function.
\begin{corollary} Given distinct points $z_1,..,z_4$, and vectors $v_1,..,v_4 \in \bg$,
\begin{equation*}
\begin{split}
<j^{v_1}(z_1) j^{v_2}(z_2) j^{v_3}(z_3) j^{v_4}(z_4)> 
&=\{\frac{(v_1,v_2)(v_3,v_4)}{(u(z_1)-u(z_2))^2 (u(z_3)-u(z_4))^2}\\
&+\frac{(v_1,v_3)(v_2,v_4)}{(u(z_1)-u(z_3))^2 (u(z_2)-u(z_4))^2}\\
&+\frac{(v_1,v_4)(v_2,v_3)}{(u(z_1)-u(z_4))^2 (u(z_2)-u(z_3))^2}\\
&+\frac{([v_1,v_2], [v_3,v_4])}{(u(z_1)-u(z_2))(u(z_2)-u(z_3))(u(z_2)-u(z_4))(u(z_3)-u(z_4))}\}\\
&-\frac{([v_1,v_3], [v_2,v_4])}{(u(z_1)-u(z_3))(u(z_2)-u(z_3))(u(z_2)-u(z_4))(u(z_3)-u(z_4))}\}\\
&+\frac{([v_1,v_4], [v_2,v_3])}{(u(z_1)-u(z_4))(u(z_3)-u(z_4))(u(z_2)-u(z_4))(u(z_2)-u(z_3))}\}\\
& \times du_{z_1}du_{z_2}du_{z_3}du_{z_4}\\
\end{split}
\end{equation*}
\end{corollary}

\subsection{Current algebra modes}\label{currentmodes}

Suppose a coordinate $u$ is chosen, vanishing at a point which we denote $O$, and with a pole at $P$. 

Let $\lO U\fg_-=\bbC \oplus H^0(\bbP \setminus O, \bg \otimes \mathfrak{m}_P) \oplus S^2 H^0(\bbP \setminus O, \bg \otimes \mathfrak{m}_P) \oplus ......\ $, where $H^0(\bbP \setminus O, \bg \otimes \mathfrak{m}_P)$ is the ideal of \emph{algebraic} functions $\bbP \setminus O \to \bg$ vanishing at $P$.

For $v \in \bg$ and $l \ge 1$, define operators $j^{v}_{-l}: U\fg_- \to  U\fg_-$ by
\begin{equation*}
j^{v}_{-l}[\bbeta] =  vu^{-l} \times \bbeta, \ \bbeta \in \lO U\fg_-
\end{equation*}
where the multiplication is the universal enveloping algebra $U\fg_-$. Note that $j^{v}_l$ preserves $\lO U\fg_- \subset U\fg_-$ and in fact the operators $j^{v}_l$ acting on the vacuum generate $\lO U\fg_-$. For $v \in \bg$ and $l \ge 0$, define operators $j^{v}_{l}:  \lO U\fg_- \to  \lO U\fg_-$ inductively by setting $j^{v}_l(\mathbf{1})=0$ and
\begin{equation*}
[j^{v}_l,j^{v'}_{-m}]=
\begin{cases}
j^{[v,v']}_{l-m}, \ m \ne l \\
l(v,v') \ m =l \\
\end{cases}
\end{equation*}
 The operators $j^{v}_l$ are all defined on $\lO U\fg_-$ and satisfy the commutation relations $[j^{v}_l,j^{v'}_m]=l(v,v')\delta_{l+m,0} + j^{[v,v']}_{l+m}$. \emph{Note that in contrast to the case of the (abelian) current algebras encountered earlier (\S \ref{bosoncurrentmodes}), the operator $j_\bL^{v}$ is defined for $l=0$ as well, even in the absence of insertions.}

Let $z \in \bbP\setminus \{O,P\}$ and set $u(z)=w \in \bbC$.  
The series $\sum_{l \ge 1} j^{v}_{-l} w^{l-1}[\mathbf{1}]$ converges uniformly on the subsets $\{u||u|\ge |w|+\delta\}$ (for $\delta >0$) to $\frac{v}{(u-w)} \in \fg_-$. Formally
\begin{equation*}
\sum_{l \ge 1} j_{-l} w^{l-1}   =  \hepsilon^{v}(z)
\end{equation*}
Turning next to the formal sum $\sum_{l \ge 0} j^{v}_{l} w^{-l-1}$, this is clearly well-defined inductively on $\lO U\fg_-$. We have, for $\balpha$ of the form $v'u^{-m}$ ($v' \in \bg, \ m \ge 1$)
\begin{equation*}
\begin{split}
[\sum_{l \ge 0} j^{v}_{l} w^{-l-1},\balpha]&=[\sum_{l \ge 0} j^{v}_{l} w^{-l-1},j^{v'}_{-m}]\\
&=\sum_{l > m} j^{[v,v']}_{l-m} w^{-l-1}+m(v,v') w^{-m-1}+\sum_{0 \le l < m} j^{[v,v']}_{l-m} w^{-l-1}\\
&=\sum_{l \ge 1} j^{[v,v'w^{-m}]}_{l} w^{-l-1}+m(v,v') w^{-m-1}+\sum_{0 \le l < m} [v,v']u^{l-m} w^{-l-1}\\
\end{split}
\end{equation*}
so that for $\balpha \in H^0(\bbP \setminus O, \bg\otimes \mathfrak{m}_P)$, we have
\begin{equation*}
[\sum_{l \ge 0} j_{l} w^{-l-1}, \balpha] = -(v,d\balpha(w)/dw)+\hiota^{[v,\balpha(z)]} (z) - [\frac{v}{u-w},\balpha(u)-\balpha(w)]
\end{equation*}
Thus
\begin{equation}
\sum_{l \ge 0} j_{l} w^{-l-1}  = \hat{\iota}(z)|_{{}_{\lO U\fg_-\subset U\fg^z_-}} : \lO U\fg_- \to \lO U\fg_-
\end{equation}
We have used the notation $\iota^{v}(z)=\hiota^{v}(z) du_z, \  \epsilon^{\nu(z)}(z)=\hepsilon^(z) du_z$.

\subsection{``Insertions"}\label{insertions}

For each $j=1,\dots,n$ let an irreducible representation $W_j$ of $\bg$ be given. Let $\bz$ be a sequence of distinct points $(z_1, \dots ,z_n)$ on $\bbP$; we denote by $|\bz|$ the set $\{z_1, \dots ,z_n\}$.  Let $\fg$ and $\fg^z$  be defined as before. Let $\cW=U\fg\otimes_{U\bg} \{W_1 \otimes \dots \otimes W_n\}$.  Let $\cW^z=U\fg^z\otimes_{U\bg} \{W_1 \otimes \dots \otimes W_n\}$.  

Choose a reference point $P$ disjoint from $|\bz|$. We will now define, for $z \ne \{P,\text{poles of $\nu$}\}$ and (in the case of $\iota$, also distinct from each $z_j$), operators $\iota^{\nu(z)}(z)$ and $\epsilon^{\nu(z)}(z)$ which generalise those encountered previously.

Let $\fg_-$ and $\fg^z_-$  be defined as before. Let $\cW_P=U\fg_- \otimes \{W_1 \otimes \dots \otimes W_n\}$, all tensor products being taken over $\bbC$.  Let $\cW^z_P=U\fg^z_- \otimes \{W_1 \otimes \dots \otimes W_n\}$.

We define $\epsilon^{\nu(z)}(z):\cW_P \to \cW_P \otimes K_z$ as before, by multiplying on the left by
\begin{equation*}
    \epsilon_z^{\nu(z)} = v \otimes \frac{1}{u-u(z)} \otimes du_z
\end{equation*}
Define $\iota^{\nu(z)}(z):\cW^z_P \to \cW^z_P \otimes K_z$ by requiring
\begin{equation*}
\begin{split}
\iota^{\nu(z)}(z) (\ttw_1\otimes \dots \otimes \ttw_n)=\sum_j \ttw_1 \otimes \dots \otimes \frac{v(\ttw_j)du_{z}}{u(z)-u(z_j)} \otimes \dots \otimes \ttw_n\\
\end{split}
\end{equation*}
and
\begin{equation*}
\iota^{\nu(z)}(z) \circ \balpha  - \balpha \circ \iota^{\nu(z)}(z)=-(v, d\balpha(z))
+ \iota^{[v, \balpha(z)]}(z)
-[\epsilon^{\nu(z)}(z),\balpha-\balpha(z)]
\end{equation*}
for $\balpha \in \fg^z_-$. 

One checks as before that $\iota$ descends to the tensor product. 

As before, we can add $\epsilon$ and $\iota$ to define the field $j$, and this patches to yield

\begin{theorem} We have a field on $\bbP$ (defined for $z$ distinct from each $z_j$):
\begin{equation*}
j^{\nu(z)}(z)  :\cW^z \to \cW \otimes K_z
\end{equation*}
satisfying the same OPE (\ref{OPEcurrent}) as before.
 \end{theorem}
 
\begin{proof} We now complete the part of the proof that was postponed. Set
$$
\bw=\ttw_1\otimes \dots \otimes \ttw_n
$$
Consider, for $z_1 \ne z_2$, $j^{v_1}(z_1)j^{v_2}(z_2)\bw$.  We have
\begin{equation*}
\begin{split}
(j^{v_1}(z_1)j^{v_2}(z_2)-j^{v_2}(z_2)j^{v_1}(z_1))\bw
&=\frac {1}{u(z_1)-u(z_2)} \{\epsilon^{[v_1,v_2]} (z_1) du_{z_2} - du_{z_1} \epsilon^{[v_1,v_2]} (z_2)\} \bw\\
+\frac{(v_1,v_2)du_{z_1}du_{z_2}}{(u(z_1)-u(z_2))^2}
&+\frac{1}{u(z_1)-u(z_2)} \{\iota^{[v_1,v_2]}(z_1)du_{z_2}+ \epsilon^{[v_1,v_2]}(z_2) du_{z_1}\} \bw\\
-\frac{(v_2,v_1)du_{z_1}du_{z_2}}{(u(z_2)-u(z_1))^2}
&-\frac{1}{u(z_2)-u(z_1)} \{\iota^{[v_2,v_1]}(z_2)du_{z_1}+ \epsilon^{[v_2,v_1]}(z_1)du_{z_2}\}\bw\\
&+(i^{v_1}(z_1)i^{v_2}(z_2)-i^{v_2}(z_2)i^{v_1}(z_1))\bw\\
&=\frac{1}{u(z_1)-u(z_2)} \{\iota^{[v_1,v_2]}(z_1)du_{z_2}-\iota^{[v_1,v_2]}(z_2)du_{z_1}\} \bw\\
&+(i^{v_1}(z_1)i^{v_2}(z_2)-i^{v_2}(z_2)i^{v_1}(z_1))\bw\\
=\frac{du_{z_1}du_{z_2}}{u(z_1)-u(z_2)} &\sum_j \ttw_1 \otimes \dots \otimes [v_1,v_2](\ttw_j) \{\frac{1}{u(z_1)-u(z_j)}+ \frac{1}{u(z_2)-u(z_j)}\}\otimes \dots \otimes \ttw_n \\
+du_{z_1}du_{z_2}&\sum_j \ttw_1 \otimes \dots \otimes \frac{[v_1,v_2](\ttw_j)}{(u(z_1)-u(z_j))(u(z_2)-u(z_j))} \otimes \dots \otimes \ttw_n \\
&=0
\end{split}
\end{equation*}

\end{proof}

\section{\textbf{Current algebras: the states and symmetries}}\label{section: Current algebras: the states and symmetries}

We now introduce subspaces of $\cW$ associated to subsets of $\bbP$.

Given a subset (eg., a domain in the analytic topology) $D \subset \bbP$, we denote by  $\fg_D$ the space of meromorphic maps $\bbP \to \bg$ which are regular at all points of $D$, and by $\lD\fg$ the space of such maps with all singularities contained in $D$. If $D$ consists of a single point $P$, we set $\fg_P=\fg_{\{P\}}$.  These conventions extend to $\cW$.

Suppose given a sequence of points $\bz=(z_1,\dots,z_n)$ distinct from each other and from $P$. We let $|\bz|$ denote the underlying set $\{z_1,\dots,z_n\}$ and $\lbz\fg$ will denote the space of functions $\bbP \to \bg$ regular outside the $z_i$.

\begin{proposition} The space $\lbz\cW$ over $\lbz\cW$ is generated (from $\mathbf{1}$) by the operators $\epsilon^{\nu(z)}(z_i)$ and their derivatives.  
\end{proposition}

\subsection{The affine lie algebra}

Given a meromorphic function $\nu:\bbP \to g$, we can integrate $j^{\nu(z)}(z)$ around a contour $\gamma$ to define an operator $J_\gamma^\nu$, informally:
\begin{equation}\label{currentintegral}
J_\gamma^\nu = \frac{1}{2\pi i} \int_{\gamma} j^{\nu} (z)
\end{equation}
Just as with the chiral boson, we will define operators $\lD J^\nu$, $J^\nu_P$ and $\lP J^\nu$.

\subsection{A formal computation} 

Anticipating definitions, let us formally compute the commutator $[J_P^{\nu}, J_P^{\mu}]$. It is convenient at this point to choose a basis $\{ \sfv _l \}$ for $\bg$, and write $\nu=\sum_a \nu_a \sfv _a, \mu=\sum_b \mu_b \sfv _b$. Then (writing $u=u(z)$, etc.)
\begin{equation*}
J_P^{\mu} \circ J_P^{\nu} = \frac{1}{2\pi i}\int_{\gamma_1} du_1 \frac{1}{2\pi i}\int_{\gamma} du \ \sum_a \sum_b \mu_b(u_1) \nu_a(u) \hj^{\sfv _b} (u_1) \hj^{\sfv _a} (u) 
\end{equation*}
with $\gamma_1$ \emph{outside} (closer to $P$ than) $\gamma$. On the other hand,
\begin{equation*}
J_P^{\nu} \circ J_P^{\mu} = \frac{1}{2\pi i}\int_{\gamma_1} du_1 \frac{1}{2\pi i}\int_{\gamma} du \ \sum_a \sum_b \mu_b(u_1) \nu_a(u) \hj^{\sfv _b} (u_1) \hj^{\sfv _a} (u) 
\end{equation*}
with $\gamma_1$ \emph{inside} $\gamma$. We will now use the fact that as $u_1 \to u$,
\begin{equation*}
\hj^{\sfv _b}(u_1)\hj^{\sfv _a}(u)+\frac{(\sfv _a,\sfv _b)}{(u-u_1)^2}-\frac{1}{u_1-u}\hj^{[\sfv _b,\sfv _a]}(u)
\end{equation*}
 has a regular limit. Therefore, moving $\gamma_1$ across $\gamma$ from outside to inside, we pick up the residue (for each $u \in \gamma$)
\begin{equation*}
\begin{split}
\sum_a \sum_b \mu_b(u)\nu_a(u)j^{[\sfv _b,\sfv _a]} (u) &- \sum_a \sum_b \nu_a(u) d\mu_b(u)/du  (\sfv _a,\sfv _b)\\
&=j^{[\mu,\nu](u)}(u)-(\nu(u),d\mu (u)/du )
\end{split}
\end{equation*}
Doing the $u$-integral, we get
\begin{equation}\label{KMformal}
\begin{split}
[J_P^{\mu},J_P^{\nu}] =J_P^{[\mu,\nu]}+\frac{1}{2\pi i}\int_P (\mu,d\nu) = J_P^{[\mu,\nu]}+Res_P (\mu, d\nu)
\end{split}
\end{equation}
which shows that a central extension of $\fg$ determined by the residue at $P$ acts on $\cW_P$.

\subsubsection{Definitions of $\lD J^\nu$ and $J^\nu_P$}\label{DefKM}

We now turn to a careful definition of (\ref{currentintegral}).  Choose a base-point $P$. In what follows, $\gamma$ will be a contour winding clockwise aound $P$, and avoiding the points $z_i$. The domains $D$ and $D'$ will be defined as before.

We will define the operators $J^{\nu}_\gamma$ inductively. Assume that $\nu$ is regular along $\gamma$. Let $\balpha \in \fg_-$, and formally compute the commutator $[J^{\nu}_\gamma,\balpha]$:
\begin{equation*}
\begin{split}
[J^\nu_\gamma,\balpha] &=\frac{1}{2\pi i}\int_\gamma [\epsilon^{\nu(z)}(z) + \iota^{\nu(z)}(z),\balpha(.)] du_z\\
&= \frac{1}{2\pi i}\int_\gamma [\frac{\nu(z)}{u(.)-u(z)},\balpha(.)] du_z  \\
&+  \frac{1}{2\pi i}\int_\gamma \{-(\nu(z), d\balpha(z))+\iota^{[\nu(z),\balpha(z)]}-[\nu(z),\frac{\balpha(.)-\balpha(z)}{u(.)-u(z)}] du_z\}\\
\end{split}
\end{equation*}
which yields finally
\begin{equation}\label{deflDJone}
[J^\nu_\gamma,\balpha] = J^{[\nu,\balpha]}_\gamma+ \frac{1}{2\pi i} \int_\gamma (\nu, d\balpha)
\end{equation}
Consider now the expression $J^\nu_\gamma (\ttw_1 \otimes \dots \otimes \ttw_n)$.  Computing formally:
\begin{equation*}
\begin{split}
J^\nu_\gamma (\ttw_1 \otimes \dots \otimes \ttw_n) &=\frac{1}{2\pi i}\int_\gamma \{\epsilon^{\nu(z)}(z)  (\ttw_1 \otimes \dots \otimes \ttw_n) + \iota^{\nu(z)}(z)(\ttw_1 \otimes \dots \otimes \ttw_n) \} du_z\\
&=\frac{1}{2\pi i}\int_\gamma \frac{\nu(z)}{u(.)-u(z)} \otimes  (\ttw_1 \otimes \dots \otimes \ttw_n) du_z\\
&+ \frac{1}{2\pi i}\int_\gamma\sum_j (\ttw_1 \otimes \dots \otimes \frac{\nu(z) (\ttw_j)}{u(z)-u(z_j)} \otimes \dots \otimes \ttw_n)  du_z\\
\end{split}
\end{equation*}
The first integral represents an element of $\fg_- \otimes W_1\otimes \dots \otimes W_n$. This suggests that we take the point ``$.$'' to be in the domain $D'$. Given this choice, this term equals $(\nu(.) - \nu(P)) \otimes (\ttw_1 \otimes \dots \otimes \ttw_n)$; the second integral yields $ \sum_j (\ttw_1 \otimes \dots \otimes \nu(P) (\ttw_j) \otimes \dots \otimes \ttw_n) 
- \underset{z_j \in D'}\sum (\ttw_1 \otimes \dots \otimes \nu(z_j) (\ttw_j) \otimes  \dots \otimes \ttw_n)$, so we get

\begin{equation}\label{deflDJtwo}
J^\nu_\gamma (\ttw_1 \otimes \dots \otimes \ttw_n)= 
\begin{cases}
\begin{aligned}
& \nu \otimes_{U\bg} (\ttw_1 \otimes \dots \otimes \ttw_n)\\
-& \underset{z_j \in D'}\sum (\ttw_1 \otimes \dots \otimes \nu(z_j) (\ttw_j) \otimes  \dots \otimes \ttw_n) \\
& \ \ \ \ \ \ \text{if} \ \nu \ \text{is regular in}\ D'\\
& \underset{z_j \in D}\sum (\ttw_1 \otimes \dots \otimes \nu(z_j) (\ttw_j) \otimes  \dots \otimes \ttw_n) \\
& \ \ \ \ \ \  \text{if} \ \nu \ \text{is regular in}\ D\\
\end{aligned}
\end{cases}
\end{equation}
\emph{We let equations (\ref{deflDJone}) and (\ref{deflDJtwo}) together define $\lD J^{\nu}: \cW \to \cW$. This leaves invariant $\lD \cW$; and we use the same notation to denote the restricted map $\lD J^{\nu}: \lD \cW \to \lD\cW$}

\begin{proposition} We have $[\lD J^{\mu},\lD J^{\nu}]-\lD J^{[\mu,\nu]}=0$ if either
\begin{itemize}
\item both $\mu$ and $\nu$ are regular in $D'$, or
\item both $\mu$ and $\nu$ are regular in $D$.
\end{itemize}
In general,
\begin{equation}\label{KMrigorous}
[\lD J^{\mu},\lD J^{\nu}]-\lD J^{[\mu,\nu]}= \frac{1}{2 \pi i} \int_{\partial D} (\mu, d\nu)
\end{equation}
\end{proposition}

As in the case of the Heisenberg algebra, we can exhaust the complement of $P$ by discs $D$, and \emph{define $J^\nu_P:\cW_P \to \cW_P$ by equations (\ref{defJPone}) and (\ref{defJPtwo}) below:
} \begin{equation}\label{defJPone}
[J^\nu_P,\balpha] = J^{[\nu,\balpha]}_P-Res_P(\nu, d\balpha), \ \balpha \in \bg_-
\end{equation}
and
\begin{equation}\label{defJPtwo}
J^\nu_P (\ttw_1 \otimes \dots \otimes \ttw_n)= 
\begin{cases}
\begin{aligned}
& (\nu(.) - \nu(P)) \otimes (\ttw_1 \otimes \dots \otimes \ttw_n)\\
& +   \sum_j (\ttw_1 \otimes \dots \otimes \nu(P) (\ttw_j) \otimes \dots \otimes \ttw_n) \\
&= \nu \otimes_{U\bg} (\ttw_1 \otimes \dots \otimes \ttw_n)\\
& \ \ \ \ \ \ \text{if} \ \nu \ \text{is regular at}\ P\ \text{and}\\
& \underset{z_j}\sum (\ttw_1 \otimes \dots \otimes \nu(z_j) (\ttw_j) \otimes  \dots \otimes \ttw_n) \\
& \ \ \ \ \ \  \text{if} \ \nu \ \text{is regular away from}\ P\\
\end{aligned}
\end{cases}
\end{equation}

We now define, for each $l=1,\dots,n$ such that $z_l \in D$,  an operator $\lD J^{\nu}_l$, by integrating $j^{\nu}$ around a contour (close enough and winding clockwise) around $z_l$.   \emph{We define this by equations (\ref{defJlone})  and (\ref{defJltwo}) below.}
\begin{equation}\label{defJlone}
[J^\nu_l,\balpha] =Res_{z_l}(\nu,d\balpha)+J^{[\nu,\balpha]}_l
\end{equation}
 \begin{equation}\label{defJltwo}
J^\nu_l (\ttw_1 \otimes \dots \otimes \ttw_n)= 
\begin{cases}
\begin{aligned}
& \nu \otimes_{U\bg} (\ttw_1 \otimes \dots \otimes \ttw_n)\\
& -  \sum_{j\ne l} (\ttw_1 \otimes \dots \otimes \nu(z_j) (\ttw_j)  \otimes \dots \otimes \ttw_n) \\
& \ \ \ \ \ \ \text{if} \ \nu \ \text{is regular away from}\ z_l\\
& \ttw_1 \otimes \dots \otimes \nu(z_l) (\ttw_l) \otimes \dots \otimes \ttw_n \\
& \ \ \ \ \ \  \text{if} \ \nu \ \text{is regular at}\ z_l\\
\end{aligned}
\end{cases}
\end{equation}

\begin{proposition} We have $[J^{\mu}_l,J^{\nu}_l]-J^{[\mu,\nu]}_l=0$ if either
\begin{itemize}
\item both $\mu$ and $\nu$ are regular at $z_l$, or
\item both $\mu$ and $\nu$ are regular away from $z_l$.
\end{itemize}
In general,
\begin{equation}
[J^{\mu}_l,J^{\nu}_l]-J^{[\mu,\nu]}_l=Res_{z_l}(\mu,d\nu)
\end{equation}
\end{proposition}

It is more interesting to compute the commutators of the loops around \emph{different} points, which we take to be labelled by indices 1 and 2.
\begin{theorem} We have $[J^{\nu_1}_1,J^{\nu_2}_2]=0$ \end{theorem}

\begin{proof} Using (\ref{defJlone}), we see that $[[J^{\nu_1}_1,J^{\nu_2}_2], \balpha]= [[J^{\nu_1}_1,\balpha],J^{\nu_2}_2]+[J^{\nu_1}_1,[J^{\nu_2}_2,\balpha]]=[J^{[\nu_1,\balpha]}_1,J^{\nu_2}_2]+[J^{\nu_1}_1,J^{[\nu_2,\balpha]}_2]$, which lets us reduce to case when the commutator is acting on $\ttw_1 \otimes \dots \otimes \ttw_n$. There are four cases to consider:

\medskip

\noindent (1) $\nu_1$ is regular at $z_1$, and $\nu_2$ is regular at $z_2$. In this case $J^{\nu_1}_1 \circ J^{\nu_2}_2 (\ttw_1 \otimes \dots \otimes \ttw_n) = \nu_1(z_1) (\ttw_1) \otimes \nu_2(z_2) \ttw_2 \otimes \dots \otimes \ttw_n = J^{\nu_2}_2 \circ J^{\nu_1}_1 (\ttw_1 \otimes \dots \otimes \ttw_n)$.

\medskip

\noindent (2) $\nu_1$ is regular at $z_1$, and $\nu_2$ away from $z_2$. In this case,
\begin{equation*}
\begin{split}
J^{\nu_1}_1 \circ J^{\nu_2}_2 (\ttw_1 \otimes \dots \otimes \ttw_n) &= J^{\nu_1}_1 \{\nu_2 \otimes_{U\bg} \ttw_1 \otimes \dots \otimes \ttw_n\\ 
&- \sum_{j \ne 2} \ttw_1 \otimes \dots \otimes \nu_2(z_j) (\ttw_j)  \otimes \dots \otimes \ttw_n\}\\
&=J^{[\nu_1,\nu_2]}_1 (\ttw_1 \otimes \dots \otimes \ttw_n)+\nu_2 \otimes_{U\bg} \nu_1(z_1)\ttw_1 \otimes \dots \otimes \ttw_n\\
&- \nu_1(z_1)\nu_2(z_2) \ttw_1 \otimes  \dots \otimes \ttw_n\\&-\sum_{j > 2} \nu_1(z_1) \ttw_1 \otimes \dots \otimes \nu_2(z_j) \ttw_j  \otimes \dots \otimes \ttw_n\\
&=[\nu_1,\nu_2] (z_1) \ttw_1 \otimes \dots \otimes \ttw_n)+\nu_2 \otimes_{U\bg} \nu_1(z_1)\ttw_1 \otimes \dots \otimes \ttw_n\\
&- \nu_1(z_1)\nu_2(z_2) \ttw_1 \otimes  \dots \otimes \ttw_n\\&-\sum_{j > 2} \nu_1(z_1) \ttw_1 \otimes \dots \otimes \nu_2(z_j) \ttw_j  \otimes \dots \otimes \ttw_n\\
\end{split}
\end{equation*}
On the other hand
\begin{equation*}
\begin{split}
J^{\nu_2}_2 \circ J^{\nu_1}_1 (\ttw_1 \otimes \dots \otimes \ttw_n) &= J^{\nu_2}_2 (\nu_1(z_1)\ttw_1 \otimes \dots \otimes \ttw_n)\\ 
&=\nu_2 \otimes_{u\bg} \nu_1(z_1)\ttw_1 \otimes \dots \otimes \ttw_n\\
&-\nu_2(z_2) \nu_1(z_1)\ttw_1 \otimes \dots \otimes \ttw_n\\
&-\sum_{j>2} \nu_1(z_1)\ttw_1 \otimes \dots \nu_2(z_j) (\ttw_j) \otimes \dots  \otimes \ttw_n\\
\end{split}
\end{equation*}
and again we see that the commutator kills $\ttw_1 \otimes \dots \otimes \ttw_n$

\medskip 

\noindent(3) $\nu_2$ is regular at $z_2$, and $\nu_1$ away from $z_1$. This case is dealt with as above.

\medskip

\noindent(4) $\nu_1$ is regular away from $z_1$, and $\nu_2$ away from $z_2$. In this case
\begin{equation*}
\begin{split}
J^{\nu_1}_1 \circ J^{\nu_2}_2 (\ttw_1 \otimes \dots \otimes \ttw_n) &= J^{\nu_1}_1 \{\nu_2 \otimes_{U\bg} \ttw_1 \otimes \dots \otimes \ttw_n\\ 
&- \sum_{j \ne 2} \ttw_1 \otimes \dots \otimes \nu_2(z_j) (\ttw_j)  \otimes \dots \otimes \ttw_n\}\\
&=Res_{z_1}(\nu_1,d\nu_2) \ttw_1 \otimes \dots \otimes \ttw_n\\
&+J^{[\nu_1,\nu_2]}_1 (\ttw_1 \otimes \dots \otimes \ttw_n)\\ 
&+\nu_2 J^{\nu_1}_1 (\ttw_1 \otimes \dots \otimes \ttw_n)\\
&- \sum_{j \ne 2} \nu_1 \otimes_{U\bg} \ttw_1 \otimes \dots \otimes \nu_2(z_j) (\ttw_j)  \otimes \dots \otimes \ttw_n\\
&+\sum_{j_1\ne 1,\ j_2 \ne 2}  \ttw_1 \otimes \dots \otimes \nu_1(z_{j_1}) (\ttw_{j_1}) \otimes \dots \nu_2(z_{j_2}) (\ttw_{j_2})  \otimes \dots \otimes \ttw_n\\
&+\sum_{j\ne 1,\  2}  \ttw_1 \otimes \dots \otimes \nu_1(z_j) \nu_2(z_j) (\ttw_{j})  \otimes \dots \otimes \ttw_n\\
&=Res_{z_1}(\nu_1,d\nu_2) \ttw_1 \otimes \dots \otimes \ttw_n\\
&+J^{[\nu_1,\nu_2]}_1 (\ttw_1 \otimes \dots \otimes \ttw_n)\\ 
&+\nu_2 \nu_1 \otimes_{U\bg} (\ttw_1 \otimes \dots \otimes \ttw_n)\\
&-\nu_2 \sum_{j\ne 1} \ttw_1 \otimes \dots \nu_1(z_j)(\ttw_j) \otimes \dots  \otimes \ttw_n\\
&- \nu_1 \sum_{j \ne 2}  \otimes_{U\bg} \ttw_1 \otimes \dots \otimes \nu_2(z_j) (\ttw_j)  \otimes \dots \otimes \ttw_n\\
&+\sum_{j_1\ne 1,\ j_2 \ne 2}  \ttw_1 \otimes \dots \otimes \nu_1(z_{j_1}) (\ttw_{j_1}) \otimes \dots \nu_2(z_{j_2}) (\ttw_{j_2})  \otimes \dots \otimes \ttw_n\\
&+\sum_{j\ne 1,\  2}  \ttw_1 \otimes \dots \otimes \nu_1(z_j) \nu_2(z_j) (\ttw_{j})  \otimes \dots \otimes \ttw_n\\
\end{split}
\end{equation*}
From this it follows that
\begin{equation*}
\begin{split}
[J^{\nu_1}_1,J^{\nu_2}_2] (\ttw_1 \otimes \dots \otimes \ttw_n) 
&=(Res_{z_1}(\nu_1,d\nu_2) - Res_{z_2}(\nu_2,d\nu_1)) \ttw_1 \otimes \dots \otimes \ttw_n\\
&+(J^{[\nu_1,\nu_2]}_1+J^{[\nu_1,\nu_2]}_2) (\ttw_1 \otimes \dots \otimes \ttw_n)\\ 
&-[\nu_1, \nu_2] \otimes_{U\bg} (\ttw_1 \otimes \dots \otimes \ttw_n)\\
&+\sum_{j\ne 1,\  2}  \ttw_1 \otimes \dots \otimes [\nu_1, \nu_2](z_j) (\ttw_{j})  \otimes \dots \otimes \ttw_n\\
\end{split}
\end{equation*}
The first term vanishes because the sum of the residues of the meromorphic form $(\nu_1,d\nu_2)$ is zero. Now write $[\nu_1,\nu_2] = \mu_1+\mu_2$, with $\mu_1$ regular away from $z_1$ and vanishing at $z_2$, and $\mu_2$ regular away from $z_2$ and vanishing at $z_1$. Then
\begin{equation*}
\begin{split}
(J^{[\nu_1,\nu_2]}_1+J^{[\nu_1,\nu_2]}_2) (\ttw_1 \otimes \dots \otimes \ttw_n)
&=(J^{\mu_1+\mu_2}_1+J^{\mu_1+\mu_2}_2) (\ttw_1 \otimes \dots \otimes \ttw_n)\\
&=\mu_1 \otimes_{U\bg} \ttw_1 \otimes \dots \otimes \ttw_n\\
&-\sum_{j \ne 1} \ttw_1 \otimes \dots \otimes \mu_1(z_j) \ttw_j \otimes \dots \otimes \ttw_n\\
&+\mu_2(z_1) \ttw_1 \otimes \dots \otimes \ttw_n\\
&+\mu_2 \otimes_{U\bg} \ttw_1 \otimes \dots \otimes \ttw_n\\
&-\sum_{j \ne 2} \ttw_1 \otimes \dots \otimes \mu_2(z_j) \ttw_j \otimes \dots \otimes \ttw_n\\
&+\ttw_1 \otimes \mu_1(z_2) \ttw_2 \otimes \dots \otimes \ttw_n\\
&=(\mu_1+\mu_2) \otimes_{U\bg} (\ttw_1 \otimes \dots \otimes \ttw_n)\\
&-\sum_{j\ne 1,\  2}  \ttw_1 \otimes \dots \otimes (\mu_1+\mu_2)(z_j) (\ttw_{j})  \otimes \dots \otimes \ttw_n\\
&=[\nu_1,\nu_2] \otimes_{U\bg} (\ttw_1 \otimes \dots \otimes \ttw_n)\\
&-\sum_{j\ne 1,\  2}  \ttw_1 \otimes \dots \otimes [\nu_1,\nu_2](z_j) (\ttw_{j})  \otimes \dots \otimes \ttw_n\\
\end{split}
\end{equation*}
This proves that the remaining terms vanish as well.
\end{proof}

We record
\begin{proposition}
$J_P^{\nu} = \sum_l J_\bL^{\nu}$
\end{proposition}

\subsection{Affine lie algebra without insertions}

Consider the case when there are no insertions. Then $J^\nu_P$ is defined inductively by
\begin{equation*}
[J^\nu_P,\balpha] (w)=Res_P(\nu, d\balpha)+J^{[\nu,\balpha]}_P
\end{equation*}
and
\begin{equation*}
J^\nu_P (\mathbf{1})= 
\begin{cases}
(\nu(w) - \nu(P)) \otimes \mathbf{1}\ \  \text{if} \ \nu \ \text{is regular at}\ P\\
0 \ \  \text{if} \ \nu \ \text{is regular away from}\ P\\
\end{cases}
\end{equation*}
By an easy induction, we can prove that if $\nu_i \in \fg_-, \ \ i=1,\dots,N$,
\begin{equation*}
J^{\nu_1}_P \dots J^{\nu_N}_P (\mathbf{1})= 
\nu_1 \dots \nu_N \in U\fg_-
\end{equation*}

\subsection{Separation of variables}

Consider the space $\lbz\cW$. By definition,
\begin{equation*}
\lbz\cW=U\lbz\fg\otimes_{U\bg} \{W_1 \otimes \dots \otimes W_n\}
\end{equation*}
where $\lbz\fg$ is the space of meromorphic maps $\bbP \to \bg$ regular away from the $z_j$.  For $j=1,\dots,n$, set
\begin{equation*}
\lzj\cW_j=U\lzj\fg\otimes_{U\bg} W_j 
\end{equation*}

\begin{theorem} There is a surjective map
\begin{equation*}
{}_{z_1}\cW_1 \otimes \dots \otimes {}_{z_n}\cW_n = \lbz\cW
\end{equation*}
which preserves filtrations, and is a morphism of $\oplus_j \hfg_j$-modules.
\end{theorem}

\begin{proof} Both sides are generated as $\oplus_j \hfg_j$-modules from 
$W_1 \otimes \dots \otimes W_n$.  
\end{proof}

\begin{remark} It is worth writing out this morphism in a simple case, when we have two insertions, but with the trivial representation attached to each point. Consider, therefore, the space ${}_{z_1,z_2}\cV=U{}_{z_1,z_2}\fg$. The morphism
sends
$$
\frac{v_1}{(u-u(z_1))^{l_1}} \mathbf{1}_1 \otimes \frac{v_2}{(u-u(z_1))^{l_2}} \mathbf{1}_2
$$
to
\begin{equation*}
\begin{split}
J_1^{\frac{v_1}{(u-u(z_1))^{l_1}}} J_2^{\frac{v_2}{(u-u(z_2))^{l_2}}} \mathbf{1}
&= J_1^{\frac{v_1}{(u-u(z_1))^{l_1}}} \frac{v_2}{(u-u(z_2))^{l_2}} \mathbf{1}\\
&= [J_1^{\frac{v_1}{(u-u(z_1))^{l_1}}},\frac{v_2}{(u-u(z_2))^{l_2}}] \mathbf{1} \\
&+ \frac{v_2}{(u-u(z_2))^{l_2}} J_1^{\frac{v_1}{(u-u(z_1))^{l_1}}} \mathbf{1}\\
&= (v_1,v_2) Res_{z_1} \{\frac{1}{(u-u(z_1))^{l_1}} d \frac{1}{(u-u(z_2))^{l_2}}\} \mathbf{1}\\
&+ J_1^{[{\frac{v_1}{(u-u(z_1))^{l_1}}},\frac{v_2}{(u-u(z_2))^{l_2}}]} \mathbf{1}\\
&+ \frac{v_2}{(u-u(z_2))^{l_2}} J_1^{\frac{v_1}{(u-u(z_1))^{l_1}}} \mathbf{1}\\
&= (v_1,v_2) Res_{z_1} \{\frac{1}{(u-u(z_1))^{l_1}} d \frac{1}{(u-u(z_2))^{l_2}}\} \mathbf{1}\\
&+[v_1,v_2]\frac{p_1(u)}{(u-u(z_1))^{l_1}} \\
&+ \frac{v_2}{(u-u(z_2))^{l_2}}\frac{v_1}{(u-u(z_1))^{l_1}} \mathbf{1}\\
\end{split}
\end{equation*}
where we have written  
$$
\frac{1}{(u-u(z_1))^{l_1}} \frac{1}{(u-u(z_2))^{l_2}}=\frac{p_1(u)}{(u-u(z_1))^{l_1}} + \frac{p_2(u)}{(u-u(z_2))^{l_2}}
$$
with the polynomials $p_1$ and $p_2$ chosen such that each of the terms on the right is regular at $P$ (the point at infinity). On the other hand,
\begin{equation*}
\begin{split}
J_2^{\frac{v_2}{(u-u(z_2))^{l_2}}} J_1^{\frac{v_1}{(u-u(z_1))^{l_1}}} \mathbf{1}
&= (v_2,v_1) Res_{z_2} \{\frac{1}{(u-u(z_2))^{l_2}} d \frac{1}{(u-u(z_1))^{l_1}}\} \mathbf{1}\\
&+[v_2,v_1]\frac{p_2(u)}{(u-u(z_2))^{l_2}} \mathbf{1} \\
&+ \frac{v_1}{(u-u(z_1))^{l_1}}\frac{v_2}{(u-u(z_2))^{l_2}} \mathbf{1}\\
\end{split}
\end{equation*}
We check that indeed
$$
J_1^{\frac{v_1}{(u-u(z_1))^{l_1}}} J_2^{\frac{v_2}{(u-u(z_2))^{l_2}}} \mathbf{1}
=J_2^{\frac{v_2}{(u-u(z_2))^{l_2}}} J_1^{\frac{v_1}{(u-u(z_1))^{l_1}}} \mathbf{1}
$$

\end{remark}

\section{\textbf{Current algebras pairings}}

\subsection{Pairings in the absence of insertions}\label{currentpairing}

We now define pairings paralleling the discussion of pairings for the chiral boson in \S \ref{bosonpairing}.

As before,  let $\gamma$ be a contour separating a domain $D$
from another domain $D'$, with $P \in D'$.
We set ${}_D\fg$ to be the space of meromorphic functions $\bbP \to \bg$ with all poles contained in $D$.   For $z \in D$, we define ${}_D\fg^z$ to be the space of such functions (with all poles in $D$, but) regular at $z$.  We denote by $\lDp\fg$ the space of meromorphic maps $\bbP \to \bg$ with all poles in $D'$; given $z' \in D$, we let $\lDp\fg^{z'}$ denote those which are regular at $z'$.

Given $v\in \bg$ and $z, \ z' \in \bbP \ P$, we define fields as before:
\begin{itemize} 
\item $\epsilon^v(z): \lD\cV \to \lD\cV \otimes K_z\ $,
\item $\iota^v(z): \lD\cV^z \to \lD\cV^z \otimes K_z\ $,
\item $\epsilon'^{v}(z'): \lDp\cV \to \lDp\cV \otimes K_{z'}\ $, 
\item $\iota'^{v}(z'): \lDp\cV^{z'} \to \lDp\cV^{z'} \otimes K_{z'}\ $.   
\end{itemize}

Given a second point $\tP$, chosen to be in $D$, we also define, as in \S \ref{globalcurrent}, fields $\tiota$, etc.

\begin{proposition}\label{proposition-no-insertions} There is a unique pairing between $\lDp\cV$ and $\lD\cV$ such that $\langle {\mathbf 1}',\mathbf 1\rangle =1$, and for $v \in \bg$ and
$\balpha' \in \lDp\cV, \ \balpha \in \lD\cV$, we have
\begin{itemize}
\item the equality of (meromorphic) sections of $K$ on $D$:
\begin{equation*}
\langle \balpha',\epsilon^v(z) \balpha \rangle =-\langle \iota'^v(z) \balpha', \balpha \rangle 
\end{equation*}
\item the equality of (meromorphic) sections of $K$ on $D' \setminus P$:
\begin{equation*}
\langle \epsilon'^{v}(z')\balpha',\balpha\rangle =-\langle \balpha',\iota^{v}(z') \balpha \rangle 
\end{equation*}
which extends to $P$ as the equality
\begin{equation*}
\langle \tepsilon'^{v}(z')\balpha',\balpha\rangle =-\langle \balpha',\tiota^{v}(z') \balpha \rangle 
\end{equation*}
\end{itemize}
\end{proposition}

By ``equality of (meromorphic) sections'' we mean the following. For fixed $\balpha', \balpha$, both sides are sections of $K$ with poles at finitely many points determined by $\balpha'$ and $\balpha$; these sections agree.  The proof follows the lines of the proof of Proposition \ref{propbosonpairing}, using the next lemma. As in that last proof, the induction starts with the step $\langle {\mathbf 1}',\mathbf 1\rangle =1$. (Cf., Remark \ref{independentofintegralpairing}.)

\begin{lemma}\label{pairingdef}  Suppose that a pairing as in the proposition exists for all $\bchi \in \lD\cV$ with $degree\ \bchi \le N$ and $\bchi' \in \lDp\cV$ with $degree\ \bchi' \le N'$. Then, given such $\bchi,\ \bchi'$, and for $z \in D$, $z' \in D'$ and $v,\ v' \in \bg$, we have 
\begin{equation*}
\langle \iota'^{v}(z) \epsilon'^{v'}(z') \bchi', \bchi\rangle =\langle \bchi', \iota^{v'}(z') \epsilon^{v}(z)\bchi\rangle 
\end{equation*}
\end{lemma}

\begin{proof} Note that $degree\ i'(z) e'(z') \chi' = degree\ \chi'$ and $degree\ i(z') e(z)\chi = degree\ \chi$, so that both sides of the above equation are determined by assumption of the Lemma. The proof itself is a computation that uses (\ref{iotaepsiloncommute})  
\begin{equation*} 
\begin{split}
 \langle \iota'^v(z) \epsilon'^{v'}(z') \bchi', \bchi\rangle &=\langle \epsilon'^{v'} (z') \iota'^{v} (z) \bchi',\bchi\rangle \\
&+\frac{(v,v'\rangle du_{z}du_{z'}}{(u(z)-u(z'))^2}\langle \bchi',\bchi\rangle \\
&+\frac{1}{u(z)-u(z')}du_{z'} \langle \iota'^{[v,v']}(z)\bchi',\bchi\rangle \\
&+\frac{1}{u(z)-u(z')} du_{z} \langle \epsilon'^{[v,v']}(z')\bchi',\bchi\rangle \\
&=\langle \bchi',\epsilon^v(z)\iota^{v'}(z')\bchi\rangle \\
&+\frac{(v,v')du_{z}du_{z'}}{(u(z)-u(z'))^2}\langle \bchi',\bchi\rangle \\
&+\frac{1}{u(z)-u(z')}du_{z'} \langle \iota'^{{[v,v']}}(z)\bchi',\bchi\rangle \\
&+\frac{1}{u(z)-u(z')} du_{z} \langle \epsilon'^{{[v,v']}}(z')\bchi',\bchi\rangle \\
&=\langle \bchi',\iota^{v'}(z')\epsilon^v(z)\bchi\rangle \\
&-\frac{(v',v)du_{z'}du_{z}}{(u(z')-u(z))^2}\langle \bchi',\bchi\rangle \\
&-\frac{1}{u(z')-u(z)}du_{z} \langle \bchi',\iota^{[v',v]}(z')\bchi\rangle \\
&-\frac{1}{u(z')-u(z)} du_{z'} \langle \bchi', \epsilon^{[v',v]}(z)\bchi\rangle \\
&+\frac{(v,v')du_{z}du_{z'}}{(u(z)-u(z'))^2}\langle \bchi',\bchi\rangle \\
&+\frac{1}{u(z)-u(z')}du_{z'} \langle \bchi',\epsilon^{[v,v']}(z)\bchi\rangle \\
&+\frac{1}{u(z)-u(z')} du_{z} \langle \bchi',\iota^{[v,v']}(z')\bchi\rangle \\
&=\langle \bchi',\iota^{v'}(z')\epsilon^v(z)\bchi\rangle \\
\end{split}
\end{equation*}
\end{proof}

One checks that in particular, this induces the pairing of spaces $\lDp\bg$ and ${}_D\fg_-$:
\begin{equation}\label{integralpairingcurrents}
\langle \balpha', \balpha \rangle=-\frac{1}{2 \pi i} \int_{\gamma} (\balpha',d\balpha)
\end{equation}
where $\balpha$ is a $\bg$-valued meromorphic function with all poles in $D$ and $\balpha'$ has all poles in $D'$. The bracket $(,)$ in the integrand denotes the invariant symmetric pairing on the  Lie algebra $\bg$, extended in this case to act between a function and a one-form.  

\subsection{Pairings in the presence of insertions}\label{currentpairinginsert}

We adapt the notation of \S \ref{insertions} and \S \ref{currentpairing}. 

For $j=1,\dots,n$, let $W'_j$  denote the representation dual to $W_j$.  We denote by $(,)_j$ the pairing $W'_j \times W_j \to \bbC$; by definition,
\begin{equation*}
(\ttw',v(\ttw))_j=-(v(\ttw'),\ttw) \ \forall \ \ttw \in W_j, \ttw' \in W'_j \ and \ v\in \bg
\end{equation*}

Choose points $\bz=(z_1,\dots,z_n)$, all in $D$, distinct from each other and from $P$; similarly choose $\bz'=(z'_1,\dots,z'_n)$, all in $D'$, distinct from each other.

Let $\lD\cW =U \lD \fg  \otimes_{\bg} \{W_1 \otimes \dots \otimes W_n\}$, and $\lD\cW^z=U \lD \fg^z  \otimes_{\bg}  \{W_1 \otimes \dots \otimes W_n\}$. The spaces $\lDp \fg$, $\lDp \fg^{z'}$, $\lDp \cW'$ and $\lDp \cW'^{z'}$ are defined similarly. Given $v\in \bg$ and $z \ z' \in \bbP$, we define fields:
\begin{itemize} 
\item $\epsilon^v(z): \lD\cW  \to \lD\cW  \otimes K_z\ $,
\item $\iota^v(z): \lD\cW^z \to \lD\cW^z \otimes \{K(\sum_j z_j)\}_z\ $,
\item $\epsilon'^{v}(z'): \lDp\cW' \to  \lDp \cW' \otimes K_{z'}\ $, 
\item $\iota'^{v}(z'): \lDp\cW'^{z'}  \to \lDp\cW'^{z'}  \otimes \{K(\sum_{j} z'_{j})\}_{z'}\ $.   
\end{itemize}
Given a second point $\tP$, chosen to be in $D$, we also define, as in \S \ref{globalcurrent}, fields $\tiota$, etc.

We now state and prove an analogue of Proposition \ref{proposition-no-insertions}.

\begin{proposition}\label{proposition-with-insertions} There is a unique pairing between $\lDp\cW'$ and $\lD\cW$  such that $\langle \ttw'_1 \otimes \dots \otimes \ttw'_n, \ttw_1 \otimes \dots \otimes \ttw_n \rangle =\prod_j <\ttw'_j,\ttw_j>$, and for $v \in \bg$ and 
$\bpsi' \in \lDp\cW', \ \bpsi \in \lD\cW$, we have
\begin{itemize}
\item the equality of (meromorphic) sections of $K$ on $D$:
\begin{equation}\label{pairingwithinsertions1}
\langle \bpsi',\epsilon^v(z) \bpsi \rangle =-\langle \iota'^v(z) \bpsi', \bpsi \rangle 
\end{equation}
\item the equality of (meromorphic) sections of $K$ on $D'\setminus P$:
\begin{equation}\label{pairingwithinsertions2}
\langle \epsilon'^{v}(z')\bpsi',\bpsi\rangle =-\langle \bpsi',\iota^{v}(z') \bpsi \rangle 
\end{equation}
 which extends to $P$ as the equality
\begin{equation*}
\langle \tepsilon'^{v}(z')\balpha',\balpha\rangle =-\langle \balpha',\tiota^{v}(z') \balpha \rangle 
\end{equation*}
\end{itemize}
\end{proposition}

\begin{proof} Lemma \ref{pairingdef} holds in this case as well, so the induction works as before.
\end{proof}

We make explicit the above pairing in certain cases. First,
\begin{equation*}
\begin{split}
\langle \epsilon'^v(z') \ttw'_1 \otimes \dots \otimes \ttw'_n, \ttw_1 \otimes \dots \otimes \ttw_n \rangle &= -\langle \ttw'_1 \otimes \dots \otimes \ttw'_n, \iota^v(z') \ttw_1 \otimes \dots \otimes \ttw_n \rangle\\
&= -\langle \ttw'_1 \otimes \dots \otimes \ttw'_n, 
\sum_j \ttw_1 \otimes \dots \otimes \frac{v(\ttw_j)du_{z'}}{u(z')-u(z_j)} \otimes \dots \otimes \ttw_n \rangle\\
\end{split}
\end{equation*}
This implies the pairing
\begin{equation*}
\langle \balpha'\otimes_{U\bg} \ttw'_1 \otimes \dots \otimes \ttw'_n, \ttw_1 \otimes \dots \otimes \ttw_n \rangle = \langle  \ttw'_1 \otimes \dots  \otimes \ttw'_n, \sum_j
\ttw_1 \otimes \dots \otimes \{\balpha'(z_j)(\ttw_j)\} \otimes \dots \otimes  \ttw_n \rangle
\end{equation*}
A similar computation gives
\begin{equation*}
\langle \ttw'_1 \otimes \dots \otimes \ttw'_n, \balpha \otimes_{U\bg} \ttw_1 \otimes  \dots \otimes \ttw_n \rangle = \langle \sum_j \ttw'_1 \dots \otimes \{\balpha(z'_j)(\ttw_j)\} \otimes \dots  \otimes \ttw'_n,  \ttw_1 \otimes \dots \otimes  \ttw_n \rangle
\end{equation*}
It is useful to introduce the notation
\begin{equation*}
\balpha(\bz) (\ttw_1 \otimes \dots \otimes \ttw_n ) = \sum_j \ttw_1 \otimes \dots \otimes \{\balpha(z_j)(\ttw_j)\} \otimes \dots \otimes  \ttw_n
\end{equation*}
in terms of which we have
\begin{equation*}
\begin{split}
\langle \balpha'\otimes_{U\bg} \ttw'_1 \otimes \dots \otimes \ttw'_n, \ttw_1 \otimes \dots \otimes \ttw_n \rangle &= \langle \ttw'_1 \otimes \dots \otimes \ttw'_n,  \balpha'(\bz) (\ttw_1 \otimes \dots \otimes  \ttw_n) \rangle\\
\langle \ttw'_1 \otimes \dots \otimes \ttw'_n, \balpha \otimes_{U\bg} \ttw_1 \otimes \dots \otimes \ttw_n \rangle &= \langle \balpha(\bz') (\ttw'_1 \otimes \dots  \otimes \ttw'_n), 
\ttw_1 \otimes \dots \otimes  \ttw_n) \rangle
\end{split}
\end{equation*}

Finally
\begin{equation*}
\begin{split}
\langle \epsilon'^v(z') \ttw'_1 \otimes \dots \otimes \ttw'_n, \balpha \otimes_{U\bg}  \ttw_1 \otimes \dots \otimes \ttw_n \rangle &= -\langle \ttw'_1 \otimes \dots \otimes \ttw'_n, \iota^v(z') \balpha \ttw_1 \otimes \dots \otimes \ttw_n \rangle\\
&=  -\langle \ttw'_1 \otimes \dots \otimes \ttw'_n, [\iota^v(z'), \balpha] \ttw_1 \otimes \dots \otimes \ttw_n \rangle\\
&-\langle \ttw'_1 \otimes \dots \otimes \ttw'_n, \balpha \iota^v(z') \ttw_1 \otimes \dots \otimes \ttw_n \rangle\\
&=  \langle \ttw'_1 \otimes \dots \otimes \ttw'_n, (v, d\balpha(z')) \ttw_1 \otimes \dots \otimes \ttw_n \rangle\\
&-\langle \ttw'_1 \otimes \dots \otimes \ttw'_n, \iota^{[v, \balpha(z')]}(z') \ttw_1 \otimes \dots \otimes \ttw_n \rangle\\
&-\langle \ttw'_1 \otimes \dots \otimes \ttw'_n, [v, \balpha_{z'}] du_{z'} \ttw_1 \otimes \dots \otimes \ttw_n \rangle\\
&-\langle \ttw'_1 \otimes \dots \otimes \ttw'_n, \balpha \iota^v(z') \ttw_1 \otimes \dots \otimes \ttw_n \rangle\\
&=\langle \ttw'_1 \otimes \dots \otimes \ttw'_n, (v, d\balpha(z')) \ttw_1 \otimes \dots \otimes \ttw_n \rangle\\
&-\langle \ttw'_1 \otimes \dots \otimes \ttw'_n, \sum_j \ttw_1 \otimes \dots 
\otimes \frac{[v, \balpha(z')] (\ttw'_j) du_{z_j}}{u(z')-u(z_j)} \otimes \dots 
\otimes \ttw_n \rangle\\
&+\langle \ttw'_1 \otimes \dots \otimes \ttw'_n, \sum_j  \ttw_1 \otimes \dots  \otimes \frac{[v, \balpha(z')-\balpha(z_j)] (\ttw'_j) du_{z_j}}{u(z')-u(z_j)} \otimes \dots 
\otimes \dots \otimes \ttw_n \rangle\\
&-\langle \ttw'_1 \otimes \dots \otimes \ttw'_n, \balpha \iota^v(z') \ttw_1 \otimes \dots \otimes \ttw_n \rangle\\
&=\langle \ttw'_1 \otimes \dots \otimes \ttw'_n, (v, d\balpha(z')) \ttw_1 \otimes \dots \otimes \ttw_n \rangle\\
&-\langle \ttw'_1 \otimes \dots \otimes \ttw'_n, \sum_j \ttw_1 \otimes \dots 
\otimes \frac{[v, \balpha(z_j)] (\ttw'_j) du_{z_j}}{u(z')-u(z_j)} \otimes \dots 
\otimes \ttw_n \rangle\\
&-\langle \ttw'_1 \otimes \dots \otimes \ttw'_n, \balpha \iota^v(z') \ttw_1 \otimes \dots \otimes \ttw_n \rangle\\
\end{split}
\end{equation*}
which yields
\begin{equation*}
\begin{split}
\langle \balpha' \otimes_{U\bg} \ttw'_1 \otimes \dots \otimes \ttw'_n, \balpha \otimes_{U\bg} \ttw_1 \otimes \dots \otimes \ttw_n \rangle &= -<\balpha',\balpha>\langle \ttw'_1 \otimes \dots \otimes \ttw'_n, \ttw_1 \otimes \dots \otimes \ttw_n \rangle\\
&+\langle \ttw'_1 \otimes \dots \otimes \ttw'_n,  [\alpha',\alpha] (\bz)] (\ttw_1 \otimes \dots 
\otimes \ttw_n) \rangle\\
&-\langle \balpha (\bz') (\ttw'_1 \otimes \dots \otimes \ttw'_n),   \balpha' (\bz) (\ttw_1 \otimes \dots \otimes \ttw_n) \rangle\\
\end{split}
\end{equation*}
where the pairing $<\balpha',\balpha>$ is defined in equation (\ref{integralpairingcurrents}).

\pagebreak

\part{\textbf{{\Large A field theory associated to a one-dimensional lattice}}}

\section{\textbf{Preliminaries}}

Our goal is to describe a rational form of a field theory associated to a one-dimensional lattice $\sqrt{N} \bbZ$, where $N$ is a positive integer. We follow (up to a point) the notations of \S 4.2, \cite{F-BZ}.

Recall that the boson fields act on the symmetric algebra over $\bK$, the space of meromorphic forms of the \emph{second} kind - these are the ones with vanishing residues at all points. (An everywhere regular 1-form is said to be of the \emph{first} kind. Such a form is of course of the second kind as well, but in genus zero there is no nonzero regular form.) Let $\cK_{\bbZ}(K)$ denote the space of meromorphic forms with integer residues. We have (in genus zero) the diagram of abelian groups:
\begin{equation*}
\begin{CD}
@. @. \cK^* @. @. @.\\
@. @. @VVV  @. @. @.\\
0 @>>> \bK @>>> \cK_{\bbZ}(K) @>R>> \oplus_{p \in \bbP} \bbZ @>sum>>  \bbZ @>>>0\\
\end{CD}
\end{equation*}
where $\cK^*$ is the multiplicative group of non-zero meromorphic functions, $R$ is the residue map, and the arrow
$\cK^* \to \cK_\bbZ(K)$
$$
\bbeta \mapsto  \bbeta^{-1} d \bbeta
$$
maps a nonzero meromorphic function to its logarithmic derivative. This map is injective modulo $\bbC^*$ (constants); its image, which we denote $\cLK$ (where $\cL$ stands for ``logarithmic''), consists of forms with all poles of order one and integer residues; such a form is said to be of the \emph{third} kind.

We will work with a slight variant of the above diagram.

Let $\kappa$ be a square-root of the canonical bundle; this is a line-bundle together with an isomorphism $\kappa^2 \sim K$; on $\bbP$ this data is unique up to isomorphisms determined up to sign. Fix a point $P$, the point at infinity.  Then $\kappa$ has a unique connection $\nabla$ with a regular singular point at $P$. (In other words, if $\sigma$ is a section regular at $P$, $\nabla \sigma$ is a section of $\kappa \otimes K(P)$.) A coordinate $u$ with its pole at $P$ determines a section $\sqrt{du}$ (upto sign) with a simple pole at $P$, and $\nabla \sqrt{du}=0$ away from $P$. 

Let $\KK$ denote the set of non-zero meromorphic sections $\sigma$ of $\kappa$. This is a torsor over $\cK^*$, and the map
$$
\KK \to \cK_{\bbZ}(K), \ \ \ \sigma \mapsto \sigma^{-1} \nabla \sigma
$$
obeys, for $\bbeta \in \cK^*$,
$$
\bbeta \sigma \mapsto \bbeta^{-1}d\bbeta + \sigma^{-1} \nabla \sigma
$$
Consider the diagram
\begin{equation*}
\begin{CD}
@. @. \KK @. @. @.\\
@. @. @VVV  @. @. @.\\
0 @>>> \bK @>>> \cK_{\bbZ}(K) @>R>> \oplus_{p \in \bbP} \bbZ @>sum>>  \bbZ @>>>0\\
\end{CD}
\end{equation*}
The vertical arrow is again injective modulo multiplication by constant scalars, and equivariant with respect to the vertical arrow in the previous diagram. The image is  again the space of logarithmic forms, with $\sqrt{du}$ mapping to zero.

Let $\KKz$ denote the subset of $\KK$ consisting of meromorphic sections of $\kappa$ which are regular and nonvanishing at $z$.

\section{\textbf{The fields}}

We will now define fields acting on spaces which approximate the the symmetric algebra over $\cK(K)$, the space of \emph{all} meromorphic forms. Up to a twist (to be explained immediately below), this space will be spanned by vectors of the form
$$
\alpha^p \exp(\beta)
$$
where $\alpha$ is of the second kind and $\beta$ of the third kind. One  might expect that the exponential term will have to be defined in a completion of the symmetric algebra, but an alternative description is possible.

Fix a positive integer $N$. The choice of $P$ induces an additive character $\chi_P:\cK^* \to \bbZ$:
$$
\bbeta \mapsto Res_P(\bbeta^{-1} d\bbeta)
$$
and a compatible map, for which we retain the same notation, $\chi_P:\KK \to \bbZ$:
$$
\sigma \mapsto Res_P(\sigma^{-1} \nabla \sigma)
$$
Clearly, $\cK^*$ is divided into cosets by $\chi_P$:
$$
\cK^*= \sqcup_{l\in \bbZ} \cK^*_l
$$
with $\cK^*_l=\{\bbeta|\chi_P(\bbeta)=l\}$; and similarly
$$
\KK= \sqcup_{l\in \bbZ} \KKl
$$

Consider the action of $\bbC^*$ on $\KK$:
\begin{equation*}
(\sigma,t) \mapsto t\sigma
\end{equation*}
We set
\begin{itemize}
\item $\tbbW_{P,-,l}=$ maps $\rho:\KKl \to \bbW_{P,-}$ which are nonzero on finitely many fibres of $\KK \to \cLK$, and equivariant in the sense $\rho(t\sigma)=t^{Nl} \rho(\sigma)$.
\item $\tbbW_{P,-.l}^z=$ maps $\KKlz \to \bbW^z_{P,-}$ as above.
\end{itemize}
In this work we do not intend to study the dependence of the theory on the choice of base-point $P$. So we drop the suffixes and from now on simply write $\tbbW_{P,-,l}=\tbbW_{l}$, $\tbbW^z_{P,-,l}=\tbbW_{l}$.

Set $\tbbW=\oplus_l \tbbW_{l}$, $\tbbW^z=\oplus_l \tbbW^z_{l}$. Note that if $N$ were zero, then $\tbbW=\bbW_{P,l} \otimes \bbC[\cLK]$, where $\bbC[\cLK]$ is the group algebra of the additive group $\cLK$. (Similarly, if $N=0$, $\tbbW^z=\bbW^z \otimes \bbC[\cLK^z]$.) Given $\ttv \in \bbW_{l}$ and $\sigma \in \KKl$, we will denote by $\ttv \hotimes E[\sigma]$ the map $\KKl \to \bbW_{P,-}$ which is such that 
$$
t \sigma \mapsto t^{Nl} \ttv, \ t \in \bbC^*
$$
and is zero on other elements of $\KKl$. Note that therefore
$$
\ttv \hotimes E[\sigma]=t^{Nl} \ttv \hotimes E[t\sigma],
 \ t \in \bbC^*
$$

We define fields in the complement of $P$. First, a preliminary bit of notation. Let $\digamma$ denote the space of sections of $K$ regular away from $P$  and are nowhere-vanishing elsewhere. Every such nonzero section has a double pole at $P$, and any two are proportional, so $\digamma$ is a one-dimensional vector space, and the evaluation map  $\digamma \to K_z$ identifies the two lines as long as $z\ne P$. In the following, given a form $\alpha$ regular at $z$, we will therefore regard $\alpha(z)$ as an element of $\digamma$. In particular, given a coordinate $u$, we have $du_z  = du_{z'} \in \digamma$.   
\begin{enumerate}
\item Define $\epsilon(z):\tbbW \to \tbbW \otimes \digamma$ by multiplying by $\epsilon_z$ (in the factor $\bbW$) as earlier.
\item Define $\iota(z):\tbbW^z \to \tbbW \otimes \digamma$ inductively by
\begin{equation*}
[\iota(z),\balpha] = -d\balpha(z),\  \balpha \in \mathfrak{m}^z_P
\end{equation*}
ie., for $\balpha$ regular and vanishing at $P$ and regular at $z$, and
\begin{equation*}
\iota(z) \mathbf{1} \hotimes E[\sigma]= -\mathbf{1} \hotimes E[\sigma]  \sqrt{N} \sigma(z)^{-1}\nabla \sigma(z),
\end{equation*}
for $\sigma$ regular and non-vanishing at $z$.
\suspend{enumerate}

\textbf{Notation:} For $\clambda \in \bbZ$, set $\lambda=\sqrt{N}\clambda$. 

\resume{enumerate}
\item Define the field $\flat^+_\lambda: \tbbW_{l}  \to \tbbW_{l+\clambda}\otimes \digamma^{b_{l,\clambda}}$ by requiring
\begin{equation}\label{flatplus}
\begin{split}
[\flat^+_\lambda(z),\balpha] &=0\\
\flat^+_\lambda(z) \mathbf{1} \hotimes E[\sigma] &= \mathbf{1} \hotimes E[(u-u(z))^{-\clambda}\sigma] du_z^{b_{l,\clambda}}\\
\end{split}
\end{equation}
Here $b_{l,\clambda}$ is an integer, which we will fix immediately in terms of $N$ and $l$.
\item Define the field $\flat^-_\lambda: \tbbW_{l}^z \to \tbbW_{l} \otimes \kappa^{N\clambda} $ by requiring that for $\balpha \in \mathfrak{m}^z_P$ and $\sigma$ regular at $z$:
\begin{equation*}
\begin{split}
[\flat^-_\lambda(z),\balpha] &= -\lambda \balpha (z) \flat^-_\lambda(z),\ \text{and} \\ 
\flat^-_\lambda(z) \mathbf{1} \hotimes  E[\sigma]&=  \mathbf{1} \hotimes  E[\sigma] \sigma(z)^{N\clambda}\\
\end{split}
\end{equation*}
\end{enumerate}
To determine $b_{l,\clambda}$, suppose $u=t\cu$, with $t\in \bbC^*$. Then 
\begin{equation*} 
\begin{split}
\mathbf{1} \hotimes E[(u-u(z))^{-\clambda}\sigma] du_z^{b_{l,\clambda}}&=\mathbf{1} \hotimes E[t^{-\clambda}(\cu-\cu(z))^{-\clambda}\sigma] t^{b_{l,\clambda}}d\cu_z^{b_{l,\clambda}}\\
&= t^{N\clambda(l+\clambda)}
\times t^{-N\clambda(l+\clambda)}
\mathbf{1} \hotimes E[t^{-\clambda}(\cu-\cu(z))^{-\clambda}\sigma] t^{b_{l,\clambda}}d\cu_z^{b_{l,\clambda}}\\
&= t^{N\clambda(l+\clambda)+{b_{l,\clambda}}}
\mathbf{1} \hotimes E[(\cu-\cu(z))^{-\clambda}\sigma]
d\cu_z^{b_{l,\clambda}}\\
 \end{split}
\end{equation*}
For this to be well-defined independent of the choice of coordinate, we need
\begin{equation}\label{ablambda}
N\clambda(l+\clambda)+b_{l,\clambda}=0
\end{equation}
There is a further ambiguity in the definitions of $\flat^+$  and $\flat^-$,  in that they are not well defined under the equivalence $\mathbf{1} \hotimes E [\sigma] = t^{Nl} \mathbf{1} \hotimes E [t\sigma]$. For the moment, we resolve this by requiring that $\sigma$ is nonvanishing at $z$ and $\sigma(z)=1$.

Note that the first of the equations (\ref{flatplus}) implies that $[\flat^+_\lambda(z),e(z')]=0$ for $z \ne z'$. One checks easily that if $z' \ne z$,
\begin{equation}\label{iflatplus}
\begin{split}
[\iota(z'), \flat^+_\lambda(z)]  \mathbf{1} \hotimes  E[\sigma]& =\mathbf{1} \hotimes  E[(u-u(z))^{-\clambda} \sigma] \frac{\lambda du}{u(z')-u(z)} du^{b_{l,\clambda}} \\
\ \text{and}\ [\flat^+_{\lambda'}(z'),\flat^+_\lambda(z)]&=0
\end{split}
\end{equation}

\begin{lemma} We have, for $z \ne z'$,
\begin{equation*}
\begin{split}
[\iota(z'), \flat^+_\lambda(z)] &= \frac{\lambda du_{z'}}{u(z')-u(z)} \flat^+_\lambda(z)\\
[\iota(z'),\flat^-_\lambda(z)]&=0\\
[\epsilon(z'), \flat^+_\lambda(z)] &= 0\\
[\epsilon(z'), \flat^-_\lambda(z)] &=-\flat^-_\lambda(z)\frac{\lambda du}{u(z')-u(z)}\\
\end{split}
\end{equation*}
\end{lemma}

\begin{proof} The first equation is a consequence of (\ref{iflatplus}). The third and fourth equations follow from the definitions. As for the second, note first that
\begin{equation*}
\begin{split}
[[\iota(z'),\flat^-_\lambda(z)],\balpha] &= [[\iota(z'),\balpha],\flat^-_\lambda(z)]]+[\iota(z'),[\flat^-_\lambda(z),\balpha]]\\
&=[\iota(z'),\flat^-_\lambda(z)]\clambda d\balpha(z)
\end{split}
\end{equation*}
On the other hand
\begin{equation*}
\begin{split}
[\iota(z'),\flat^-_\lambda(z)] E[\sigma]&=  \iota(z') E[\sigma] \sigma(z)^{\clambda N}+\flat^-_\lambda(z) E[\sigma] \sigma(z')^{-1} d\sigma(z')  \\
&=-\sigma(z')^{-1}d\sigma(z') E[\sigma] \sigma(z)^{\clambda N}+E[\sigma] \sigma(z)^{\clambda N} \sigma(z')^{-1}d\sigma(z')\\
&=0
\end{split}
\end{equation*}
\end{proof}

Consider now the product 
$$
V_\lambda(z) \equiv \flat^+_\lambda(z)\flat^-_\lambda(z) : \tbbW^z_l \to \tbbW_{l+\clambda} \otimes \kappa^{N\clambda} \otimes \digamma^{-N\clambda(l+\clambda)}
$$ 
We have
\begin{equation*}
\begin{split}
[V_\lambda(z),\balpha] &= \lambda \balpha (z) V_\lambda(z),\\
V_\lambda(z) \{\mathbf{1} \hotimes  E[\sigma]\} &=  \mathbf{1} \hotimes E[(u-u(z))^{-\clambda}\sigma]\sigma(z)^{\clambda N} du_z^{b_{l,\clambda}}
\end{split}
\end{equation*}
with $\balpha$ and $\sigma$ as  above. We first deal with the ambiguities in the definitions of $\flat^+$  and $\flat^-$ by checking that they ``cancel'' in the definition of $V_\lambda(z)$: 
\begin{equation*}
\begin{split}
V_\lambda(z) \{t^{Nl}\mathbf{1} \hotimes  E[t\sigma]\} &=  (t)^{Nl} \mathbf{1} \hotimes E[(u-u(z))^{-\clambda}t\sigma]t^{\clambda N}\sigma(z)^{\clambda N} du_z^{b_{l,\clambda}}\\
&= t^{\clambda N-N\clambda} (t)^{a(l+\clambda)} \mathbf{1} \hotimes E[(u-u(z))^{-\clambda}t\sigma]t^{\clambda N}\sigma(z)^{\clambda N} du_z^{b_{l,\clambda}}\\
&=  \mathbf{1} \hotimes E[(u-u(z))^{-\clambda}\sigma]\sigma(z)^{\clambda N} du_z^{b_{l,\clambda}}
\end{split}
\end{equation*}

\begin{proposition}
For $z'\ne z$ we have
\begin{equation*}
\begin{split}
[j(z'),V_{\lambda}(z)]&=0\\
V_{\lambda'}(z')V_\lambda(z)&=(-1)^{\lambda\lambda'} V_\lambda(z)V_{\lambda'}(z')
\end{split}
\end{equation*}
\end{proposition}

\begin{proof} The first relation is straightforward.  Consider next
\begin{equation*}
\begin{split}
V_{\lambda'}(z')V_\lambda(z) \mathbf{1} \hotimes E[\sigma] 
&= \mathbf{1} \hotimes E[(u'-u(z'))^{-\clambda'}(u-u(z))^{-\clambda}\sigma]\\
&\ \ \ \sigma(z')^{\clambda' N} \sigma(z)^{\clambda N}  (u(z')-u(z))^{-\clambda' \clambda N} du_{z'}^{b_{l+\clambda,\clambda'}}du_z^{b_{l,\clambda}}\\
\end{split}
\end{equation*}
We use (\ref{ablambda}) to write
$$
du_{z'}^{b_{l+\clambda,\clambda'}}du_z^{b_{l,\clambda}}=
du_{z'}^{-N\clambda'(l+\clambda')}du_z^{-N\clambda(l+\clambda)} du_{z'}^{-N\clambda' \clambda}
$$
which yields 
\begin{equation*}
\begin{split}
V_{\lambda'}(z')V_\lambda(z) \mathbf{1} \hotimes E[\sigma] 
&= \mathbf{1} \hotimes E[(u'-u(z'))^{-\clambda'}(u-u(z))^{-\clambda}\sigma]\\
\times \sigma(z')^{\clambda' N} \sigma(z)^{\clambda N}  & (u(z')-u(z))^{-\clambda' \clambda N} du_{z'}^{-N\clambda'(l+\clambda')}du_z^{-N\clambda(l+\clambda)} du_{z'}^{-N\clambda' \clambda} \\
\end{split}
\end{equation*}
Recalling that we identify $du_z$ and $du_{z'}$ in $\digamma$, we get
$$
V_{\lambda'}(z')V_\lambda(z) \mathbf{1} \hotimes E[\sigma] = (-1)^{-\clambda' \clambda N}
V_{\lambda}(z)V_{\lambda'}(z') \mathbf{1} \hotimes E[\sigma] 
$$
The second equation of the Proposition follows by an easy induction. (Note that $\clambda' \clambda N = \lambda' \lambda$.)
\end{proof}

A detailed exploration of the above construction, including extensions to higher dimensional lattices, will appear in a sequel. For the moment, we briefly consider the case $N=1$. We  have the following fields:
\begin{enumerate}
\item $j(z):\tbbW^z \to \tbbW \otimes \digamma$
\item $V_\lambda(z): \tbbW^z_l \to \tbbW_{l+\lambda} \otimes \kappa^{\lambda} \otimes \digamma^{-\lambda(l+\lambda)}$
\end{enumerate}
which are mutually local with respect to each other. In particular, the fields
$$
V_{\pm 1}(z): \tbbW^z_l \to \tbbW_{l \pm 1} \otimes \kappa^{\pm1} \otimes  \digamma^{-1 \mp l}
$$
obey: given distinct points $z,z'$, we have
\begin{equation*}
\begin{split}
V_{+1}(z')V_{-1}(z) &=-V_{-1}(z)V_{+1}(z')\\
V_{+1}(z')V_{+1}(z) &=V_{+1}(z)V_{+1}(z')\\
V_{-1}(z')V_{-1}(z) &=V_{-1}(z)V_{-1}(z')\\
\end{split}
\end{equation*}
Note the minus sign in the first equation -- \emph{the fields $V_{\pm 1}$ anticommute with each other.}

Given any meromorphic function $\phi$, consider the operator $H_P^\phi=\Phi_P$, defined formally by
\begin{equation}
H^\phi_P \equiv \Phi_P =`` \frac{1}{2\pi i} \int_{\gamma} \phi(z) j(z)''
\end{equation}
where $\gamma$ is sufficiently close to $P$. If $\phi=1$ (constant), we see that
\begin{equation}\label{zeromodes}
\begin{split}
\Phi_P (\mathbf{1} \hotimes E[\sigma])&= \{-\frac{1}{2\pi i}\int_\gamma \sigma(z)^{-1}\nabla \sigma(z)\}
\mathbf{1} \hotimes E[\sigma]\\
&=l \mathbf{1} \hotimes E[\sigma]\\
\end{split}
\end{equation}
if $\sigma \in \cK^*_l(\kappa)$. 

\pagebreak

\part{\textbf{{\Large Fields on an arbitrary smooth curve}}}

The next sections are heavily influenced by the papers of Raina, in particular, \cite{Ra1}. (See also \cite{T1} and \cite{T2}).

Let $\bbX$ be a smooth complex projective curve of genus $g$. Let $K$ denote its canonical bundle, and let $\xi$ be a line bundle on $\bbX$ of degree $g-1$ with no nontrivial sections. Consider, on the product $\bbX \times \bbX$, the line bundle $K \boxtimes \xi^{-1}\otimes \xi(\Delta)$, where $\Delta$ is the diagonal; note that its restriction to $\Delta$ is (by the adjunction formula) canonically trivial. We will use the following lemma, which is key to Raina's approach. (The subscript in $\sSz$ below stands for Szego.)

\begin{lemma}\label{Szegolemma} There is a unique section $\sSz$ of $K\otimes \xi^{-1} \boxtimes  \xi (\Delta)$ that restricts to the constant function 1 on the diagonal. 
\end{lemma}

\section{\textbf{Arbitrary genus: the neutral fermion}}\label{section: Arbitrary genus: the neutral fermion}

This is perhaps the simplest example of a CFT;  it is canonically associated to a smooth projective curve together with a noneffective theta characteristic.

In this section we will suppose that $\xi$ is a square-root of the canonical bundle, and denote it by $\kappa$. Fix an isomorphism $\kappa^2 \to K$. In classical terminology, $\kappa$ is a theta-characteristic; we will assume that it has no nontrivial sections, i.e., that it is a \emph{non-effective} one. (Such a theta-characteristic always exists.  Note that with this choice, $\sSz$ is a section of $\kappa \boxtimes  \kappa (\Delta)$.  Note also that the \emph{square} of the section $\sSz$,
\begin{equation*}
\rho_B \equiv \sSz^2
\end{equation*}
defines a section of $K\boxtimes K(2\Delta)$ which restricts to $1$ on the diagonal.

Consider the exchange map $\bbX \times \bbX \to \bbX \times \bbX,\  (z,w) \mapsto (w,z)$, and lift this to sections of $\kappa \boxtimes  \kappa (\Delta)$ as follows. For any nonconstant meromorphic function $u$ on $\bbX$, let $U$ be the affine open set where $u$ is regular and $du$ nonzero; let $\sqrt{du}$ be one of the two sections of $\kappa$ such that $\sqrt{du} \otimes \sqrt{du} = du$. Let $f \sqrt{du} \boxtimes \sqrt{du}$ be a section of $\kappa \boxtimes  \kappa (\Delta)$ on $U \times U$; in other words,  $f$ is a function on $U \times U$ regular except possibly for a simple pole along the diagonal $\Delta_U$. The lift is defined by demanding $f(z,w) \mapsto f(w,z)$ for all $u$ and $f$.

\begin{lemma} The section $\sSz$ is odd under the exchange map.
\end{lemma}

Given $z\in \bbX$, set 
$$
\psi^e_{z} =\sSz|_{\bbX \times z} \in H^0(\bbX,\kappa(z)) \otimes \kappa_z
$$
For any line bundle $\eta$ on $\bbX$, recall that $\cK(\eta)$ denotes the space of meromorphic sections; and given $z \in \bbX$ the subspace of sections regular at $z$ is $\cK^z(\eta)$. Set
\begin{equation*}
\begin{split}
\bF &= \bigwedge^* \cK(\kappa) \\
\bF^{z} &= \bigwedge^* \cK^z(\kappa) \\
\end{split}
\end{equation*}
Given $z\in \bbX$, define

\noindent (1) $\psi_e(z):\bF \to \bF \otimes \kappa_z$ by 
\begin{equation*}
\psi_e(z) (\eta_1 \wedge \dots \wedge \eta_p) =\psi^e_z \wedge \eta_1 \wedge \dots \wedge \eta_p
\end{equation*}

\noindent (2) $\psi_i(z):\bF^{z} \to \bF^{z} \otimes \kappa_z$ by 
\begin{equation*}
\psi_i(z) (\eta_1 \wedge \dots \wedge \eta_p) =\sum_j (-1)^{j-1}   \eta_1 \wedge \dots \widehat{\eta_j} \dots \wedge \eta_p \otimes  \eta_j(z)
\end{equation*}

These will be \emph{fermionic} fields; in other words, when necessary we consider anti-commutators (for which we will use the notation $[-,-]_+$) rather than commutators. (In tensor products, the lines $K_z, \ \xi_z$ etc.,  $(z \in \bbX)$  will be taken to ``commute without sign'' with all factors.)
Define the "chiral neutral fermion" field $\psi$ thus:
\begin{equation*}
\psi(z)=\psi_i(z)+\psi_e(z):\bF^{z} \to \bF \otimes \kappa_z
\end{equation*}

Let $z_1$ and $z_2$ be two distinct points and consider the anti-commutator:
\begin{equation*}
\psi(z_1) \circ \psi(z_2) + \psi(z_2) \circ \psi(z_1):\bF^{z_1,z_2} \to \bF \otimes \kappa_{z_1} \otimes \kappa_{z_2}
\end{equation*}
where $\bF^{z_1,z_2} = \bF^{z_1} \cap \bW^{z_2}$. This is equal to
\begin{equation*}
\begin{split}
\psi_i(z_1) \circ \psi_e(z_2) + \psi_e(z_1) \circ \psi_i(z_2)&+\psi_i(z_2) \circ \psi_e(z_1) + \psi_e(z_2) \circ \psi_i(z_1) \\
&=\sigma_{Sz} (z_1,z_2)+\sigma_{Sz} (z_2,z_1)\\
&=0
\end{split}
\end{equation*}
In addition, we easily check that 
\begin{equation*}
\lim_{z_2 \to z_1} \psi(z_2) \circ \psi(z_1) -\sigma_{Sz} (z_2,z_1) =0
\end{equation*}
and this proves that $\psi$ is local with respect to itself.

Let $z_1,\dots, z_n$ be distinct points on $\bbX$; we define the $n$-point function $\langle \psi(z_1)\psi(z_2)\dots \psi(z_n)\rangle$ by the condition
\begin{equation*}
\psi(z_1)\psi(z_2)\dots \psi(z_n) \mathbf{1} = \langle \psi(z_1)\psi(z_2)\dots \psi(z_n)\rangle + \text{terms in} \ \oplus_{m>0} \bigwedge^m \cK(\kappa) 
\end{equation*}
We imitate the computation in the case of the boson:
\begin{equation*}
\begin{split}
\langle \psi(z_1)\psi(z_2)\dots \psi(z_n)\rangle &=\langle (\psi_e(z_1)+\psi_i(z_1))\psi(z_2) \dots \psi(z_n) \rangle \\
&=\langle \psi_i(z_1)\psi(z_2)\dots \psi(z_n) \rangle \\
&=\langle \sum_{l=2}^n (-1)^{l-1}  \psi(z_2)\dots [\psi_i(z_1),\psi(z_l)]_+ \dots \psi(z_n) \rangle \\
&= \sum_{l=2}^n (-1)^{l-1} \sigma_{Sz} (z_1,z_l) \langle \psi(z_2)\dots \widehat{\psi(z_l)} \dots \psi(z_n) \rangle \\
\end{split}
\end{equation*}
where the hat marks a term to be omitted.  By induction, we get the expression given by (fermionic) Wick's theorem:
\begin{equation*}
\langle \psi(z_1)\psi(z_2)\dots \psi(z_n)\rangle=
\begin{cases}
\begin{aligned}
& \ \ \ 0 \ \ \ \ \ \text{if $n$ is odd, and}\\
&\{\sum_{\substack{\rho}}  \ \ (-1)^{sign(\rho)} \prod_{\substack{j=1,\dots,m}}  \sSz (z_{\rho(2j-1)},z_{\rho(2j)})\}\\
& \ \ \ \ \ \ \ \ \text{if $n=2m$ is even}\\
\end{aligned}
\end{cases}
\end{equation*}
where $\rho$ runs over permutations of$\{1,\dots,n\}$.

\subsection{Mode expansions in genus zero}

Specialising to genus zero, let $\bbX=\bbP$ and choose a square root $\kappa$ of $K$; this is unique up to an isomorphism, which is in turn also unique up to sign.  Choose a coordinate $u$ on $\bbP$ (with its zero at $O$ and pole at $P$ following the notation agreed in \S \ref{Notation}). Choose a square root  $\sqrt{du}$. This will be a section of $\kappa$ with a simple pole at $P$.

Let $H^0(\bbP \setminus O, \kappa)$ denote the space of (algebraic) sections of $\kappa$ on $\bbP \setminus O$ and let $\lO \bF$ denote the exterior algebra over $H^0(\bbP \setminus O,K)$:
\begin{equation*}
\lO \bF=\bbC \oplus H^0(\bbP \setminus O, \kappa) \oplus \bigwedge^2 H^0(\bbP \setminus O, \kappa) \oplus ......
\end{equation*}

For $l \in \bbZ+1/2$, define operators $\psi_{l}:  \bF \to  \bF$ as follows. If $l>0$, let
\begin{equation*}
\psi_{-l}[\eta] = u^{-l-1/2}\sqrt{du} \wedge \eta 
\end{equation*}
(Note that $\psi_{-l}$ leaves the flag $\lO \bF  \subset \bF$ invariant.) For $l < 0$, define $\psi_{l}: \cK(\kappa) \to \bbC$ by setting
\begin{equation*}
\psi_{l} [u^{-m-1/2}\sqrt{du}] =  \delta_{l,m}, \ m > 0,\ m \in \bbZ+1/2
\end{equation*}
and extending by linearity to  $\cK(\kappa)$.  Extend to $\psi_{l}:\bF \to \bF$ by setting $\psi_l ({\mathbf 1})=0$ and
\begin{equation*}
\psi_{l} (\eta_1 \wedge \dots \wedge \eta_p) =  \sum_i (-1)^{i-1}({\eta}_1 \wedge \dots \widehat{{{\eta}_i}} \dots \otimes^s {\eta}_p) \psi_{l}(\eta_i)
\end{equation*}
The operators $\psi_{l}, \  l \in \bbZ+1/2$ are all defined on $\bF$ and satisfy the commutation relations $[\psi_{l},\psi_{m}]_+=\delta_{l+m,0}$. 

Let $z \in \bbP \setminus \{O,P\}$ and set $u(z)=w \in \bbC$. The series $\sum_{l > 0} \psi_{-l} w^{l-1/2}[\mathbf{1}]$ converges uniformly on the subsets $\{u||u|\ge |w|+\delta\}$ (for $\delta >0$) to $\frac{du}{(u-w)} \in \cK(\kappa)$. Formally,
\begin{equation*}
\sum_{l > 0} \psi_{-l} w^{l-1/2}  = \hat{\psi}^e(z)
\end{equation*}
On the other hand,  $\sum_{l >0} \psi_{l} w^{-l-1/2}[u^{-m-1/2}du]=w^{-m-1/2}$, so that if $\eta = f \sqrt{du}\in  H^0(\bbP \setminus O,\kappa)$,
\begin{equation*}
\sum_{l >0 } \psi_{l} w^{-l-1/2} [\eta] = f(z)
\end{equation*}
(This is actually a finite sum for any given $\eta$.)
Thus
\begin{equation}
\sum_{l > 0} \psi_l w^{-l-1/2}= \hpsi(z)|_{{}_{\lO \bF\subset \bF^z}} : \lO \bF \to \lO \bF
\end{equation}
We have used the notation $\psi^i(z)=\hpsi^i(z) \sqrt{du}_z, \  \psi^e(z)=\hpsi^e(z) \sqrt{du}_z$.

We summarise the definition of the operators $\psi_{l}$:
\begin{equation*}
\hpsi_{l}=
\begin{cases}
\begin{aligned}
& u^{l-1/2}\sqrt{du} \wedge \ \ \ \ \ \ \ \ \ \ \ \ \  l < 0\\
& u^{-m-1/2} \sqrt{du} \mapsto  \delta_{l+m} \ \ l >0,\ \text{extended as an antiderivation}\\
\end{aligned}
\end{cases}
\end{equation*}

\section{\textbf{Arbitrary genus: the $b-c$ system}}\label{sectionbcsystem}

Consider now a general $\xi$ of degree $g-1$ with no nontrivial sections. (Such a $\xi$ always exists.) Given $z\in \bbX$, set
\begin{itemize}
\item $b^e_{z} \in H^0(\bbX,K \otimes \xi^{-1}(z)) \otimes \xi_z=\sSz|_{\bbX \times z}$, and
\item $c^e_z \in H^0(\bbX,\xi(z)) \otimes K_z \otimes \xi^{-1}_z=\sSz|_{z \times \bbX}$
\end{itemize}
Set
\begin{equation*}
\begin{split}
\bW_{bc} &=\bigwedge^* \cK(K \otimes \xi^{-1}) \otimes \bigwedge^* \cK(\xi) \\
\bW^{z|}_{bc} &= \bigwedge^* \cK^z(K \otimes \xi^{-1}) \otimes \bigwedge^* \cK(\xi) \\
\bW^{|z}_{bc} &= \bigwedge^* \cK(K \otimes \xi^{-1}) \otimes \bigwedge^* \cK^z(\xi) \\
\end{split}
\end{equation*}
Given $z\in \bbX$, define

\noindent (1) $b_e(z):\bW_{bc} \to \bW_{bc} \otimes \xi_z$ by 
\begin{equation*}
b_e(z) (\beta_1 \wedge \dots \wedge \beta_p \otimes \gamma_1 \wedge \dots \wedge \gamma_q) =b^e_z \wedge \beta_1 \wedge \dots \wedge \beta_p \otimes \gamma_1 \wedge \dots \wedge \gamma_q
\end{equation*}

\noindent (2)  $c_e(z):\bW_{bc} \to \bW_{bc} \otimes K_z \otimes \xi^{-1}_z$ by 
\begin{equation*}
c_e(z) (\beta_1 \wedge \dots \wedge \beta_p \otimes \gamma_1 \wedge \dots \wedge \gamma_q) = (-1)^p \beta_1 \wedge \dots \wedge \beta_p  \otimes c^e_z \wedge \gamma_1 \wedge \dots \wedge \gamma_q
\end{equation*}

\noindent (3) $b_i(z):\bW^{z|}_{bc} \to \bW^{z|}_{bc} \otimes \xi_z$ by 
\begin{equation*}
b_i(z) (\beta_1 \wedge \dots \wedge \beta_p \otimes \gamma_1 \wedge \dots \wedge \gamma_q) = \sum_j (-1)^{p+j-1} \beta_1 \wedge \dots \wedge \beta_p \otimes \gamma_1 \wedge \dots \widehat{\gamma_j} \dots \wedge \gamma_q \otimes  \gamma_j(z)
\end{equation*}

\noindent (4) $c_i(z):\bW^{|z}_{bc} \to \bW^{|z}_{bc} \otimes K_z \otimes \xi^{-1}_z$ by 
\begin{equation*}
c_i(z) (\beta_1 \wedge \dots \wedge \beta_p \otimes \gamma_1 \wedge \dots \wedge \gamma_q) = -\sum_j (-1)^{j-1} \beta_1 \wedge \dots  \widehat{\beta_j} \dots \wedge \beta_p \otimes \gamma_1 \wedge \dots \wedge \gamma_q \otimes \beta_j(z) 
\end{equation*}
 
Let $z_1,z_2$ be distinct points. Note the anticommutation rules:
\begin{enumerate}
\item $[b_e(z_1),b_e(z_2)]_+=0$, $[b_i(z_1),b_i(z_2)]_+=0$.
\item $[c_e(z_1),c_e(z_2)]_+=0$, $[c_i(z_1),c_i(z_2)]_+=0$.
\item $[b_i(z_1),c_e(z_2)]_+=c^e_{z_2} (z_1)=\sSz(z_2,z_1) \in 
K_{z_2} \otimes \xi^{-1}_{z_2} \otimes \xi_{z_1}$.
\item $[c_i(z_1),b_e(z_2)]_+=-b^e_{z_2} (z_1)=-\sSz(z_1,z_2) \in 
K_{z_1} \otimes \xi^{-1}_{z_1} \otimes \xi_{z_2}$
\end{enumerate}

Define fields $b,\ c$ thus:
\begin{equation*}
\begin{split}
b(z)&=b_i(z)+b_e(z):\bW^{z|}_{bc} \to \bW_{bc} \otimes \xi_z\\
c(z)&=c_i(z)+c_e(z):\bW^{|z}_{bc} \to \bW_{bc} \otimes K_z \otimes \xi^{-1}_z\\
\end{split}
\end{equation*}

One checks easily that each of the fields $b$ and $c$ is  local with respect to itself; we will now verify that these two are mutually local.  Let $z_1$ and $z_2$ be two distinct points and consider the anti-commutator:
\begin{equation*}
b(z_1) \circ c(z_2) + c(z_2) \circ b(z_1):\bW^{z_1|z_2}_{bc} \to \bW_{bc} \otimes \xi_{z_1} \otimes K_{z_2}  \otimes \xi^{-1}_{z_2}
\end{equation*}
where $\bW^{z_1|z_2}_{bc} = \bW^{z_1|}_{bc} \cap \bW^{|z_2}_{bc}$. This is equal to
\begin{equation*}
\begin{split}
b_i(z_1) \circ c_e(z_2) + b_e(z_1) \circ c_i(z_2)&+c_i(z_2) \circ b_e(z_1)+c_e(z_2) \circ b_i(z_1)\\
&=\sigma_{Sz} (z_2,z_1)-\sigma_{Sz} (z_2,z_1)\\
&=0
\end{split}
\end{equation*}
Consider now the limit:
\begin{equation*}
\begin{split}
\lim_{z_2 \to z_1} b(z_2) \circ c(z_1) -\sigma_{Sz} (z_1,z_2)
&=\lim_{z_2 \to z_1} \{b_i(z_2) \circ c_i(z_1) + b_e(z_2) \circ c_e(z_1)\\
&+b_i(z_2) \circ c_e(z_1)+ b_e(z_2) \circ c_i(z_1)-\sigma_{Sz} (z_1,z_2)\}\\
&=\lim_{z_2 \to z_1} \{b_i(z_2) \circ c_i(z_1) + b_e(z_2) \circ c_e(z_1)\\
&+b_e(z_2)c_i(z_1)-c_e(z_1)b_i(z_2)\}\\
&=b_i(z_1)c_i(z_1)+b_e(z_1)c_e(z_1)+b_e(z_1)c_i(z_1)-c_e(z_1)b_i(z_1)
\end{split}
\end{equation*}
Define the field $\sfb$ by
\begin{equation*}
\sfb(z)=b_i(z) \circ c_i(z) + b_e(z) \circ c_e(z)+b_e(z) \circ c_i(z)-c_e(z) \circ b_i(z)
\end{equation*}

Set
\begin{equation*}
\begin{split}
\bW &=  \bigwedge^* \cK(K \otimes \xi^{-1}) \otimes \bigwedge^* \cK(\xi) \\
\bW^{z} &=   \bigwedge^* \cK^z(K \otimes \xi^{-1}) \otimes \bigwedge^* \cK^z(\xi) \\
\end{split}
\end{equation*}

Clearly the field $\sfb$ acts from $\bW^z$ to $\bW \otimes K_z$. We check now that it is a bosonic field that is local with respect to itself. Let $z_1$ and $z_2$ be two distinct points and consider the commutator: $[\sfb(z_1),\sfb(z_2)]$.  This is equal to
\begin{equation*}
\begin{split}
&[b_i(z_1)c_i(z_1),b_e(z_2)c_e(z_2)]+[b_e(z_1)c_e(z_1),b_i(z_2)c_i(z_2)]\\
+&[b_i(z_1)c_i(z_1),b_e(z_2)c_i(z_2)]-[b_i(z_1)c_i(z_1),c_e(z_2)b_i(z_2)]\\
+&[b_e(z_1)c_e(z_1),b_e(z_2)c_i(z_2)]-[b_e(z_1)c_e(z_1),c_e(z_2)b_i(z_2)]\\
-&[b_i(z_2)c_i(z_2),b_e(z_1)c_i(z_1)]+[b_i(z_2)c_i(z_2),c_e(z_1)b_i(z_1)]\\
-&[b_e(z_2)c_e(z_2),b_e(z_1)c_i(z_1)]+[b_e(z_2)c_e(z_2),c_e(z_1)b_i(z_1)]\\
+&[b_e(z_1)c_i(z_1),b_e(z_2)c_i(z_2)]\\
+&[c_e(z_1)b_i(z_1),c_e(z_2)b_i(z_2)]\\
&=[b_i(z_1)c_i(z_1),b_e(z_2)]c_e(z_2)
+b_e(z_2)[b_i(z_1)c_i(z_1),c_e(z_2)]\\
&+[b_e(z_1)c_e(z_1),b_i(z_2)]c_i(z_2)
+b_i(z_2)[b_e(z_1)c_e(z_1),c_i(z_2)]\\
&+[b_i(z_1)c_i(z_1),b_e(z_2)]c_i(z_2)
-[b_i(z_1)c_i(z_1),c_e(z_2)]b_i(z_2)\\
&+b_e(z_2)[b_e(z_1)c_e(z_1),c_i(z_2)]
-c_e(z_2)[b_e(z_1)c_e(z_1),b_i(z_2)]\\
&-[b_i(z_2)c_i(z_2),b_e(z_1)]c_i(z_1)
+[b_i(z_2)c_i(z_2),c_e(z_1)]b_i(z_1)\\
&-b_e(z_1)[b_e(z_2)c_e(z_2),c_i(z_1)]
+c_e(z_1)[b_e(z_2)c_e(z_2),b_i(z_1)]\\
&+[b_e(z_1)c_i(z_1),b_e(z_2)]c_i(z_2)\\
&+b_e(z_2)[b_e(z_1)c_i(z_1),c_i(z_2)]\\
&+[c_e(z_1)b_i(z_1),c_e(z_2)]b_i(z_2)\\
&+c_e(z_2)[c_e(z_1)b_i(z_1),b_i(z_2)]\\
\end{split}
\end{equation*}
We use the following fact: if the operator $A$ anticommutes with each of $B,\ C$, then
\begin{itemize}
\item $[AB,C]=A[B,C]_+$, and
\item $[BA,C]=-[B,C]_+A$
\end{itemize}
which yields
\begin{equation*}
\begin{split}
[\sfb(z_1),\sfb(z_2)]
&=b_i(z_1)[c_i(z_1),b_e(z_2)]_+c_e(z_2)
-b_e(z_2)[b_i(z_1),c_e(z_2)]_+c_i(z_1)\\
&+b_e(z_1)[c_e(z_1),b_i(z_2)]_+c_i(z_2)
-b_i(z_2)[b_e(z_1),c_i(z_2)]_+c_e(z_1)\\
&+b_i(z_1)[c_i(z_1),b_e(z_2)]_+c_i(z_2)
+[b_i(z_1),c_e(z_2)]_+c_i(z_1)b_i(z_2)\\
&-b_e(z_2)[b_e(z_1),c_i(z_2)]_+c_e(z_1)
-c_e(z_2)b_e(z_1)[c_e(z_1),b_i(z_2)]_+\\
&-b_i(z_2)[c_i(z_2),b_e(z_1)]_+c_i(z_1)
-[b_i(z_2),c_e(z_1)]_+c_i(z_2)b_i(z_1)\\
&+b_e(z_1)[b_e(z_2),c_i(z_1)]_+c_e(z_2)
+c_e(z_1)b_e(z_2)[c_e(z_2),b_i(z_1)]_+\\
&+b_e(z_1)[c_i(z_1),b_e(z_2)]_+c_i(z_2)\\
&-b_e(z_2)[b_e(z_1),c_i(z_2)]_+c_i(z_1)\\
&+c_e(z_1)[b_i(z_1),c_e(z_2)]_+b_i(z_2)\\
&-c_e(z_2)[c_e(z_1),b_i(z_2)]_+b_i(z_1)\\
&=-\sSz(z_1,z_2) b_i(z_1)c_e(z_2)
-\sSz(z_2,z_1) b_e(z_2)c_i(z_1)\\
&+\sSz(z_1,z_2) b_e(z_1)c_i(z_2)
+\sSz(z_2,z_1) b_i(z_2)c_e(z_1)\\
&-\sSz(z_1,z_2) b_i(z_1)c_i(z_2)
+\sSz(z_2,z_1)  c_i(z_1)b_i(z_2)\\
&+\sSz(z_2,z_1) b_e(z_2)c_e(z_1)
-\sSz(z_1,z_2) c_e(z_2)b_e(z_1) \\
&+\sSz(z_2,z_1)b_i(z_2)c_i(z_1)
-\sSz(z_1,z_2)  c_i(z_2)b_i(z_1)\\
&-\sSz(z_1,z_2) b_e(z_1)c_e(z_2)
+\sSz(z_2,z_1) c_e(z_1)b_e(z_2)\\
&-\sSz(z_1,z_2) b_e(z_1)c_i(z_2)\\
&+\sSz(z_2,z_1) b_e(z_2)c_i(z_1)\\
&+\sSz(z_2,z_1) c_e(z_1)b_i(z_2)\\
&-\sSz(z_1,z_2) c_e(z_2)b_i(z_1)\\
&=0
\end{split}
\end{equation*}

Let us compute the two-point function of the field $\sfb$. For $z_1, \ne z_2$, consider
\begin{equation*}
\begin{split}
\sfb(z_1)\sfb(z_2)\mathbf{1} \otimes \mathbf{1}=&[b_i(z_1) \circ c_i(z_1) + b_e(z_1) \circ c_e(z_1)+b_e(z_1) \circ c_i(z_1)-c_e(z_1) \circ b_i(z_1)]\\ &[b_i(z_2) \circ c_i(z_2) + b_e(z_2) \circ c_e(z_2)+b_e(z_2) \circ c_i(z_2)-c_e(z_2) \circ b_i(z_2)] \mathbf{1} \otimes \mathbf{1}\\
=&[b_i(z_1) \circ c_i(z_1)]\circ [b_e(z_2) \circ c_e(z_2)] \mathbf{1} \otimes \mathbf{1} + \text{higher order terms}\\ 
=& c_i(z_1) \circ b_e(z_2) \circ b_i (z_1) \circ c_e(z_2) \mathbf{1} \otimes \mathbf{1} + \text{higher order terms}\\ 
=& -\sSz(z_1,z_2) \otimes \sSz(z_2,z_1) \mathbf{1} \otimes \mathbf{1} + \text{higher order terms}\\ 
\end{split}
\end{equation*}
which yields 
\begin{equation*}
\begin{split}
<\sfb(z_1)\sfb(z_2)>&=-\sSz(z_1,z_2) \otimes \sSz(z_2,z_1)\\  
&\in  K_{z_1} \otimes \xi^{-1}_{z_1} \otimes \xi_{z_2} \otimes K_{z_2} \otimes \xi^{-1}_{z_2} \otimes \xi_{z_1} \sim K_{z_1} \otimes  K_{z_2}
\end{split}
\end{equation*}

\section{\textbf{Arbitrary genus: the boson}}\label{section: Arbitrary genus: the boson}

We saw that the composite field $\sfb$ of the $b-c$ system is $K$-valued and has a two-point function with a double pole and residue 1.  Such a kernel can be used as the starting point for the definition of a boson on an arbitrary smooth projective curve $\bbX$.

Consider, on $\bbX \times \bbX$, the line bundle $K \boxtimes K (2\Delta)$. Its restriction to the diagonal $\Delta$ is canonically trivial. By a Bergman(-type) kernel, we will understand a section of $K \boxtimes K (2\Delta)$ which
\begin{itemize}
\item is symmetric with respect to the exchange of the two factors in the product $\bbX \times \bbX$, and 
\item has as residue on the diagonal the constant function 1.
\end{itemize}

\begin{proposition} The space of Bergman-type kernels is a affine space over $S^2H^0(\bbX,K)$
\end{proposition}

Fix a Bergman kernel $\bomega \in H^0(\bbX \times \bbX, K \boxtimes K (2\Delta))$.  We define
$$
\bW_\bbX= S^* \bK_\bbX
$$
where $\bK_\bbX$ is the space of meromorphic forms of the second kind on $\bbX$. Given $z \in \bbX$, we define
$$
\bW^z_\bbX= S^* \bK^z_\bbX
$$
where $\bK^z_\bbX$ is the space of meromorphic forms of the second kind which are regular at $z$. Define $\bomega_z \in H^0(\bbX,K(2z)) \otimes K_z$ by
$$
\bomega_z = \bomega|_{z \times \bbX}
$$
Next, define fields $e$ and $i$ as on $\bbP$, using $\bomega_z$  in place of $e_z$; let $b(z)=i(z)+e(z): \bW^z_\bbX \to \bW_\bbX \otimes K_z$. One checks that $b$ is local with respect to itself (the symmetry of $\bomega$ is used here). The correlation functions are given in terms of the two-point function by Wick's theorem, and the two-point function
is
\begin{equation*}
<b(z_1)b(z_2)>=\bomega(z_1,z_2)   
\end{equation*}

\pagebreak

\part{\textbf{{\Large Further writing.}}}

In sequels to this work I will address several issues.

\begin{enumerate}

\item \textit{Current algebras on $\bbP$.} Virasoro action, vertex algebra structure, hermitian pairing with integral central charge and reflection positivity, conformal blocks, construction of the KZ connection, unitarity.
\item \textit{Current algebras in higher genus.} Conformal blocks, construction of the ADW/Hitchin/KZB/TUY connection, unitarity.
\item \textit{A precise description of the parameter space of each theory.} We saw above that to construct the neutral fermion on  a curve, we need to choose a non-effective theta characteristic. The definition of the $b-c$ system involves a choice of line-bundle of degree $g-1$ ($g$ being the genus.) In the case of a boson, a Bergmann-type kernel is chosen. What happens if the theta characteristic chosen has non-trivial sections? Analogous questions can be asked regarding non-generic choices in other cases. These have been considered in the Physics literature.
\item Rational forms of vertex algebras associated to higher-dimensional lattices.

\suspend{enumerate}

In the works of Tsuchiya-Ueno-Yamada [T-U-Y], Beilinson-Drinfeld [B-D] and Frenkel-Ben-Zvi [F-BZ], quantum fields are constructed as vertex operators first (in terms of power series) and then globalised over curves. 

\resume{enumerate}

\item \textit{Projective structures.} How do the above choices relate to a choice of projective structure, which (as far as I can tell) is the choice made in globalising a vertex algebra? 
\item \textit{Vertex algebra bundles.} As described in this work, quantum fields are maps of (not necessarily quasi-coherent) sheaves of $\cO_\bbX$-modules, which are also $\cD_\bbX$-modules. We hope to study the relation with vertex algebra bundles, as described in \cite{F-BZ}. 
\end{enumerate}

\end{document}